\documentclass{amsart}

\usepackage[utf8]{inputenc}
\usepackage[T1]{fontenc}
\usepackage{lmodern}
\usepackage{amsmath,amssymb,amsthm,mathtools}
\usepackage{enumitem}
\usepackage{xcolor}
\usepackage[colorlinks=true,linkcolor=blue!55!black,citecolor=blue!55!black,urlcolor=blue!55!black]{hyperref}
\hypersetup{
  pdftitle={On Split Forms of Fusion Categories},
  pdfauthor={Cesar Galindo}
}

\newtheorem{theorem}{Theorem}[section]
\newtheorem{proposition}[theorem]{Proposition}
\newtheorem{corollary}[theorem]{Corollary}
\newtheorem{lemma}[theorem]{Lemma}
\theoremstyle{definition}
\newtheorem{definition}[theorem]{Definition}

\theoremstyle{remark}
\newtheorem{remark}[theorem]{Remark}
\numberwithin{equation}{section}

\newcommand{\C}{\mathbb C}
\newcommand{\R}{\mathbb R}
\newcommand{\Q}{\mathbb Q}
\newcommand{\F}{\mathbb F}
\newcommand{\Rep}{\operatorname{Rep}}
\newcommand{\Vect}{\operatorname{Vec}}
\newcommand{\Aut}{\operatorname{Aut}}
\newcommand{\Gal}{\operatorname{Gal}}
\newcommand{\Br}{\operatorname{Br}}
\newcommand{\id}{\operatorname{id}}

\title{On Split Forms of Fusion Categories}
\author{César Galindo}
\address{Departamento de Matemáticas, Universidad de los Andes, Bogotá, Colombia}
\date{July 2026}

\begin{document}

\begin{abstract}
We formulate split Galois descent for fusion categories through framed
equivalences, reducing the existence of a split form to a group-theoretic
splitting problem. The associated degree-two obstruction combines categorical
coherence with stable descent of the simple objects, while fusion spaces impose
index and parity restrictions. For quantum groups at roots of unity, we use the
quasi-$R$-matrix to make the bar involution tensor-compatible and prove that the
underlying fusion category $\mathcal C(\mathfrak g,k)$, associated with the simply
connected root datum, has a split form over the maximal totally real subfield of
the cyclotomic field generated by the quantum parameter. The same holds for all
its cyclotomic Galois conjugates; in particular, every
$\mathcal C(\mathfrak g,k)$ has a split real form. By contrast, no split real form
exists for the standard braiding except on $\mathcal C(E_8,1)$ and
$\mathcal C(D_{4m},1)$, $m\geq1$. We further show that associative
zesting can destroy real descent of $\mathcal C(\mathfrak g,k)$, while every  braided zesting of
$\mathcal C(\mathfrak g,k)$ still admits a split real form as an underlying
fusion category.
Applications to pointed and finite-group representation categories give criteria
for real descent in terms of cohomology and Frobenius--Schur indicators.
\end{abstract}

\maketitle

\section{Introduction}

Let $K\subset L$ and let $\mathcal C$ be a split fusion category over $L$. We study
whether $\mathcal C$ arises by scalar extension from a split fusion category over
$K$. An arbitrary $K$-form may combine several Galois-conjugate simples of
$\mathcal C$ into one simple object, and the endomorphism algebra of a descended
simple may be a nontrivial division algebra. Neither phenomenon occurs for a split
form. Every simple object descends individually with endomorphism algebra $K$,
compatibly with tensor product. Over $\R$, the nonsplit endomorphism algebras may be
$\C$ or the quaternion division algebra. Split real descent therefore gives an
intrinsic formulation of the real-form problem considered
in~\cite{BHP,BHPReality}.
Proposition~\ref{prop:dictionary} relates it to associativity coefficients, but our
arguments remain coordinate-free. The question is particularly natural for
fusion categories arising from quantum groups at roots of unity.

We use the categorical descent theory developed in
\cite[Sections~3.1--3.8]{EGDescent}, which separates descent of the category from
descent of its individual objects. For a finite Galois extension, one first seeks
twisted tensor equivalences whose isomorphism classes form a multiplicative
section. A degree-three class then obstructs coherent comparison isomorphisms.
Once a section preserving every simple class has been made coherent, each simple
object carries a relative Brauer class obstructing its stable descent
\cite[Sections~3.4 and~3.6]{EGDescent}. The coherent structures on a fixed section
form an $H^2$-torsor~\cite[Proposition~3.8(ii)]{EGDescent}. Consequently, changing
the coherent choice modifies all the Brauer classes simultaneously. Split descent
requires one coherent choice for which every simple object is stable with trivial
endomorphism division algebra.

For each twisted equivalence that fixes the simple isomorphism classes, we choose
isomorphisms from the images of the simple objects to the original objects. We
call the resulting data a \emph{framed equivalence}. Their equivalence classes
form a group that maps to the Galois group. A section of this map gives both a
coherent descent datum and compatible descent structures on the simple objects.
Proposition~\ref{prop:framed-split-descent} therefore reduces split descent to the
existence of such a section. Over $\C/\R$, this amounts to finding an element of
order two that maps to complex conjugation. The same construction applies to
equivalences preserving a braiding, a pivotal or spherical pivotal structure, or
a ribbon structure.

Fix a multiplicative section $\Phi$. Pulling back the framed extension defines a
class $\Omega_\Phi(\mathcal C)$ in degree two. Its image under the connecting map
is the inverse of the degree-three coherence obstruction of
\cite[Section~3.4]{EGDescent}. When the latter obstruction vanishes,
$\Omega_\Phi(\mathcal C)$ is represented by the relative Brauer classes of the
simple objects, modulo changes of the coherent monoidal structure. Thus its
vanishing is the remaining condition for split descent along $\Phi$
(Proposition~\ref{prop:stability-obstruction}). The same criterion applies when
$\Phi$ is required to preserve additional structure.

Fusion spaces give necessary conditions on $\Omega_\Phi(\mathcal C)$. An iterated
fusion space of dimension $n>0$ defines a central simple $K$-algebra of degree
$n$. Therefore the index of the corresponding Brauer class divides $n$
(Proposition~\ref{prop:fusion-space-index}). Taking determinants gives further
torsion relations among the Brauer classes. These relations define the
determinant obstruction group~\eqref{eq:determinant-obstruction-group}, which
depends only on the based Grothendieck ring and the extension $L/K$. It gives
necessary restrictions on the classes that can occur. Over $\C/\R$, odd fusion
multiplicities give parity conditions on the Brauer signs. In the unitary case, a
split real form gives orthonormal fusion bases with real orthogonal $F$-matrices.

Our main application of the framed criterion concerns the categories
$\mathcal C(\mathfrak g,k)$, for a complex simple Lie algebra $\mathfrak g$ and a
positive integral level $k$, defined using the simply connected root datum.
We use their realization as semisimplifications of tilting categories for
divided-power quantum groups at roots of unity
\cite{LusztigQuantumGroups,AndersenParadowski,SawinQuantum}. Complex conjugation
inverts the root-of-unity parameter, but the algebra bar is not a Hopf algebra
involution. The quasi-$R$-matrix supplies the tensorator that makes the bar
involution tensor-compatible
\cite[Sections~27.3.1--27.3.3 and~27.3.6]{LusztigQuantumGroups}. We show that this
tensor involution preserves the tilting category and its negligible ideal, while
the canonical bars on the simple highest-weight modules provide compatible
frames. The resulting framed element above complex conjugation is an involution.
In fact, the construction is defined over the cyclotomic field generated by the
quantum parameter and descends to its maximal totally real subfield. Consequently,
every $\mathcal C(\mathfrak g,k)$, as well as every cyclotomic Galois conjugate,
has a split form over that totally real field
(Theorem~\ref{thm:Cgk-split-cyclotomic} and
Corollary~\ref{cor:Cgk-galois-split}). In particular, every
$\mathcal C(\mathfrak g,k)$ has a split real form.
Together with unitarity, this gives orthonormal fusion bases with real orthogonal
$F$-matrices (Corollary~\ref{cor:Cgk-real-orthogonal-F}). The construction is
uniform in type and level and avoids computing associator matrices. It concerns
only the underlying fusion category and does not by itself provide descent of the
braiding.
A separate argument based on the unitary ribbon twist rules out a split real form
for the standard braiding except on $\mathcal C(E_8,1)$ and
$\mathcal C(D_{4m},1)$, $m\geq1$
(Theorem~\ref{thm:Cgk-split-braided-real}).

For every positive integral level, $\mathcal C(\mathfrak g,k)$ has a unique split
real form in types $A$, $B$, $C$, and $G_2$, up to equivalence compatible with
the chosen complexification. In even type~$D$, by contrast, we construct at least
two distinct simple-fixing split real forms, related by an explicit nontrivial
grading-cocycle twist. They are inequivalent under equivalences compatible
with the chosen complexifications.

Zesting is a procedure for modifying the tensor structure of a graded fusion
category using a $2$-cocycle with values in the invertible objects of its neutral
component, together with compatible coherence data
\cite[Sections~2--3]{BraidedZesting}. When this object-valued $2$-cocycle is
trivial, the tensor product remains unchanged and a scalar $3$-cocycle twists the
associator. This is the cohomological case. A braided zesting further modifies the
braiding compatibly. Applied along the universal grading, cohomological zesting
can destroy real descent. Indeed, for every $k\geq1$,
twisting $\mathcal C(\mathfrak{sl}_3,k)$ by a nonzero class in
$H^3(\mathbb Z/3\mathbb Z,\C^\times)$ produces a fusion category with no real
form. By contrast, for the braided zestings of
$\mathcal C(\mathfrak g,k)$ classified in~\cite{GalindoMoraRowellVerlinde}, the
hexagon equations force the nonreal coherence phases to cancel the half-braidings
of the relevant invertible objects. Their underlying fusion categories therefore
still have split real forms. As above, this conclusion concerns only the
underlying fusion categories.

Pointed categories provide an explicit cohomological application. For a finite
Galois extension $L/K$, the category $\Vect_{G,L}^\omega$ has a split $K$-form if
and only if $[\omega]$ lies in the image of $H^3(G,K^\times)$
(Proposition~\ref{prop:pointed-split-field}). Over $\C/\R$, this is equivalent to
$2[\omega]=0$ (Corollary~\ref{cor:pointed-split-real}). Thus the descent behavior
of a pointed category can be read from its associator class.

A second family of applications concerns $\Rep_{\C}(G)$ for a finite group $G$.
Suppose that an automorphism of $G$ sends every element to a conjugate of its
inverse and lifts coherently to a conjugate-linear tensor involution of
$\Rep_{\C}(G)$. The generalized twisted Frobenius--Schur indicators of
\cite{KawanakaMatsuyama} then determine the relative real or quaternionic type of
each simple representation. This gives a criterion for split symmetric real descent,
including the possible correction by central elements of order two. For groups of
odd order, split real forms exist precisely in the abelian case.

For irreducible Frobenius groups in odd characteristic, we use results on
bi-Galois algebras~\cite[Corollary~6.2]{DavydovGalois} and invariant
twists~\cite[Theorem~2.4]{DavydovTwistedGroup} to study arbitrary split real
forms. Under an orbit-rigidity hypothesis, a split real form exists if and only
if a split symmetric real form exists. This condition is determined by twisted
Frobenius--Schur indicators of the Frobenius complement
(Theorem~\ref{thm:orbit-rigid-frobenius-real}). A family with quaternionic
complement shows that a class-inverting automorphism is not sufficient. For
cyclic complements, Theorem~\ref{thm:frobenius} determines the smallest field of
definition and shows that every split form can be chosen symmetric over the same
field; this includes the examples studied in~\cite{BHP}.

This paper is organized as follows. Section~\ref{sec:descent} develops framed
descent and its cohomological obstruction, including restrictions from fusion
spaces and pointed categories. Section~\ref{sec:Cgk} studies split real forms of
$\mathcal C(\mathfrak g,k)$, braided descent, uniqueness, and zesting.
Section~\ref{sec:groups} treats representation categories of finite groups using
twisted Frobenius--Schur indicators. Sections~\ref{sec:frobenius-groups}
and~\ref{sec:cyclic-frobenius} study Frobenius groups and determine the minimal
split fields in the cyclic case.

\subsection*{Acknowledgements}

I am grateful to Pavel Etingof for several insightful comments on an earlier
version of this work. He suggested 
Remark~\ref{rem:affine-KZ-perspective}, raised the questions of uniqueness and
braided descent, and pointed out the cyclotomic refinement and its application
to Galois conjugates.
I am also grateful to Matthew Buican, Peter Huston, and Jiannis K. Pachos for
sharing a preliminary version of their work~\cite{BHPReality} and for helpful
discussions about the connections between the two papers.
I also thank Zhenghan Wang for helpful comments and for drawing my attention to
related literature on arithmetic aspects of modular categories.
This work was partially supported by Grant INV-2025-213-3452 from the
School of Science of Universidad de los Andes.
During the preparation of this paper, I collaborated with ChatGPT 5.5 and 5.6.

\enlargethispage{3\baselineskip}
\section{Split forms and framed Galois descent}\label{sec:descent}

We refer to~\cite{EGNO,EO} for the basic definitions and conventions concerning
fusion categories, tensor functors, and braidings. A fusion category $\mathcal D$ over
a field $k$ is a $k$-linear semisimple abelian category with finite-dimensional
morphism spaces and finitely many isomorphism classes of simple objects, equipped
with a rigid monoidal structure satisfying
$\operatorname{End}_{\mathcal D}(\mathbf 1)=k$. It is
\emph{split} if $\operatorname{End}_{\mathcal D}(X)=k$ for every simple object $X$.

\subsection{Split forms}

Throughout this section, let $K\subset L$ be a field extension and let $\mathcal C$
be a split fusion category over $L$. We use \emph{form}, \emph{split}, and scalar
extension, including idempotent completion, as in~\cite[Section~3.1]{EGDescent}.

\begin{definition}
A $K$-form of $\mathcal C$ is a fusion category $\mathcal C_K$ over $K$, together
with an $L$-linear tensor equivalence
\[
 \mathcal C_K\otimes_K L\simeq\mathcal C.
\]
Here $\mathcal C_K\otimes_K L$ denotes scalar extension from $K$ to $L$, including
idempotent completion. The form is \emph{split} if $\mathcal C_K$ is split.
\end{definition}

Write $\operatorname{Irr}(\mathcal C)$ for the set of simple isomorphism classes
and choose once and for all a representative of each class, denoted by the same
letter. For $X,Y,Z\in\operatorname{Irr}(\mathcal C)$, the
vector spaces $\operatorname{Hom}_{\mathcal C}(X\otimes Y,Z)$
are called \emph{fusion spaces}, and their dimensions
$N_{XY}^{Z}=\dim_L \operatorname{Hom}_{\mathcal C}(X\otimes Y,Z)$ are the fusion coefficients. The based
Grothendieck ring $K_0(\mathcal C)$ has distinguished basis
$\{[X]\mid X\in\operatorname{Irr}(\mathcal C)\}$ and multiplication
\[
 [X][Y]=\sum_{Z\in\operatorname{Irr}(\mathcal C)}N_{XY}^{Z}[Z].
\]
These products are the fusion rules. A \emph{fusion basis} is a choice of bases in
the fusion spaces, normalized at the tensor unit. In such bases the associator is
encoded by its
$6j$-symbols, assembled into $F$-matrices, and a braiding by its $R$-symbols,
assembled into $R$-matrices; see
\cite[Sections~2.1, 2.3, and~2.5.1]{Bonderson} and
\cite[Sections~4.1--4.2]{Wang}.

\begin{proposition}\label{prop:dictionary}
The following conditions are equivalent.
\begin{enumerate}[label=\textup{(\alph*)}]
\item There is a fusion basis in which all $6j$-symbols, equivalently all entries of the
$F$-matrices, lie in $K$.
\item The category $\mathcal C$ has a split $K$-form.
\end{enumerate}
\end{proposition}

\begin{proof}
If $\mathcal C_K$ is a split $K$-form, its simple objects and $K$-bases of their
fusion spaces extend to those of $\mathcal C$, and $K$-linearity of the associator
gives~\textup{(a)}.

Conversely, the $K$-spans of the chosen bases and the given $F$-matrices satisfy
over $K$ the pentagon and unit identities. The algebraic reconstruction in
\cite[Section~1]{Yamagami} gives a split semisimple $K$-linear monoidal category
$\mathcal C_K$ whose scalar extension is $\mathcal C$.
The local rigidity criterion of~\cite[Lemma~2.2]{Yamagami} also holds over $K$, since
its scalar identities and nonvanishing conditions may be checked after extension to
$L$. Thus $\mathcal C_K$ is a split $K$-form.
\end{proof}

If $\mathcal C$ is braided, the same argument, now including the hexagon identities,
shows that it has a split braided $K$-form precisely when its $F$- and $R$-matrices
are simultaneously $K$-valued in suitable fusion bases; structured forms are
treated intrinsically in Subsection~\ref{subsec:structured-forms}.

\subsection{Framed descent}

We now assume that $L/K$ is a finite Galois extension and write
$\Gamma=\Gal(L/K)$. For $\sigma\in\Gamma$, let ${}^{\sigma}\mathcal C$ denote the
category obtained by twisting scalar multiplication by $\sigma$. Recall that a
categorical group is a monoidal groupoid in which every object is tensor-invertible,
equivalently a rigid monoidal groupoid~\cite[Definition~2.11.4]{EGNO}; see
\cite[Sections~3--4]{GarzonInassaridze} for split extensions of categorical groups
and their obstruction theory. As in
\cite[Section~3.2]{EGDescent}, let $\underline{\Aut}_K(\mathcal C)$ be the
categorical group whose objects are pairs
\[
 (F,\sigma),
 \qquad F:{}^{\sigma}\mathcal C\longrightarrow\mathcal C,
 \qquad \sigma\in\Gamma,
\]
where $F$ is a twisted tensor equivalence. An arrow
$(F,\sigma)\to(F',\sigma)$ is a tensor natural isomorphism $F\Rightarrow F'$, and
there are no arrows between pairs with different Galois components. The monoidal
product is
\[
 (F,\sigma)\otimes(G,\tau)
 =\bigl(F\circ{}^\sigma G,\sigma\tau\bigr).
\]

Regard $\Gamma$ as the discrete categorical group $\underline{\Gamma}$.
Projection to the Galois component is the monoidal functor
\begin{equation}\label{eq:categorical-extension}
 \Psi:\underline{\Aut}_K(\mathcal C)\longrightarrow\underline{\Gamma},
 \qquad (F,\sigma)\longmapsto\sigma.
\end{equation}
Its kernel is the categorical group of $L$-linear tensor autoequivalences. We write
$\pi_0$ for the group of isomorphism classes of objects of a categorical group; in
particular, $\pi_0(\underline{\Gamma})=\Gamma$.
Thus~\eqref{eq:categorical-extension} induces a group homomorphism
\[
 \pi_0(\Psi):\pi_0\bigl(\underline{\Aut}_K(\mathcal C)\bigr)
 \longrightarrow\Gamma.
\]

Let $\Pi_{\mathrm{sf}}(\mathcal C/K)$ be the subgroup of
$\pi_0(\underline{\Aut}_K(\mathcal C))$ consisting of the classes
$[(F,\sigma)]$ such that
\[
 F:{}^\sigma\mathcal C\longrightarrow\mathcal C,
 \qquad F({}^\sigma X)\cong X
 \quad\text{for every simple object }X.
\]
Composition gives a homomorphism
\[
 p_{\mathrm{sf}}:\Pi_{\mathrm{sf}}(\mathcal C/K)\longrightarrow\Gamma.
\]
Here the subscript $\mathrm{sf}$ stands for simple-fixing.
A \emph{frame} on such a pair $(F,\sigma)$ is a family
\[
 \mathbf f=(f_X)_{X\in\operatorname{Irr}(\mathcal C)},
 \qquad
 f_X:F({}^\sigma X)\longrightarrow X.
\]
A \emph{framed twisted equivalence} is a triple $(F,\sigma,\mathbf f)$ consisting
of a twisted tensor equivalence together with a frame. Two framed twisted
equivalences $(F,\sigma,\mathbf f)$ and $(F',\sigma,\mathbf f')$ are equivalent if
there is a tensor natural isomorphism $\eta:F\Rightarrow F'$ such that
\[
 f'_X\circ\eta_{{}^\sigma X}=f_X
 \qquad(X\in\operatorname{Irr}(\mathcal C)).
\]
The composite of $(F,\sigma,\mathbf f)$ and $(G,\tau,\mathbf g)$ has underlying
functor $F\circ{}^\sigma G$ and frame
\[
 (\mathbf f*\mathbf g)_X
 =f_X\circ F({}^\sigma g_X).
\]
Composition preserves framed equivalence. The identity has underlying pair
$(\id_{\mathcal C},1)$ and the identity frame at every simple object.
We use the subscript $\mathrm{fr}$ for framed equivalences.

\begin{proposition}\label{prop:framed-group}
The equivalence classes of framed twisted equivalences form a group under this
composition. We denote it by $\Pi_{\mathrm{fr}}(\mathcal C/K)$.
\end{proposition}

\begin{proof}
Using the canonical identifications
${}^\sigma({}^\tau\mathcal C)={}^{\sigma\tau}\mathcal C$, associativity and the
identity follow from composition of twisted functors and their frames. It remains
to verify inverses, since a tensor equivalence is only invertible up to tensor
natural isomorphism.

Let $(F,\sigma,\mathbf f)$ be framed. Choose a tensor quasi-inverse
\[
 H:\mathcal C\longrightarrow{}^\sigma\mathcal C
\]
and tensor natural isomorphisms
\[
 u:\id_{{}^\sigma\mathcal C}\Longrightarrow HF,
 \qquad
 \epsilon:FH\Longrightarrow\id_{\mathcal C}
\]
satisfying the triangle identities. Put
\[
 G={} ^{\sigma^{-1}}H:
 {}^{\sigma^{-1}}\mathcal C\longrightarrow\mathcal C.
\]
For every simple $X$, define
\[
 h_X=u_{{}^\sigma X}^{-1}\circ H(f_X^{-1}):H(X)\longrightarrow{}^\sigma X,
 \qquad
 g_X={} ^{\sigma^{-1}}h_X:
 G({}^{\sigma^{-1}}X)\longrightarrow X.
\]
Thus $(G,\sigma^{-1},\mathbf g)$ is a framed twisted equivalence. The frame of
the product in the first order is
\[
 \begin{aligned}
 f_X\circ F({}^\sigma g_X)
 &=f_X\circ F(u_{{}^\sigma X}^{-1})\circ FH(f_X^{-1})\\
 &=f_X\circ\epsilon_{F({}^\sigma X)}\circ FH(f_X^{-1})
 =\epsilon_X,
 \end{aligned}
\]
where the second equality is a triangle identity and the last one is naturality
of $\epsilon$. Hence $\epsilon$ identifies this product with the identity framed
equivalence. The frame of the product in the other order is
\[
 \begin{aligned}
 g_X\circ G({}^{\sigma^{-1}}f_X)
 &={}^{\sigma^{-1}}\bigl(h_X\circ H(f_X)\bigr)\\
 &={}^{\sigma^{-1}}\bigl(u_{{}^\sigma X}^{-1}\bigr).
 \end{aligned}
\]
Therefore ${}^{\sigma^{-1}}u^{-1}$ identifies the other product with the
identity and also preserves the frames. The class of
$(G,\sigma^{-1},\mathbf g)$ is consequently a two-sided inverse. No strict
inverse of $F$ is required.
\end{proof}

Different choices of simple representatives yield isomorphic framed groups and
extensions. Evaluation on simple objects shows that
$\Aut(\id_{\mathcal C})$ is abelian.

\begin{proposition}\label{prop:framed-extension}
Forgetting the frames gives an exact sequence of groups
\begin{equation}\label{eq:framed-extension}
 1\longrightarrow
 \Aut_\otimes(\id_{\mathcal C})
 \longrightarrow\Aut(\id_{\mathcal C})
 \longrightarrow\Pi_{\mathrm{fr}}(\mathcal C/K)
 \longrightarrow\Pi_{\mathrm{sf}}(\mathcal C/K)
 \longrightarrow1.
\end{equation}
The map $\Aut(\id_{\mathcal C})\to\Pi_{\mathrm{fr}}(\mathcal C/K)$ sends $e$ to the class of
$(\id_{\mathcal C},1,(e_X)_X)$.
It is compatible with the projections to $\Gamma$.
\end{proposition}

\begin{proof}
Every class in $\Pi_{\mathrm{sf}}(\mathcal C/K)$ can be framed, which proves
surjectivity. An element $e\in\Aut(\id_{\mathcal C})$ maps to the identity framed
class precisely when $e$ is monoidal. Consider now a framed equivalence whose
unframed class is the identity.
A tensor natural isomorphism from its underlying functor to $\id_{\mathcal C}$
identifies the class with one represented by
$(\id_{\mathcal C},1,(e_X)_X)$ for some
$e\in\Aut(\id_{\mathcal C})$. This proves exactness.
\end{proof}

Let $p_{\mathrm{fr}}:\Pi_{\mathrm{fr}}(\mathcal C/K)\to\Gamma$ be the composite
of the forgetful map with $p_{\mathrm{sf}}$. The following proposition expresses
split descent directly in terms of this map.
Subsection~\ref{subsec:section-obstruction} relates this criterion to the usual
two-stage obstruction theory.

\begin{proposition}\label{prop:framed-split-descent}
A split $K$-form of $\mathcal C$ exists if and only if
$p_{\mathrm{fr}}$ admits a group-theoretic section.
\end{proposition}

\begin{proof}
A split $K$-form supplies the coherent descent datum and compatible descent
isomorphisms on all simple objects described in
\cite[Proposition~3.5 and its proof]{EGDescent}. The corresponding framed classes
define a homomorphism $s$ satisfying $p_{\mathrm{fr}}\circ s=\id_\Gamma$.

Conversely, let $s:\Gamma\to\Pi_{\mathrm{fr}}(\mathcal C/K)$ be a section of
$p_{\mathrm{fr}}$.
Choose normalized representatives
$(\Phi_\sigma,\sigma,(f_{\sigma,X})_X)$ of $s(\sigma)$, with the identity
representative over $1\in\Gamma$. The equality
$s(\sigma)s(\tau)=s(\sigma\tau)$ gives a unique
tensor natural isomorphism
\[
 J_{\sigma,\tau}:\Phi_\sigma\circ{}^\sigma\Phi_\tau
 \Longrightarrow\Phi_{\sigma\tau}
\]
compatible with every frame; uniqueness follows because a framed automorphism is
the identity on simple objects and hence on all objects. The two composites of the
$J_{\sigma,\tau}$ associated with a triple $\sigma,\tau,\rho$ are isomorphisms
between the same framed equivalences. They coincide by uniqueness, and hence the
$J_{\sigma,\tau}$ satisfy
the descent coherence; the same uniqueness gives
$J_{1,\sigma}=J_{\sigma,1}=\id$ under the normalization above. Compatibility with
the frames gives
\[
 f_{\sigma,X}\circ\Phi_\sigma({}^\sigma f_{\tau,X})
 =f_{\sigma\tau,X}\circ
 (J_{\sigma,\tau})_{{}^{\sigma\tau}X},
\]
which is the compatibility condition in
\cite[proof of Proposition~3.5]{EGDescent}. The reconstruction in that proof gives a $K$-form
whose simple objects are the descended objects indexed by
$\operatorname{Irr}(\mathcal C)$. Their endomorphism algebras are
$\operatorname{End}_{\mathcal C}(X)^\Gamma=L^\Gamma=K$, and hence the form is split.
\end{proof}

For $L/K=\C/\R$, let $\gamma$ be complex conjugation. If
$p_{\mathrm{fr}}^{-1}(\gamma)$ is empty, there is no split real form. Otherwise,
choose $t\in p_{\mathrm{fr}}^{-1}(\gamma)$. The fiber over $\gamma$ is
$\ker(p_{\mathrm{fr}})t$, and
\[
 (nt)^2=n(tnt^{-1})t^2
 \qquad(n\in\ker(p_{\mathrm{fr}})).
\]
Thus Proposition~\ref{prop:framed-split-descent} says that a split real form exists
if and only if the fiber $p_{\mathrm{fr}}^{-1}(\gamma)$ contains an involution.
Fixing a section of $p_{\mathrm{sf}}$ gives the abelian criterion below.

\subsection{The obstruction attached to a section}\label{subsec:section-obstruction}

Assume that $p_{\mathrm{sf}}$ admits a section
\[
 \Phi:\Gamma\longrightarrow\Pi_{\mathrm{sf}}(\mathcal C/K),
 \qquad p_{\mathrm{sf}}\circ\Phi=\id_\Gamma.
\]
This is stronger than surjectivity of $p_{\mathrm{sf}}$. It selects
multiplicatively compatible isomorphism classes, but no representative functors
or comparison isomorphisms.

Pulling back~\eqref{eq:framed-extension} along $\Phi$ and taking conjugation in
the resulting extension makes
\[
 Q_\Phi:=
 \left(
 \frac{\Aut(\id_{\mathcal C})}
 {\Aut_\otimes(\id_{\mathcal C})}
 \right)_\Phi
\]
a $\Gamma$-module and gives
\begin{equation}\label{eq:section-framed-extension}
 1\longrightarrow Q_\Phi
 \longrightarrow
 \Gamma\mathop{\times}_{\Pi_{\mathrm{sf}}(\mathcal C/K)}
 \Pi_{\mathrm{fr}}(\mathcal C/K)
 \longrightarrow\Gamma
 \longrightarrow1.
\end{equation}
The extension therefore defines
\begin{equation}\label{eq:global-stability-obstruction}
 \Omega_\Phi(\mathcal C)
 \in H^2(\Gamma,Q_\Phi).
\end{equation}
By construction, $\Omega_\Phi(\mathcal C)$ vanishes precisely when $\Phi$ lifts to
a section of $p_{\mathrm{fr}}$.

To compare this lifting class with categorical descent, choose representatives
$\Phi_\sigma$ such that $\Phi(\sigma)=[(\Phi_\sigma,\sigma)]$. Applying $\sigma$
to scalar components and transporting through $\Phi_\sigma$ defines compatible
$\Gamma$-actions on
$\Aut(\id_{\mathcal C})$ and
$\Aut_\otimes(\id_{\mathcal C})$; the subscript $\Phi$ will indicate these
actions. Their quotient is $Q_\Phi$. Since $\mathcal C$ is split over $L$,
evaluation on simple objects gives
\[
 \Aut(\id_{\mathcal C})
 \xrightarrow{\sim}(L^\times)^{\operatorname{Irr}(\mathcal C)}.
\]
Under this identification, $\Aut_\otimes(\id_{\mathcal C})$ consists of the
families satisfying
$e_{\mathbf 1}=1$ and $e_Z=e_Xe_Y$ whenever $N_{XY}^{Z}>0$.
Because $\Phi$ fixes every simple class, the actions are componentwise,
$(\sigma\cdot e)_X=\sigma(e_X)$. We use ordinary group cohomology throughout.

Associated with the fixed section is the coherence obstruction
\[
 o_\Phi(\mathcal C)\in
 H^3\!\left(\Gamma,\Aut_\otimes(\id_{\mathcal C})_\Phi\right)
\]
defined in~\cite[Section~3.4]{EGDescent}.

\begin{proposition}\label{prop:stability-obstruction}
Let $\Phi$ be a section of $p_{\mathrm{sf}}$. Then
\[
 \Omega_\Phi(\mathcal C)=0
\]
if and only if $\Phi$ can be represented by a coherent descent datum equipped with
compatible descent isomorphisms on every simple object. Such a datum defines a
split $K$-form.

Moreover, the connecting homomorphism associated with
\[
 1\longrightarrow\Aut_\otimes(\id_{\mathcal C})_\Phi
 \longrightarrow
 \Aut(\id_{\mathcal C})_\Phi
 \longrightarrow Q_\Phi\longrightarrow1
\]
satisfies
\begin{equation}\label{eq:omega-boundary}
 \partial\Omega_\Phi(\mathcal C)=o_\Phi(\mathcal C)^{-1}
 \in H^3\!\left(\Gamma,\Aut_\otimes(\id_{\mathcal C})_\Phi\right).
\end{equation}
Here the inverse is taken in the cohomology group; in additive notation, the
right-hand side is $-o_\Phi(\mathcal C)$.
\end{proposition}

\begin{proof}
A class in $H^2(\Gamma,Q_\Phi)$ vanishes precisely when the corresponding extension
splits. Proposition~\ref{prop:framed-split-descent} identifies such a splitting
with the datum in the first assertion.

To compute the boundary, choose
normalized tensor natural isomorphisms
\[
 J_{\sigma,\tau}:\Phi_\sigma\circ{}^\sigma\Phi_\tau
 \Longrightarrow\Phi_{\sigma\tau}
\]
and frames
$f_{\sigma,X}:\Phi_\sigma({}^\sigma X)\to X$. Define
$b_{\sigma,\tau}=(b_X(\sigma,\tau))_X
\in\Aut(\id_{\mathcal C})_\Phi$ by
\begin{equation}\label{eq:simultaneous-defect}
 f_{\sigma,X}\circ\Phi_\sigma({}^\sigma f_{\tau,X})
 =b_X(\sigma,\tau)f_{\sigma\tau,X}\circ
 (J_{\sigma,\tau})_{{}^{\sigma\tau}X}.
\end{equation}
Define
\(a_{\sigma,\tau,\rho}\in\Aut_\otimes(\id_{\mathcal C})_\Phi\) by
\[
 J_{\sigma\tau,\rho}\circ
 (J_{\sigma,\tau}\mathbin{\ast}\id)
 =
 a_{\sigma,\tau,\rho}\,
 J_{\sigma,\tau\rho}\circ
 (\id\mathbin{\ast}{}^\sigma J_{\tau,\rho}),
\]
as natural transformations from
\(\Phi_\sigma\circ{}^\sigma\Phi_\tau\circ
{}^{\sigma\tau}\Phi_\rho\) to \(\Phi_{\sigma\tau\rho}\), where $\ast$ denotes
whiskering of natural transformations. We use the
multiplicative group-cohomology coboundary
\[
 (\delta b)_{\sigma,\tau,\rho}
 =(\sigma\cdot b_{\tau,\rho})\,b_{\sigma,\tau\rho}\,
 b_{\sigma\tau,\rho}^{-1}b_{\sigma,\tau}^{-1}.
\]
Comparison of the two evaluations of the triple composite on every simple object
then gives
\begin{equation}\label{eq:defect-coboundary}
 \delta b=a,
\end{equation}
where $a$ is viewed as a cochain with values in the natural-automorphism group.
Thus the image
$\overline b\in C^2(\Gamma,Q_\Phi)$ is a cocycle. It is the factor set of
\eqref{eq:section-framed-extension}, since~\eqref{eq:simultaneous-defect} says exactly
that the product of the chosen framed lifts differs from the chosen lift of
$\sigma\tau$ by $\overline b_{\sigma,\tau}$. Consequently,
$[\overline b]=\Omega_\Phi(\mathcal C)$. Equation~\eqref{eq:defect-coboundary}
gives
\[
 \partial[\overline b]=[a].
\]
With the cocycle convention used in~\cite[Section~3.4]{EGDescent},
\[
 \omega_{\sigma,\tau,\rho}\,
 \bigl(J_{\sigma\tau,\rho}\circ
 (J_{\sigma,\tau}\mathbin{\ast}\id)\bigr)
 =
 J_{\sigma,\tau\rho}\circ
 (\id\mathbin{\ast}{}^\sigma J_{\tau,\rho}).
\]
Thus \(a_{\sigma,\tau,\rho}=\omega_{\sigma,\tau,\rho}^{-1}\), which proves
\eqref{eq:omega-boundary}. The inversion has no effect on any vanishing
statement.
\end{proof}

The relevant fragment of the long exact cohomology sequence associated with the
coefficient sequence in Proposition~\ref{prop:stability-obstruction} is
\begin{multline*}
 H^2\!\left(\Gamma,\Aut_\otimes(\id_{\mathcal C})_\Phi\right)
 \longrightarrow
 H^2\!\left(\Gamma,\Aut(\id_{\mathcal C})_\Phi\right)
 \longrightarrow\\
 H^2(\Gamma,Q_\Phi)
 \xrightarrow{\partial}
 H^3\!\left(\Gamma,\Aut_\otimes(\id_{\mathcal C})_\Phi\right).
\end{multline*}
Exactness shows that the middle map induces an injection
\begin{equation}\label{eq:cokernel-injection}
 \begin{aligned}
 \operatorname{coker}\Bigl(
 &H^2\!\left(\Gamma,\Aut_\otimes(\id_{\mathcal C})_\Phi\right)\\[-2pt]
 &\longrightarrow
 H^2\!\left(\Gamma,\Aut(\id_{\mathcal C})_\Phi\right)
 \Bigr)
 \lhook\joinrel\longrightarrow H^2(\Gamma,Q_\Phi).
 \end{aligned}
\end{equation}
whose image is $\ker(\partial)$.

Suppose now that $o_\Phi(\mathcal C)=0$. By~\eqref{eq:omega-boundary},
$\Omega_\Phi(\mathcal C)$ belongs to $\ker(\partial)$ and therefore has a unique
preimage under~\eqref{eq:cokernel-injection}. A coherent structure $J$ exists by
\cite[Section~3.4 and Proposition~3.5]{EGDescent}. Any choice of frames gives,
through~\eqref{eq:simultaneous-defect}, a class in
$H^2\!\left(\Gamma,\Aut(\id_{\mathcal C})_\Phi\right)$. Changing the frames changes
its representing cocycle by a coboundary, while coherent choices form a torsor
under
$H^2\!\left(\Gamma,\Aut_\otimes(\id_{\mathcal C})_\Phi\right)$
\cite[Proposition~3.8(ii)]{EGDescent}. Thus the resulting cokernel class is
independent of the frames and of the coherent choice $J$, and it is precisely the
preimage of $\Omega_\Phi(\mathcal C)$ under~\eqref{eq:cokernel-injection}.
Consequently,
\[
 \Omega_\Phi(\mathcal C)=0
 \quad\Longleftrightarrow\quad
 \begin{array}{l}
 o_\Phi(\mathcal C)=0\text{ and the displayed cokernel class vanishes}.
 \end{array}
\]

Consequently, $\mathcal C$ has a split $K$-form precisely when some section
$\Phi$ of $p_{\mathrm{sf}}$ satisfies $\Omega_\Phi(\mathcal C)=0$.

\subsection{Forms preserving additional structure}\label{subsec:structured-forms}

Let $\mathfrak s$ be a fixed braiding, a pivotal or spherical pivotal structure,
or a ribbon structure. In the ribbon case, both the braiding and the twist are
fixed. We use the
general formalism of~\cite[Sections~3.1--3.2]{EGDescent}; for braided and symmetric
forms see also~\cite[Section~3.9]{EGDescent}. Over $\sigma\in\Gamma$, restrict to
twisted tensor equivalences
\[
 F:({}^{\sigma}\mathcal C,{}^{\sigma}\mathfrak s)
 \longrightarrow(\mathcal C,\mathfrak s),
\]
and define $\Pi_{\mathrm{sf}}^{\mathfrak s}(\mathcal C/K)$ and
$\Pi_{\mathrm{fr}}^{\mathfrak s}(\mathcal C/K)$ using only these equivalences. When
$\mathfrak s$ is a symmetry, we use the superscript $\mathrm{sym}$. Then
\[
 1\longrightarrow
 \Aut_\otimes(\id_{\mathcal C})
 \longrightarrow\Aut(\id_{\mathcal C})
 \longrightarrow\Pi_{\mathrm{fr}}^{\mathfrak s}(\mathcal C/K)
 \longrightarrow\Pi_{\mathrm{sf}}^{\mathfrak s}(\mathcal C/K)
 \longrightarrow1.
\]
The kernel is unchanged. Monoidal natural isomorphisms are automatically
compatible with duality, while compatibility with the braiding, pivotal
structure, and twist follows from naturality. Sphericality can be checked after
the faithful scalar extension $K\subset L$.

Write $p_{\mathrm{sf}}^{\mathfrak s}$ and $p_{\mathrm{fr}}^{\mathfrak s}$ for the
projections to $\Gamma$. For a section
$\Phi$ of $p_{\mathrm{sf}}^{\mathfrak s}$, forgetting $\mathfrak s$ identifies its
pullback extension with the underlying tensor extension. Its class is
$\Omega_\Phi(\mathcal C)$, and the preceding argument gives
\[
 (\mathcal C,\mathfrak s)\text{ has a split $K$-form preserving }\mathfrak s
 \quad\Longleftrightarrow\quad
 \begin{array}{c}
 p_{\mathrm{sf}}^{\mathfrak s}\text{ admits a section }\Phi\text{ with}\rule{0pt}{2.2ex}\\[-2pt]
 \Omega_\Phi(\mathcal C)=0.
 \end{array}
\]
The extra requirement is the existence of a section through
$\Pi_{\mathrm{sf}}^{\mathfrak s}(\mathcal C/K)$. This concerns descent of the chosen
structure, rather than existence of some structure of the same type on a form.

\subsection{Fusion spaces and the determinant obstruction group}

Fusion spaces impose restrictions on $\Omega_\Phi(\mathcal C)$ even before a
coherent structure has been chosen. Let $\Phi$ be a section of $p_{\mathrm{sf}}$ and
let
$X_1,\ldots,X_r,Z$ be simple objects. Put
\[
 N_{X_1,\ldots,X_r}^{Z}
 =\dim_L\operatorname{Hom}(X_1\otimes\cdots\otimes X_r,Z).
\]
If this multiplicity is nonzero, monoidality gives
$a_Z=a_{X_1}\cdots a_{X_r}$ for every
$a\in\Aut_\otimes(\id_{\mathcal C})$. Hence the formula
\begin{equation}\label{eq:fusion-channel-map}
 \lambda_{X_1,\ldots,X_r}^{Z}:Q_\Phi\longrightarrow L^\times,
 \qquad
 [(u_T)_T]\longmapsto
 \frac{u_Z}{u_{X_1}\cdots u_{X_r}}
\end{equation}
defines a $\Gamma$-equivariant homomorphism. Here $\Br(K)$ consists of Brauer classes
of central simple $K$-algebras,
$\Br(L/K)=\ker(\Br(K)\to\Br(L))$ of those split by $L$, and
$\operatorname{ind}(\alpha)$ is the degree of the division-algebra representative
of $\alpha$; we use additive notation.

If the section carries a coherent structure $J$, the stable-object construction of
\cite[proof of Proposition~3.5 and Example~3.9]{EGDescent} gives
\[
 \beta_X(J)=[b_X]\in H^2(\Gamma,L^\times)
 \cong\Br(L/K).
\]
This class vanishes precisely when $X$ admits a compatible descent structure.

\begin{proposition}\label{prop:fusion-space-index}
Let $L/K$ be a finite Galois extension and let $\Phi$ be a section of
$p_{\mathrm{sf}}$. If
$n=N_{X_1,\ldots,X_r}^{Z}>0$, then
\begin{equation}\label{eq:fusion-space-index}
 \operatorname{ind}\!\left(
 (\lambda_{X_1,\ldots,X_r}^{Z})_*
 \Omega_\Phi(\mathcal C)
 \right)\mid n.
\end{equation}
If $\Phi$ admits a coherent structure $J$, then
\begin{equation}\label{eq:fusion-channel-brauer-class}
 (\lambda_{X_1,\ldots,X_r}^{Z})_*
 \Omega_\Phi(\mathcal C)
 =\beta_Z(J)-\sum_{i=1}^r\beta_{X_i}(J).
\end{equation}
In particular,
\[
 n\left(\beta_Z(J)-\sum_{i=1}^r\beta_{X_i}(J)\right)=0.
\]
\end{proposition}

\begin{proof}
Choose normalized representatives $\Phi_\sigma$, tensor isomorphisms
$J_{\sigma,\tau}$, and frames $f_{\sigma,T}$ as in the proof of
Proposition~\ref{prop:stability-obstruction}. Iterating the tensorators gives frames
on $X_1\otimes\cdots\otimes X_r$ and $\sigma$-semilinear maps on
\[
 V=\operatorname{Hom}(X_1\otimes\cdots\otimes X_r,Z).
\]
Let $b_T(\sigma,\tau)$ be the corresponding defects from
\eqref{eq:simultaneous-defect}. The factor set on $V$ is
\[
 d(\sigma,\tau)
 =\frac{b_Z(\sigma,\tau)}{
 b_{X_1}(\sigma,\tau)\cdots b_{X_r}(\sigma,\tau)}.
\]
Equation~\eqref{eq:defect-coboundary} and the fact that the scalar ratio defining
\eqref{eq:fusion-channel-map} is trivial on
$\Aut_\otimes(\id_{\mathcal C})$ show that $d$ is a cocycle. Since the image
of the family $(b_T)_T$ in $Q_\Phi$ represents $\Omega_\Phi(\mathcal C)$, the class of $d$ is
$(\lambda_{X_1,\ldots,X_r}^{Z})_*\Omega_\Phi(\mathcal C)$.

Write $T_\sigma$ for the resulting $\sigma$-semilinear operators on $V$. Our
factor-set convention is
\[
 T_\sigma T_\tau=d(\sigma,\tau)T_{\sigma\tau}.
\]
We identify $H^2(\Gamma,L^\times)$ with $\Br(L/K)$ by the convention that
$[d]$ is the Brauer class of the fixed algebra for the action
\[
 A\longmapsto T_\sigma A T_\sigma^{-1}
 \qquad(A\in\operatorname{End}_L(V)).
\]
This is an honest semilinear action on $\operatorname{End}_L(V)$, and its fixed
algebra is a central simple $K$-algebra of degree $\dim_L(V)=n$, split by $L$.
The index of a Brauer class divides the degree of every central simple algebra
representing it. This proves~\eqref{eq:fusion-space-index}.

If $J$ is coherent, each $b_T$ is a cocycle representing $\beta_T(J)$, which gives
\eqref{eq:fusion-channel-brauer-class}. The last assertion follows because the
exponent of a Brauer class divides its index.
\end{proof}

Taking determinants in the projective semilinear action on a fusion space
explains the following construction. If
$\operatorname{Hom}(X\otimes Y,Z)$ has dimension $n>0$ and factor set $d$, then
the determinant relation shows that $d^n$ is a coboundary. Thus taking
determinants forgets the sharper index bound of
Proposition~\ref{prop:fusion-space-index}, but retains the relation
\[
 n\bigl(\beta_Z(J)-\beta_X(J)-\beta_Y(J)\bigr)=0.
\]
The following quotient of the additive group of the based Grothendieck ring is
universal for precisely these determinant relations.
\begin{equation}\label{eq:determinant-group}
 D_{\det}(K_0(\mathcal C))=K_0(\mathcal C)\Big/
 \left\langle
 [\mathbf 1],\;
 N_{XY}^{Z}([Z]-[X]-[Y])
 \ \middle|\ N_{XY}^{Z}>0
 \right\rangle.
\end{equation}
For every abelian group $B$, written additively,
\begin{equation}\label{eq:determinant-group-universal}
 \begin{aligned}
 &\operatorname{Hom}(D_{\det}(K_0(\mathcal C)),B)\\
 &\qquad=\left\{b:\operatorname{Irr}(\mathcal C)\to B\ \middle|\
 \begin{array}{l}
 b_{\mathbf 1}=0,\\[2pt]
 N_{XY}^{Z}(b_Z-b_X-b_Y)=0
 \text{ if }N_{XY}^{Z}>0
 \end{array}
 \right\}.
 \end{aligned}
\end{equation}
In particular, Proposition~\ref{prop:fusion-space-index} shows that, whenever a
section $\Phi$ of $p_{\mathrm{sf}}$ is endowed with a coherent structure $J$,
$[X]\mapsto\beta_X(J)$ factors through $D_{\det}(K_0(\mathcal C))$.

The coefficient modules are themselves determined by the based ring and the field
extension. The universal grading group $U(\mathcal C)$ and the degrees of the simple
objects are determined by
$K_0(\mathcal C)$
\cite[Sections~3.6 and~4.14]{EGNO}. Schur's lemma gives the first identification
below, while the proof of~\cite[Proposition~4.14.3]{EGNO}, which applies verbatim
to a split category over $L$, gives the second. With componentwise $\Gamma$-action,
\[
 \begin{aligned}
 \Aut(\id_{\mathcal C})_\Phi
 &\cong\operatorname{Map}(\operatorname{Irr}(\mathcal C),L^\times),\\
 \Aut_\otimes(\id_{\mathcal C})_\Phi
 &\cong\operatorname{Hom}(U(\mathcal C),L^\times)
 \subseteq\operatorname{Map}(\operatorname{Irr}(\mathcal C),L^\times).
 \end{aligned}
\]
The inclusion evaluates a character on the degree of each simple object, and the
quotient is $Q_\Phi$. Since
$H^2(\Gamma,L^\times)\cong\Br(L/K)$, componentwise cohomology gives
\[
 H^2\!\left(\Gamma,
 \operatorname{Map}(\operatorname{Irr}(\mathcal C),L^\times)\right)
 \cong\Br(L/K)^{\operatorname{Irr}(\mathcal C)},
\]
and
\[
 \operatorname{im}\bigl(
 H^2(\Gamma,\operatorname{Hom}(U(\mathcal C),L^\times))
 \to\Br(L/K)^{\operatorname{Irr}(\mathcal C)}\bigr)
 \subseteq
 \operatorname{Hom}(D_{\det}(K_0(\mathcal C)),\Br(L/K)).
\]
We call the quotient
\begin{equation}\label{eq:determinant-obstruction-group}
 \mathcal O_{K_0(\mathcal C)}^{\det}(L/K)
 =\frac{\operatorname{Hom}(D_{\det}(K_0(\mathcal C)),\Br(L/K))}
 {\operatorname{im}\bigl(
 H^2(\Gamma,\operatorname{Hom}(U(\mathcal C),L^\times))
 \to\Br(L/K)^{\operatorname{Irr}(\mathcal C)}\bigr)}
\end{equation}
the \emph{determinant obstruction group} of $K_0(\mathcal C)$ over $L/K$. Its numerator
consists of the Brauer vectors satisfying the determinant relations from the
fusion spaces. Its denominator identifies vectors that differ by a change of
coherent monoidal structure. Therefore it depends only on the based Grothendieck
ring and the extension $L/K$.
It injects canonically into
\[
 \mathcal O_{K_0(\mathcal C)}^{\det}(L/K)
 \lhook\joinrel\longrightarrow
 \operatorname{coker}\bigl(
 H^2(\Gamma,\operatorname{Hom}(U(\mathcal C),L^\times))
 \to H^2(\Gamma,\operatorname{Map}(\operatorname{Irr}(\mathcal C),L^\times))\bigr).
\]
The long exact sequence gives a further injection
\begin{multline*}
 \operatorname{coker}\Bigl(
 H^2(\Gamma,\operatorname{Hom}(U(\mathcal C),L^\times))
 \longrightarrow{}\\[-2pt]
 H^2(\Gamma,\operatorname{Map}(\operatorname{Irr}(\mathcal C),L^\times))\Bigr)
 \lhook\joinrel\longrightarrow
 H^2\!\left(\Gamma,
 \frac{\operatorname{Map}(\operatorname{Irr}(\mathcal C),L^\times)}
 {\operatorname{Hom}(U(\mathcal C),L^\times)}\right).
\end{multline*}
For a section $\Phi$ with coherent structure $J$, the class of its Brauer vector in
\eqref{eq:determinant-obstruction-group} maps to $\Omega_\Phi(\mathcal C)$. However,
not every class in the determinant obstruction group must arise from a coherent
descent datum.

\subsection{Real forms and odd support}\label{subsec:real-case}

Assume that $L/K=\C/\R$ and fix a section $\Phi$ of $p_{\mathrm{sf}}$ with coherent
structure $J$. Write $C_2=\Gal(\C/\R)=\{1,\gamma\}$. Group
cohomology and~\cite[Proposition~4.14.3]{EGNO} give
\[
 H^2\!\left(C_2,\Aut(\id_{\mathcal C})_\Phi\right)
 \cong\{\pm1\}^{\operatorname{Irr}(\mathcal C)},
 \qquad
 H^2\!\left(C_2,\Aut_\otimes(\id_{\mathcal C})_\Phi\right)
 \cong\operatorname{Hom}(U(\mathcal C),\{\pm1\}).
\]
The first group records the relative Brauer signs of the simple objects. In the second,
conjugation acts by inversion on the characters of the universal grading group, and
all norms are trivial. Write
$\varepsilon_X(J)$ for the sign of $\beta_X(J)$ and let
$\varepsilon(J)=(\varepsilon_X(J))_X
\in\{\pm1\}^{\operatorname{Irr}(\mathcal C)}$.

\begin{corollary}\label{cor:real-signs}
Let $\Phi$ be a section of $p_{\mathrm{sf}}$ with coherent structure $J$. Then
\[
 \frac{\{\pm1\}^{\operatorname{Irr}(\mathcal C)}}
 {\operatorname{Hom}(U(\mathcal C),\{\pm1\})}
 \lhook\joinrel\longrightarrow H^2(C_2,Q_\Phi),
 \qquad
 [\varepsilon(J)]\longmapsto\Omega_\Phi(\mathcal C).
\]
In particular,
$\Omega_\Phi(\mathcal C)=0$ if and only if
$\varepsilon(J)\in\operatorname{Hom}(U(\mathcal C),\{\pm1\})$,
equivalently if and only if the coherent datum can be modified to define a split
real form.
\end{corollary}

\begin{proof}
Since $J$ is coherent, Proposition~\ref{prop:stability-obstruction} identifies
$\Omega_\Phi(\mathcal C)$ with the image of the simultaneous Brauer vector
$(\beta_X(J))_{X\in\operatorname{Irr}(\mathcal C)}$.
Under the identifications above, the map
$H^2(C_2,\Aut_\otimes(\id_{\mathcal C})_\Phi)\to
H^2\!\left(C_2,\Aut(\id_{\mathcal C})_\Phi\right)$
is the inclusion of the monoidal sign vectors into all sign vectors. Its cokernel
is therefore the displayed quotient, and the long
exact sequence gives the displayed injection and identification.
\end{proof}

The order-two monoidal automorphisms can be read from the fusion rules. A natural
automorphism with components $\varepsilon_X\id_X$, $\varepsilon_X\in\{\pm1\}$,
is monoidal precisely when its signs satisfy
\[
 \operatorname{Hom}(U(\mathcal C),\{\pm1\})
 =\left\{\varepsilon\in\{\pm1\}^{\operatorname{Irr}(\mathcal C)}\ \middle|\
 \begin{array}{l}
  \varepsilon_{\mathbf 1}=1,\\[2pt]
  \varepsilon_Z=\varepsilon_X\varepsilon_Y
  \text{ whenever }N_{XY}^{Z}>0
 \end{array}\right\}.
\]
Define
\[
 \mathcal S_{\mathrm{odd}}(\mathcal C)
 =\left\{\varepsilon\in\{\pm1\}^{\operatorname{Irr}(\mathcal C)}\ \middle|\
 \begin{array}{l}
  \varepsilon_{\mathbf 1}=1,\\[2pt]
  \varepsilon_Z=\varepsilon_X\varepsilon_Y
  \text{ whenever }N_{XY}^{Z}\text{ is odd}
 \end{array}\right\}.
\]
The universal property~\eqref{eq:determinant-group-universal} and
\eqref{eq:determinant-obstruction-group} now give
\begin{equation}\label{eq:real-determinant-obstruction-group}
 \operatorname{Hom}(D_{\det}(K_0(\mathcal C)),\{\pm1\})
 \cong\mathcal S_{\mathrm{odd}}(\mathcal C),
 \qquad
 \mathcal O_{K_0(\mathcal C)}^{\det}(\C/\R)
 \cong
 \frac{\mathcal S_{\mathrm{odd}}(\mathcal C)}
 {\operatorname{Hom}(U(\mathcal C),\{\pm1\})}.
\end{equation}
\begin{proposition}\label{prop:odd-support}
For every section $\Phi$ of $p_{\mathrm{sf}}$ with coherent structure $J$, one has
\[
 \varepsilon(J)\in\mathcal S_{\mathrm{odd}}(\mathcal C).
\]
Consequently, if
$\mathcal S_{\mathrm{odd}}(\mathcal C)
=\operatorname{Hom}(U(\mathcal C),\{\pm1\})$, then
$\Omega_\Phi(\mathcal C)=0$ and the datum can be modified to give a split real
form.
\end{proposition}

\begin{proof}
The Brauer vector factors through $D_{\det}(K_0(\mathcal C))$, and under
\eqref{eq:real-determinant-obstruction-group} its homomorphism is the sign vector
$\varepsilon(J)$. Hence
$\varepsilon(J)\in\mathcal S_{\mathrm{odd}}(\mathcal C)$. If the latter group is
$\operatorname{Hom}(U(\mathcal C),\{\pm1\})$,
Corollary~\ref{cor:real-signs} gives the conclusion.
\end{proof}

If
$\varepsilon_X=(-1)^{x_X}$ with $x_X\in\mathbb F_2$, then
$\mathcal S_{\mathrm{odd}}(\mathcal C)$ is the solution space of the linear equations
\[
 x_{\mathbf 1}=0,
 \qquad
 x_X+x_Y+x_Z=0
 \quad\text{whenever }N_{XY}^{Z}\text{ is odd}.
\]

\begin{corollary}\label{cor:odd-multiplicities}
Let $\Phi$ be a section of $p_{\mathrm{sf}}$ with coherent structure $J$. If every
nonzero fusion coefficient $N_{XY}^{Z}$ is odd, then
$\Omega_\Phi(\mathcal C)=0$. Thus the given datum can be modified to define a split
real form.
\end{corollary}

\begin{proof}
Under the hypothesis, the odd-support relations are exactly the monoidal sign
relations; hence the two groups coincide. Apply
Proposition~\ref{prop:odd-support}.
\end{proof}

Say that a family $V_1,\ldots,V_s\in\operatorname{Irr}(\mathcal C)$
tensor-generates $\mathcal C$ if every simple object is a summand of a tensor word
in the $V_i$. For a simple object $V$, consider the
directed graph on
$\operatorname{Irr}(\mathcal C)$ with an arrow $X\to Y$, for $X\ne Y$, whenever
$N_{VX}^{Y}$ is odd.

\begin{proposition}\label{prop:odd-support-tests}
Let $\mathcal C$ be a fusion category.
\begin{enumerate}[label=\textup{(\alph*)}]
\item Suppose that $V_1,\ldots,V_s$ tensor-generate $\mathcal C$ and every nonzero
$N_{V_iX}^{Y}$ is odd. Then
$\mathcal S_{\mathrm{odd}}(\mathcal C)
=\operatorname{Hom}(U(\mathcal C),\{\pm1\})$.

\item Let $V$ be simple. If the underlying undirected graph on
$\operatorname{Irr}(\mathcal C)$ associated above with $V$ is connected, evaluation at
$V$ is injective on $\mathcal S_{\mathrm{odd}}(\mathcal C)$, and hence
$|\mathcal S_{\mathrm{odd}}(\mathcal C)|\leq2$. If, in addition,
$N_{VX}^{X}$ is odd for some simple $X$, then
$\mathcal S_{\mathrm{odd}}(\mathcal C)
=\operatorname{Hom}(U(\mathcal C),\{\pm1\})=1$.
\end{enumerate}
\end{proposition}

\begin{proof}
For~\textup{(a)}, let $\varepsilon\in\mathcal S_{\mathrm{odd}}(\mathcal C)$.
If a simple object $X$ is a summand of the right-associated tensor word
$V_{i_1}\otimes\cdots\otimes V_{i_t}$, semisimplicity gives simple objects
$Y_0=X,\ldots,Y_t=\mathbf 1$ such that
\[
 N_{V_{i_j}Y_j}^{Y_{j-1}}>0
 \qquad(1\leq j\leq t).
\]
These coefficients are odd by hypothesis, and hence
\[
 \varepsilon_X=\prod_{j=1}^t\varepsilon_{V_{i_j}}.
\]
If $N_{XY}^{Z}>0$, choose tensor words containing $X$ and $Y$. Their concatenation
contains $X\otimes Y$, and hence $Z$. The preceding formula therefore gives
$\varepsilon_Z=\varepsilon_X\varepsilon_Y$. Thus
$\varepsilon\in\operatorname{Hom}(U(\mathcal C),\{\pm1\})$; the reverse inclusion
is immediate.

For~\textup{(b)}, the relation along an edge is
$\varepsilon_Y=\varepsilon_V\varepsilon_X$. Since $\varepsilon_V^2=1$, this
relation propagates in either direction along an edge. Since
$\varepsilon_{\mathbf 1}=1$, connectedness determines every sign from
$\varepsilon_V$. An odd diagonal coefficient gives
$\varepsilon_X=\varepsilon_V\varepsilon_X$, and hence $\varepsilon_V=1$.
Connectedness then makes every sign positive.
\end{proof}

The odd-support criterion applies more generally than
Corollary~\ref{cor:odd-multiplicities}. Label the nontrivial irreducible
representations of the alternating group $A_5$ by their dimensions as
$V_3,V_{3'},V_4,V_5$. Then
\[
 \begin{aligned}
 V_3\otimes V_3&\cong\mathbf 1\oplus V_3\oplus V_5,&
 V_{3'}\otimes V_{3'}&\cong\mathbf 1\oplus V_{3'}\oplus V_5,\\
 V_3\otimes V_{3'}&\cong V_4\oplus V_5.
 \end{aligned}
\]
These rules force every element of
$\mathcal S_{\mathrm{odd}}(\Rep_{\C}(A_5))$ to be trivial, whereas
$N_{V_5V_5}^{V_4}=2$. Hence
Proposition~\ref{prop:odd-support} applies to the coefficientwise-conjugation datum
although the hypothesis of
Corollary~\ref{cor:odd-multiplicities} does not.

\subsection{Unitary consequences}\label{subsec:unitary-consequences}

By a \emph{unitary fusion category} we mean a fusion category with a positive dagger
compatible with tensor products and with unitary monoidal constraints; positivity
means that every endomorphism algebra is a finite-dimensional $C^*$-algebra
\cite[Sections~1.1--1.2]{Reutter}.

Let $\gamma$ denote complex conjugation. If a unitary fusion category has a split
real form, view its conjugate-linear descent
functor as a linear monoidal equivalence
$\Phi:{}^{\gamma}\mathcal C\to\mathcal C$. This equivalence and its coherence
$\mu:\Phi^2\Rightarrow\id_{\mathcal C}$ may be chosen unitary by
\cite[Theorems~1--2]{Reutter}. For a frame
$f_X:\Phi(X)\to X$ satisfying the descent identity, one has
$f_X\Phi(f_X)=\mu_X$. Writing
$f_X^\dagger f_X=\lambda_X\id_{\Phi(X)}$ with $\lambda_X>0$ gives
\[
 \id_{\Phi^2(X)}=\mu_X^\dagger\mu_X
 =\Phi(f_X)^\dagger f_X^\dagger f_X\Phi(f_X)
 =\lambda_X\Phi(f_X^\dagger f_X)
 =\lambda_X^2\id_{\Phi^2(X)}.
\]
The last equality uses $\Phi(\lambda_X)=\lambda_X$, since
$\lambda_X\in\R_{>0}$.
Thus every $f_X$ is unitary. The induced real structures on the fusion spaces are
antiunitary and preserve the associator; orthonormal bases of their fixed subspaces
give real orthogonal $F$-matrices. For the duality datum below, the sign criterion
recovers the Frobenius--Schur mechanism of~\cite{FRS} and is related to the
flat charge-conjugation criterion of~\cite[Sections~3--5]{BHPReality}.

\begin{proposition}
\label{prop:unitary-braided-real-twist}
Let $(\mathcal C,c)$ be a braided fusion category admitting a unitary structure,
and let $\theta$ be its unique unitary ribbon structure. Suppose that
$(\mathcal C,c)$ has a braided real form, and write
\[
 \Phi:({}^{\gamma}\mathcal C,{}^{\gamma}c)
 \longrightarrow(\mathcal C,c)
\]
for the braided equivalence over complex conjugation. If the induced permutation
of the simple objects is denoted by $\pi$, so that
$\Phi({}^{\gamma}X)\cong\pi(X)$, then
\begin{equation}\label{eq:braided-real-twist-conjugation}
 \theta_{\pi(X)}=\overline{\theta_X}
 \qquad\bigl(X\in\operatorname{Irr}(\mathcal C)\bigr).
\end{equation}
Moreover, the descent datum may be chosen to preserve $\theta$. Consequently,
a split braided real form is automatically a split ribbon real form, and its
existence forces
\[
 \theta_X\in\{\pm1\}
 \qquad\bigl(X\in\operatorname{Irr}(\mathcal C)\bigr).
\]
\end{proposition}

\begin{proof}
Equip ${}^{\gamma}\mathcal C$ with the conjugate unitary structure. By
\cite[Theorem~5]{Reutter}, the braided equivalence $\Phi$ is braided-monoidally
isomorphic to a unitary braided equivalence. Transporting the conjugate of
$\theta$ through this unitary equivalence gives a unitary ribbon structure on
$(\mathcal C,c)$. By the uniqueness theorem
\cite[Theorem~3.5]{GalindoUnitary}, this transported structure is $\theta$.
Replacing the descent datum within its braided monoidal isomorphism class and
transporting its coherence therefore makes it ribbon preserving.

On the simple object ${}^{\gamma}X$, the conjugate twist acts by
$\overline{\theta_X}$. Ribbon preservation now gives
\eqref{eq:braided-real-twist-conjugation}. In the split case $\pi$ is the
identity, so every $\theta_X$ is real. Since the unitary twist has absolute value
one, it follows that $\theta_X\in\{\pm1\}$.
\end{proof}

\begin{corollary}\label{cor:selfdual-odd}
Let $\mathcal C$ be a braided fusion category admitting a unitary structure. Suppose
every simple object is self-dual and
$\mathcal S_{\mathrm{odd}}(\mathcal C)
=\operatorname{Hom}(U(\mathcal C),\{\pm1\})$. Then
$\mathcal C$ has a split real form. Moreover, after choosing a unitary structure, it
admits orthonormal fusion bases in which all $F$-matrices are real.

In particular, the conclusion holds if every nonzero fusion multiplicity is odd.
\end{corollary}

\begin{proof}
Choose a unitary structure and a unitary dual functor, with anti-monoidal tensorator
$\nu_{X,Y}:Y^\vee\otimes X^\vee\to(X\otimes Y)^\vee$ and unitary pivotal
structure $j:\id_{\mathcal C}\Rightarrow(-)^{\vee\vee}$, as
in~\cite[Proposition~3.9, Corollaries~3.10 and~3.32, and Section~3.5]{Penneys}.
Put $\Phi(X)=X^\vee$ and $\Phi(f)=(f^\dagger)^\vee$.
By~\cite[Theorem~3.2]{GalindoUnitary}, the braiding $c$ is unitary. Define
\[
 \Phi_2(X,Y)=\nu_{X,Y}\circ c_{X^\vee,Y^\vee}
\]
and use the unit constraint of the dual functor. The anti-monoidal coherence of $\nu$
and the hexagon show that this makes $\Phi$ a conjugate-linear tensor
autoequivalence. Moreover,~\cite[Exercise~8.9.2]{EGNO} and unitarity give
\[
 \nu_{Y,X}^{-1}\circ\Phi(c_{X,Y})\circ\nu_{X,Y}
 =c_{X^\vee,Y^\vee}^{-1}.
\]
The displayed identity cancels the two braidings in the tensorator of $\Phi^2$,
leaving the canonical double-dual tensorator. Hence
$\mu=j^{-1}:\Phi^2\Rightarrow\id_{\mathcal C}$ is monoidal, and the pivotal identity
$j_{X^\vee}=(j_X^{-1})^\vee$, together with unitarity, gives
$\Phi(\mu_X)=\mu_{\Phi(X)}$. Thus $(\Phi,\mu)$ is a coherent descent datum over
$\C/\R$. Self-duality makes it a section of $p_{\mathrm{sf}}$, and
Proposition~\ref{prop:odd-support} gives a split real form. The unitary refinement
above gives the final assertion.
\end{proof}

An independent criterion in~\cite[Corollary~1]{BHPReality} produces real
$F$-symbols, together with symmetric $R$-matrices, when the
Frobenius--Schur indicators of a self-dual unitary ribbon category define a
grading. The criterion above is instead formulated in terms of odd fusion
multiplicities and produces a split real form of the underlying fusion category.

Examples include unitary $N$-metaplectic categories when $N$ is odd or divisible by
$4$~\cite[Sections~3.1 and~3.3]{GRR}. The self-dual cases among the modular
equivariantizations of~\cite{GLM} also apply
\cite[Section~2 and Propositions~5.8, 5.15, and~5.16]{GLM}. Their construction by
extension and equivariantization makes them weakly group-theoretical
\cite[Proposition~4.1]{ENOWGT}, hence unitarizable
\cite[Theorem~5.20]{GHR}. Deligne products give further examples.

\subsection{Split forms of pointed categories}

A fusion category is \emph{pointed} if every simple object is
tensor-invertible. Its simple isomorphism classes form a finite group $G$ under
tensor product, and its based Grothendieck ring is $\mathbb ZG$. In this case the
determinant data has a simple description. The definition of
$\mathcal S_{\mathrm{odd}}$ applies to any based ring, and for $\mathbb ZG$ the
relations in~\eqref{eq:determinant-group} give
\begin{equation}\label{eq:pointed-determinant-group}
 D_{\det}(\mathbb ZG)\cong G_{\mathrm{ab}},
 \qquad
 \mathcal S_{\mathrm{odd}}(\mathbb ZG)
 =\operatorname{Hom}(G,\{\pm1\}).
\end{equation}
Here $G_{\mathrm{ab}}$ denotes the abelianization of $G$.
Consequently,
$\mathcal O_{\mathbb ZG}^{\det}(\C/\R)=0$. Split descent is then determined by
the associator class, as the following criterion shows.

For a finite group $G$, a field $L$, and a normalized cocycle
$\omega\in Z^3(G,L^\times)$, let $\Vect_{G,L}^{\omega}$ be the category of
finite-dimensional $G$-graded $L$-vector spaces with associator determined by
$\omega$; over $\C$ write $\Vect_G^\omega$. Its tensor-equivalence class depends
only on $[\omega]\in H^3(G,L^\times)$. We use the classification of pointed
categories and
their tensor equivalences in~\cite[Section~2.6]{EGNO}; see also
\cite[Proposition~4.10.3 and Remark~4.10.4]{EGNO}.

\begin{proposition}\label{prop:pointed-split-field}
Let $L/K$ be a finite Galois extension, let $G$ be a finite group, and let
$\omega\in Z^3(G,L^\times)$. Then $\Vect_{G,L}^{\omega}$ has a split $K$-form
if and only if
\begin{equation}\label{eq:pointed-split-field}
 [\omega]\in
 \operatorname{im}\left(
 H^3(G,K^\times)\longrightarrow H^3(G,L^\times)
 \right).
\end{equation}
\end{proposition}

\begin{proof}
A split form of a pointed category is pointed because invertibility of its simple
objects may be checked after scalar extension. The skeletal proof of the cited
classification is field-independent for split categories and identifies the form with
$\Vect_{G',K}^{\eta}$ for a finite group $G'\cong G$ and
$[\eta]\in H^3(G',K^\times)$. After identifying $G'$ with $G$, scalar extension
places $[\omega]$ in the $\Aut(G)$-orbit of the image of $[\eta]$. This image is
$\Aut(G)$-stable by naturality, and hence contains $[\omega]$. Conversely, a
$K$-valued representative of a class in~\eqref{eq:pointed-split-field} defines a
split pointed category over $K$ whose scalar extension is tensor equivalent to
$\Vect_{G,L}^{\omega}$.
\end{proof}

\begin{corollary}\label{cor:pointed-split-real}
For a finite group $G$ and a normalized cocycle
$\omega\in Z^3(G,\C^\times)$,
\[
 \Vect_G^\omega\text{ has a split real form}
 \quad\Longleftrightarrow\quad
 2[\omega]=0.
\]
Equivalently, $[\omega]$ has a representative with values in $\{\pm1\}$.
\end{corollary}

\begin{proof}
Proposition~\ref{prop:pointed-split-field} reduces the question to the image of
\[
 H^3(G,\R^\times)\longrightarrow H^3(G,\C^\times).
\]
Since $\R^\times\cong\{\pm1\}\times\R_{>0}$ and the positive reals are uniquely
divisible, their positive-degree cohomology under the finite group $G$ vanishes.
The displayed image is therefore the image of $H^3(G,\{\pm1\})$. Exactness of the
cohomology sequence associated with
\[
 1\longrightarrow\{\pm1\}\longrightarrow\C^\times
 \xrightarrow{z\mapsto z^2}\C^\times\longrightarrow1
\]
identifies this image with the kernel of multiplication by two on
$H^3(G,\C^\times)$.
\end{proof}

The equivalent condition $[\omega^2]=1$ for the existence of real
$F$-symbols is obtained independently in~\cite[Section~7]{BHPReality}.
Proposition~\ref{prop:pointed-split-field} places this criterion in the more
general setting of split forms over finite Galois extensions.

Thus every complex categorification of $\mathbb ZG$ has a split real form precisely
when $H^3(G,\C^\times)$ has exponent at most two. For example,
\cite[Example~2.6.4]{EGNO} gives
\[
 H^3(C_n,\C^\times)\cong\mathbb Z/n\mathbb Z.
\]
The category corresponding to $a\in\mathbb Z/n\mathbb Z$ has a split real form
precisely when $2a=0$.
In particular, for $C_3$ only the trivial associator class admits a split real
form, while for $C_4$ exactly the classes of order at most two do.

\section{Split forms for
\texorpdfstring{$\mathcal C(\mathfrak g,k)$}{C(g,k)}}\label{sec:Cgk}

Let $\mathfrak g$ be a complex simple Lie algebra and let $k$ be a positive
integer. Following
\cite[Section~8.18.2]{EGNO} and
\cite[Section~5.1]{GalindoMoraRowellVerlinde}, we write
$\mathcal C(\mathfrak g,k)$ for the modular category associated with the simply
connected root datum of $\mathfrak g$ at level $k$. We prove that its underlying
fusion category has a split form over a totally real cyclotomic field, and that
the same is true for all its cyclotomic Galois conjugates. The split real form is
then obtained by extension of scalars. By
Proposition~\ref{prop:framed-split-descent}, the essential step is to construct a
tensor equivalence over cyclotomic conjugation, a coherent monoidal isomorphism
$\mu:\Phi^2\Rightarrow\id$, and, for every simple object $X$, an isomorphism
$f_X:\Phi(X)\to X$ satisfying
\[
 f_X\circ\Phi(f_X)=\mu_X.
\]
We first construct the tensor involution on the ambient quantum-group module
category. We then show that it preserves the tilting category and its negligible
ideal, pass to the semisimplification, and finally construct compatible frames on
the simple objects.

\subsection{The quantum-group realization}

Let $h^\vee$ be the dual Coxeter number of $\mathfrak g$ and let
$r^\vee\in\{1,2,3\}$ be its lacing number, the ratio of the squared length of a
long root to that of a short root. Put
\[
 \ell=r^\vee(k+h^\vee),\qquad
 \zeta=\exp\!\left(\frac{\pi i}{\ell}\right).
\]
Thus $\zeta$ is a primitive root of unity of order $N=2\ell$. Set
\begin{equation}\label{eq:Cgk-cyclotomic-fields}
 L=\Q(\zeta),
 \qquad
 K=L^+=\Q(\zeta+\zeta^{-1}).
\end{equation}
Since $N>2$, the field $K$ is the maximal totally real subfield of $L$, and
\[
 \Gal(L/K)=\langle\gamma\rangle,
 \qquad
 \gamma(\zeta)=\zeta^{-1}.
\]
In the notation of
\cite[Section~3.3]{BakalovKirillov}, the shifted level is $k+h^\vee$ and the
parameter denoted there by $q$ is our $\zeta$. For comparison, the parameters
$D,l,l'$ of~\cite[Introduction and Section~2]{SawinQuantum} have the values
\[
 D=r^\vee,\qquad l=2\ell,\qquad l'=\ell,
\]
and hence $2D\mid l$ and $D\mid l'$. In particular, this is the integer-level case
of~\cite[Remark~8]{SawinQuantum}.

For the simply connected root datum, let $P$ be the weight lattice and let
$Y=Q^\vee$ be the cocharacter lattice, which is the coroot lattice. Write
$\langle\ ,\ \rangle:P\times Y\to\mathbb Z$ for the natural pairing. Let $P_+$
be the set of dominant integral weights and let $\theta^\vee$ be the coroot of
the highest long root. Put
\[
 \Lambda_k=\{\lambda\in P_+\mid
 \langle\lambda,\theta^\vee\rangle\leq k\}.
\]
In the conventions of~\cite[Section~3.3]{BakalovKirillov}, the open principal
alcove is
\[
 \{\lambda\in P_+\mid
   \langle\lambda+\rho,\theta^\vee\rangle<k+h^\vee\}.
\]
Since $\langle\rho,\theta^\vee\rangle=h^\vee-1$ and the pairings are integral,
this alcove is precisely $\Lambda_k$. This also verifies the numerical hypothesis
in~\cite[Theorem~2]{SawinQuantum}, because its parameters satisfy
$l'=\ell=r^\vee(k+h^\vee)\geq r^\vee h^\vee=D h^\vee$.

We recall the quantum-group conventions used below; compare
\cite[Section~3.1]{RowellUnitaryQuantumGroups} for an overview and
\cite[Sections~3.1.1--3.1.4]{LusztigQuantumGroups} for the algebra and its Hopf
structure. Let $I$ index the simple roots $\alpha_i$ and coroots
$\alpha_i^\vee$. Normalize the invariant form by requiring every short root to
have squared length $2$, and put
\[
 d_i=\frac{(\alpha_i,\alpha_i)}2,\qquad
 v_i=v^{d_i},\qquad
 \widetilde K_i=K_{d_i\alpha_i^\vee}.
\]
Thus $d_i=1$ for a short root and $d_i=r^\vee$ for a long root. Over
$\Q(v)$, the simply connected Drinfeld--Jimbo algebra $U_v(\mathfrak g)$ is
generated by $E_i,F_i$ $(i\in I)$ and $K_\mu$ $(\mu\in Y)$. To fix the
conventions used below, we record the following relations:
\begin{align*}
 K_0&=1,& K_\mu K_\nu&=K_{\mu+\nu},\\
 K_\mu E_iK_{-\mu}&=v^{\langle\alpha_i,\mu\rangle}E_i,
 &K_\mu F_iK_{-\mu}&=v^{-\langle\alpha_i,\mu\rangle}F_i,\\
 [E_i,F_j]&=\delta_{ij}
 \frac{\widetilde K_i-\widetilde K_i^{-1}}{v_i-v_i^{-1}}.
\end{align*}
We use the coproduct convention
\begin{equation}\label{eq:quantum-coproduct}
 \begin{aligned}
  \Delta(K_\mu)&=K_\mu\otimes K_\mu,\\
  \Delta(E_i)&=E_i\otimes1+\widetilde K_i\otimes E_i,\\
  \Delta(F_i)&=1\otimes F_i+F_i\otimes\widetilde K_i^{-1}.
 \end{aligned}
\end{equation}

Put $\mathcal A=\mathbb Z[v,v^{-1}]$ and
\[
 [n]_{v_i}=\frac{v_i^n-v_i^{-n}}{v_i-v_i^{-1}},\qquad
 E_i^{(n)}=\frac{E_i^n}{[n]_{v_i}!},\qquad
 F_i^{(n)}=\frac{F_i^n}{[n]_{v_i}!}\qquad(n\geq0).
\]
Lusztig's simply connected divided-power form
$U_{\mathcal A}=U_{\mathcal A}(\mathfrak g)$ is the $\mathcal A$-subalgebra
generated by these divided powers and the elements $K_\mu$ $(\mu\in Y)$
\cite[Section~3.1.13]{LusztigQuantumGroups}. The coproduct formulas for divided
powers show that it is stable under~\eqref{eq:quantum-coproduct}
\cite[Section~3.1.5]{LusztigQuantumGroups}. Write
\[
 U_{\zeta,L}(\mathfrak g)=L\otimes_{\mathcal A}U_{\mathcal A},
 \qquad
 U_\zeta(\mathfrak g)=\C\otimes_LU_{\zeta,L}(\mathfrak g),
\]
where $\mathcal A$ acts through $v\mapsto\zeta$.

A finite-dimensional $U_\zeta(\mathfrak g)$-module is of \emph{type~$1$} if it
has a weight decomposition
\[
 M=\bigoplus_{\lambda\in P}M_\lambda,
 \qquad
 K_\mu m=\zeta^{\langle\lambda,\mu\rangle}m
 \quad(m\in M_\lambda,\ \mu\in Y),
\]
and it is \emph{integrable} if every $E_i$ and $F_i$ acts locally nilpotently. A
type~$1$ module is \emph{tilting} if it admits both a Weyl filtration and a dual
Weyl filtration, equivalently if it and its dual admit Weyl filtrations; see
\cite[Section~3.1]{RowellUnitaryQuantumGroups} and
\cite[Section~3.3]{BakalovKirillov}. Let $\mathcal T_\zeta$ be the category of
finite-dimensional type~$1$ tilting modules. Its ribbon structure is given in
\cite[Corollary~3.3.14]{BakalovKirillov}.

For $\lambda\in P_+$, denote by $\Delta_\zeta(\lambda)$,
$L_\zeta(\lambda)$, and $T_\zeta(\lambda)$ the Weyl, simple, and indecomposable
tilting modules of highest weight $\lambda$.

\begin{lemma}\label{lem:Cgk-tilting-semisimplification}
At the root of unity $\zeta$ fixed above, the following statements hold.
\begin{enumerate}[label=\textup{(\alph*)}]
\item For every $\lambda\in\Lambda_k$,
\begin{equation}\label{eq:alcove-simple-module}
 T_\zeta(\lambda)\cong\Delta_\zeta(\lambda)\cong L_\zeta(\lambda)
 \qquad(\lambda\in\Lambda_k).
\end{equation}
\item Every tilting module is a direct sum of the modules in
\textup{(a)} and indecomposable tiltings $T_\zeta(\lambda)$ with
$\lambda\notin\Lambda_k$. The latter are negligible.
\item Let $\mathcal I_{\mathrm{neg}}$ be the trace ideal: a morphism
$f:X\to Y$ belongs to $\mathcal I_{\mathrm{neg}}$ when the quantum trace of
$gf$ vanishes for every $g:Y\to X$. Then $\mathcal I_{\mathrm{neg}}$ is precisely
the ideal of morphisms factoring through direct sums of the negligible
indecomposable tiltings. The quotient
\begin{equation}\label{eq:Cgk-quantum-realization}
 \mathcal C(\mathfrak g,k)
 \simeq\mathcal T_\zeta/\mathcal I_{\mathrm{neg}}.
\end{equation}
The quotient is semisimple, has simple objects indexed by $\Lambda_k$, and is
modular.
\end{enumerate}
\end{lemma}

\begin{proof}
The linkage principle and the tilting-module theorems give~\textup{(a)} and the
decomposition in~\textup{(b)}; see
\cite[Theorems~3.3.9 and~3.3.13]{BakalovKirillov}. The negligibility of the
indecomposable summands outside the alcove follows from
\cite[Lemma~3.3.16]{BakalovKirillov} and, in the present normalization, from
\cite[Theorem~2]{SawinQuantum}. These results originate in the tilting construction
and reduced tensor product of~\cite[Sections~3.18--3.21]{AndersenParadowski}; see
also~\cite[Theorem~3.5]{KirillovInnerProduct}.

For~\textup{(c)}, decompose each tilting module as the direct sum of its alcove
part and its negligible part. The trace pairing on the full subcategory generated
by the alcove simples is nondegenerate, since these simples have nonzero quantum
dimension. Hence a trace-negligible morphism has zero alcove-to-alcove block and
therefore factors through the sum of the negligible parts of its source and target.
The converse follows because every endomorphism of a negligible tilting has zero
quantum trace. This proves the asserted description of the ideal; compare
\cite[Lemmas~3.3.16--3.3.18]{BakalovKirillov},
\cite[Definition~3.6]{KirillovInnerProduct}, and
\cite[Section~3.3(c)--(d)]{WenzlCStar}. The quotient description and
semisimplicity are also given in~\cite[Theorem~5]{SawinQuantum}. Modularity in the
integer-level, simply connected case follows from
\cite[Theorem~3.3.20]{BakalovKirillov}; see also
\cite[Theorem~6 and Remark~8]{SawinQuantum}.
\end{proof}

\subsection{The tensor-product bar}

The bar involution of $U_{\mathcal A}$ is the $\mathbb Z$-linear algebra
involution determined by
\begin{equation}\label{eq:quantum-bar}
 \overline v=v^{-1},\qquad
 \overline{E_i^{(n)}}=E_i^{(n)},\qquad
 \overline{F_i^{(n)}}=F_i^{(n)},\qquad
 \overline{K_\mu}=K_{-\mu}.
\end{equation}
Here $\mu\in Y$. The bar of
\cite[Section~3.1.12]{LusztigQuantumGroups} preserves the integral form defined
in~\cite[Section~3.1.13]{LusztigQuantumGroups}, since it fixes the divided powers
and sends $K_\mu$ to $K_{-\mu}$. The algebra bar is not a Hopf involution for the coproduct
in~\eqref{eq:quantum-coproduct}; for example, with the bar acting componentwise,
\[
 \overline{\Delta(E_i)}
 =E_i\otimes1+\widetilde K_i^{-1}\otimes E_i
 \neq\Delta(\overline{E_i}).
\]
In the following formula, the bar on scalars is complex conjugation and the bar on
$U_{\mathcal A}$ is the involution in~\eqref{eq:quantum-bar}. Since
$\overline\zeta=\zeta^{-1}$, the formula
\[
 \sigma_\zeta(a\otimes x)=\overline a\otimes\overline x
 \qquad(a\in\C,\ x\in U_{\mathcal A})
\]
is well defined on $U_\zeta(\mathfrak g)$ and defines a conjugate-linear algebra
involution. Thus $\sigma_\zeta$ denotes the specialized bar.
For a complex vector space $M$, write $\overline M$ for its conjugate vector space,
whose elements are denoted $\overline m$ and whose scalar multiplication is
$a\,\overline m=\overline{\overline a m}$.

Define a conjugate-linear functor on finite-dimensional type~$1$ integrable
$U_\zeta(\mathfrak g)$-modules by
\begin{equation}\label{eq:bar-twisted-module}
 \Phi(M)=\overline M,\qquad
 u\cdot\overline m=\overline{\sigma_\zeta(u)m},\qquad
 \Phi(f)(\overline m)=\overline{f(m)}.
\end{equation}

The evident vector-space identification
\[
 \Phi(M)\otimes\Phi(N)\longrightarrow\Phi(M\otimes N),
 \qquad
 \overline m\otimes\overline n\longmapsto\overline{m\otimes n},
\]
is not generally a module map, because
$\Delta\circ\sigma_\zeta\neq
(\sigma_\zeta\otimes\sigma_\zeta)\circ\Delta$.
Lusztig's tensor-product bar
\cite[Sections~27.3.1--27.3.3 and~27.3.6]{LusztigQuantumGroups} uses the
quasi-$R$-matrix to supply an intertwiner between these two actions and hence the
required tensorator.
Write $\Theta=\sum_\nu\Theta_\nu$ for the quasi-$R$-matrix of
\cite[Section~4.1]{LusztigQuantumGroups}.

\begin{lemma}\label{lem:specialized-quasi-R}
For finite-dimensional type~$1$ integrable modules $M,N$, every homogeneous
component $\Theta_\nu$ specializes at $v=\zeta$, and the sum of the specialized
operators
\[
 \Theta_\zeta=\sum_\nu(\Theta_\nu)_\zeta
\]
is locally finite on $M\otimes N$ and defines a natural invertible operator
satisfying
\begin{equation}\label{eq:quasi-R-intertwining}
 \Delta(\sigma_\zeta(u))\Theta_\zeta
 =\Theta_\zeta(\sigma_\zeta\otimes\sigma_\zeta)\Delta(u)
 \qquad(u\in U_\zeta(\mathfrak g)).
\end{equation}
On triple tensor products one has
\begin{equation}\label{eq:quasi-R-associativity}
 (\Delta\otimes\id)(\Theta_\zeta)\Theta_\zeta^{12}
 =(\id\otimes\Delta)(\Theta_\zeta)\Theta_\zeta^{23},
\end{equation}
and on every tensor product,
\begin{equation}\label{eq:bar-quasi-R-inverse}
 (\sigma_\zeta\otimes\sigma_\zeta)(\Theta_\zeta)=\Theta_\zeta^{-1}.
\end{equation}
These are identities of locally finite operators.
\end{lemma}

\begin{proof}
In finite type, Corollary~24.1.6 of~\cite{LusztigQuantumGroups} expresses each
$\Theta_\nu$ in the canonical bases of the negative and positive parts with
coefficients in $\mathcal A$. Thus every homogeneous component specializes at
$v=\zeta$. On a vector in a finite-dimensional integrable tensor product, only
finitely many components act nontrivially; hence their sum defines the asserted
locally finite natural operator.

For $x\in U_{\mathcal A}$, Theorem~4.1.2 of~\cite{LusztigQuantumGroups} gives
\[
 \Delta(x)\Theta=\Theta\,\overline{\Delta(\overline x)},
\]
where the outer bar acts componentwise on
$U_{\mathcal A}\otimes U_{\mathcal A}$. Replacing $x$ by $\overline x$ and
specializing gives~\eqref{eq:quasi-R-intertwining}.
Proposition~4.2.4 of~\cite{LusztigQuantumGroups}, summed over the homogeneous
components and then specialized, gives~\eqref{eq:quasi-R-associativity}.
Finally, Corollary~4.1.3 of~\cite{LusztigQuantumGroups} gives
$\overline\Theta=\Theta^{-1}$ in the completed tensor product; specialization on
integrable modules gives~\eqref{eq:bar-quasi-R-inverse}.
\end{proof}

For modules $M,N$, use the operator of
Lemma~\ref{lem:specialized-quasi-R} to define
\begin{align}
 \Phi_2(M,N):\Phi(M)\otimes\Phi(N)&\longrightarrow\Phi(M\otimes N),
 &\Phi_2(M,N)(\overline m\otimes\overline n)
  &=\overline{\Theta_\zeta(m\otimes n)},
  \label{eq:bar-tensorator}\\
 \mu_M:\Phi^2(M)&\longrightarrow M,
 &\mu_M(\overline{\overline m})&=m.
 \label{eq:double-bar-mu}
\end{align}
Let $\Phi_0:\mathbf 1\to\Phi(\mathbf 1)$ be the map
$\C\to\overline\C$ sending $1$ to $\overline1$.

\begin{proposition}\label{prop:quantum-bar-tensor-involution}
The data $(\Phi,\Phi_2,\Phi_0,\mu)$ define a coherent conjugate-linear tensor
involution on the category of finite-dimensional type~$1$ integrable
$U_\zeta(\mathfrak g)$-modules.
\end{proposition}

\begin{proof}
The functor in~\eqref{eq:bar-twisted-module} is well defined because $\sigma_\zeta$
is a conjugate-linear algebra involution. If $M_\lambda$ is a weight space, then
$\Phi(M)_\lambda=\overline{M_\lambda}$ because
$\sigma_\zeta(K_\mu)=K_{-\mu}$ and $\overline\zeta=\zeta^{-1}$.
Moreover, $\sigma_\zeta$ fixes the divided powers of the $E_i$ and $F_i$.
Consequently, $\Phi$ preserves type~$1$ modules and integrability.

The identities in Lemma~\ref{lem:specialized-quasi-R} control the tensor and
coherence axioms.
For a pure tensor $x=m\otimes n\in M\otimes N$, applying the action
before~\eqref{eq:bar-tensorator} gives
\[
 \overline{\Theta_\zeta(\sigma_\zeta\otimes\sigma_\zeta)\Delta(u)x},
\]
whereas applying the tensorator first gives
\[
 \overline{\Delta(\sigma_\zeta(u))\Theta_\zeta x}.
\]
These expressions agree by~\eqref{eq:quasi-R-intertwining}. Hence $\Phi_2(M,N)$ is
a natural module isomorphism; its inverse is obtained from $\Theta_\zeta^{-1}$.

Equation~\eqref{eq:quasi-R-associativity} gives the pentagon for $\Phi_2$, and the
degree-zero term $\Theta_0=1\otimes1$ gives the unit axioms for $\Phi_2$ and
$\Phi_0$. Under $\mu$, the tensorator of $\Phi^2$ acts by
$\Theta_\zeta(\sigma_\zeta\otimes\sigma_\zeta)(\Theta_\zeta)=1$.
Equation~\eqref{eq:bar-quasi-R-inverse} therefore shows that $\mu$ is monoidal, and
$\Phi(\mu_M)=\mu_{\Phi(M)}$ follows directly from
$\overline{\overline{\overline m}}\mapsto\overline m$. Compatibility with the unit
follows from the degree-zero component of $\Theta$, which is $1$, and the definition
of $\Phi_0$.
\end{proof}

\subsection{Descent to
\texorpdfstring{$\mathcal C(\mathfrak g,k)$}{C(g,k)}}

\begin{lemma}\label{lem:bar-preserves-tilting-negligible}
The tensor involution of Proposition~\ref{prop:quantum-bar-tensor-involution}
preserves the tilting category $\mathcal T_\zeta$ and the ideal
$\mathcal I_{\mathrm{neg}}$. It therefore descends, together with its tensorator
and coherence, to a conjugate-linear tensor involution of
$\mathcal C(\mathfrak g,k)$.
\end{lemma}

\begin{proof}
Let $\lambda\in P_+$ and choose a highest-weight vector
$v_\lambda\in\Delta_\zeta(\lambda)$. In the bar-twisted module,
\begin{align*}
 K_\mu\cdot\overline{v_\lambda}
 &=\overline{K_{-\mu}v_\lambda}
  =\overline{\zeta^{-\langle\lambda,\mu\rangle}v_\lambda}
  =\zeta^{\langle\lambda,\mu\rangle}\overline{v_\lambda},\\
 E_i\cdot\overline{v_\lambda}&=0.
\end{align*}
Consequently,
\begin{equation}\label{eq:bar-fixes-standards}
 \Phi(\Delta_\zeta(\lambda))\cong\Delta_\zeta(\lambda).
\end{equation}
Indeed, the bar-twisted module is generated by $\overline{v_\lambda}$ and has highest
weight $\lambda$. The universal property of the Weyl module gives a surjection onto
it, which is an isomorphism because the two modules have the same dimension.
The conjugate-linear strong tensor equivalence $\Phi$ preserves duals. Since
the dual Weyl module of highest weight $\lambda$ is
$\Delta_\zeta(-w_0\lambda)^\vee$, it also fixes every dual Weyl module. It therefore
preserves modules admitting both Weyl and dual Weyl filtrations; here we use that
$\Phi$ is exact. By uniqueness of the indecomposable tilting module of highest
weight $\lambda$
\cite[Theorem~3.3.13(ii)]{BakalovKirillov},
\begin{equation}\label{eq:bar-fixes-tiltings}
 \Phi(T_\zeta(\lambda))\cong T_\zeta(\lambda).
\end{equation}
Lemma~\ref{lem:Cgk-tilting-semisimplification} and
\eqref{eq:bar-fixes-tiltings} show that $\Phi$ preserves both
$\mathcal T_\zeta$ and the morphisms factoring through negligible indecomposable
tiltings;
hence it preserves $\mathcal I_{\mathrm{neg}}$. Notice that this argument does not
require the bar to preserve the quantum trace or the ribbon structure. The functor
therefore descends, together with $\Phi_2$ and $\mu$, to the quotient
$\mathcal T_\zeta/\mathcal I_{\mathrm{neg}}\simeq\mathcal C(\mathfrak g,k)$.
We regard the descended conjugate-linear endofunctor as the corresponding twisted
$\C$-linear equivalence
${}^{\gamma}\mathcal C(\mathfrak g,k)\to\mathcal C(\mathfrak g,k)$.
\end{proof}

\begin{lemma}\label{lem:canonical-simple-bars}
For every $\lambda\in\Lambda_k$, the integral highest-weight bar specializes to a
linear module isomorphism
\[
 f_\lambda:\Phi(L_\zeta(\lambda))\longrightarrow L_\zeta(\lambda)
\]
satisfying
$f_\lambda\circ\Phi(f_\lambda)=\mu_{L_\zeta(\lambda)}$.
\end{lemma}

\begin{proof}
For $\lambda\in\Lambda_k$, let
$\Delta_{\mathcal A}(\lambda)={}_{\mathcal A}\Lambda_\lambda$ be the integral
highest-weight module with highest-weight vector $\eta_\lambda$, in the notation of
\cite[Sections~19.3.1--19.3.2]{LusztigQuantumGroups}. It is stable under the
semilinear involution $\psi_{\lambda,\mathcal A}$ characterized by
\[
 \psi_{\lambda,\mathcal A}(u\eta_\lambda)
 =\overline u\,\eta_\lambda,
 \qquad
 \psi_{\lambda,\mathcal A}^2=\id
 \qquad(u\in U_{\mathcal A});
\]
see~\cite[Section~19.3.4]{LusztigQuantumGroups}. The specialization construction
in~\cite[equation~(3.13)]{KirillovInnerProduct} gives
\[
 \Delta_\zeta(\lambda)
 \cong\C\otimes_{\mathcal A}\Delta_{\mathcal A}(\lambda).
\]
Under this identification, take $v_\lambda=1\otimes\eta_\lambda$. Since
$\overline{r(\zeta)}=\overline r(\zeta)$ for $r\in\mathcal A$, where the first bar
is complex conjugation and the second is the involution $v\mapsto v^{-1}$, the
formula
\[
 \psi_{\lambda,\zeta}(c\otimes m)
 =\overline c\otimes\psi_{\lambda,\mathcal A}(m)
\]
is well defined. It is a conjugate-linear involution satisfying
\[
 \psi_{\lambda,\zeta}(um)
 =\sigma_\zeta(u)\psi_{\lambda,\zeta}(m)
 \qquad(u\in U_\zeta(\mathfrak g)).
\]
By~\eqref{eq:alcove-simple-module}, the canonical quotient
$\Delta_\zeta(\lambda)\to L_\zeta(\lambda)$ is an isomorphism. Denote the induced
involution of $L_\zeta(\lambda)$ by $\psi_\lambda$; it fixes $v_\lambda$ and
satisfies $\psi_\lambda^2=\id$.
Thus
\begin{equation}\label{eq:simple-stable-bar}
 f_\lambda:\Phi(L_\zeta(\lambda))\longrightarrow L_\zeta(\lambda),
 \qquad f_\lambda(\overline m)=\psi_\lambda(m),
\end{equation}
is a linear module isomorphism and satisfies
\begin{equation}\label{eq:positive-bar-sign}
 f_\lambda\circ\Phi(f_\lambda)=\mu_{L_\zeta(\lambda)},
\end{equation}
because both sides send $\overline{\overline m}$ to $m$. Hence every simple object
is fixed by $\Phi$ up to isomorphism.
\end{proof}

For $\lambda\in\Lambda_k$, write
\begin{equation}\label{eq:cyclotomic-simple-module}
 \Delta_{\zeta,L}(\lambda)
 =L\otimes_{\mathcal A}\Delta_{\mathcal A}(\lambda).
\end{equation}

The pivotal structure needed for semisimplification is also defined over $L$.
Put
\begin{equation}\label{eq:quantum-pivotal-element}
 \kappa_\rho
 =\sum_{\alpha>0}\frac{(\alpha,\alpha)}2\,\alpha^\vee\in Y,
 \qquad
 g_{\mathrm{piv}}=K_{-\kappa_\rho}\in U_{\mathcal A}^0.
\end{equation}
For the coproduct convention~\eqref{eq:quantum-coproduct}, the antipode satisfies
$S^2(u)=g_{\mathrm{piv}}u g_{\mathrm{piv}}^{-1}$. On a vector of weight
$\lambda$, the pivotal element acts by
\[
 v^{-\langle\lambda,\kappa_\rho\rangle}
 =v^{-(\lambda,2\rho)}.
\]
The exponent is integral because $2\rho$ belongs to the root lattice, so
$(\lambda,2\rho)$ is an integral linear combination of the integers
$(\lambda,\alpha_i)=d_i\langle\lambda,\alpha_i^\vee\rangle$.
Thus $g_{\mathrm{piv}}$ specializes over $L$, and the associated pivotal traces,
given by ordinary traces with insertion of $g_{\mathrm{piv}}^{\pm1}$ according
to the left or right convention, take values in $L$.

\begin{proposition}\label{prop:bar-descent-Cgk}
Let $\mathcal T_{\zeta,L}$ be the full category of finite-dimensional type~$1$
$U_{\zeta,L}(\mathfrak g)$-modules whose scalar extensions to $\C$ are tilting,
let
$\mathcal I_{\mathrm{neg},L}$ be the negligible ideal defined by the standard
pivotal trace on $\mathcal T_{\zeta,L}$, and put
\[
 \mathcal C_L(\mathfrak g,k)
 =\operatorname{Kar}
 \bigl(\mathcal T_{\zeta,L}/\mathcal I_{\mathrm{neg},L}\bigr).
\]
Then $\mathcal C_L(\mathfrak g,k)$ is a split fusion category over $L$, and
\begin{equation}\label{eq:Cgk-cyclotomic-complexification}
 \mathcal C_L(\mathfrak g,k)\otimes_L\C
 \simeq\mathcal C(\mathfrak g,k).
\end{equation}
Moreover, the bar involution, its quasi-$R$-matrix tensorator, and the canonical
frames determine an element
\[
 \tau_L\in
 \Pi_{\mathrm{fr}}\bigl(\mathcal C_L(\mathfrak g,k)/K\bigr)
\]
such that $p_{\mathrm{fr}}(\tau_L)=\gamma$ and $\tau_L^2=1$.
\end{proposition}

\begin{proof}
The quantum-group, tilting, duality, and pivotal constructions needed below are
defined over $L$. They are obtained from $U_{\mathcal A}$, its highest-weight
modules, and the pivotal element~\eqref{eq:quantum-pivotal-element} by the
specialization $v\mapsto\zeta$. In particular, the pivotal trace of an
$L$-linear endomorphism takes values in $L$. After extension of scalars to $\C$,
this pivotal trace is the standard quantum trace on $\mathcal T_\zeta$.

Since the Hom spaces are finite-dimensional, the radicals of the pivotal-trace
pairings commute with the scalar extension $L\subset\C$. Thus, for
$X,Y\in\mathcal T_{\zeta,L}$,
\begin{equation}\label{eq:negligible-ideal-base-change}
 \mathcal I_{\mathrm{neg},L}(X,Y)\otimes_L\C
 =\mathcal I_{\mathrm{neg}}(X_{\C},Y_{\C}).
\end{equation}
Equation~\eqref{eq:negligible-ideal-base-change} also shows that the cyclotomic
bar preserves $\mathcal I_{\mathrm{neg},L}$. Indeed, if
$f\in\mathcal I_{\mathrm{neg},L}$, then
$f_{\C}\in\mathcal I_{\mathrm{neg}}$; the complex bar preserves the latter
ideal by Lemma~\ref{lem:bar-preserves-tilting-negligible}, and its image is the
scalar extension of the cyclotomic bar of $f$. Faithfulness of scalar extension
then implies that the cyclotomic bar of $f$ belongs to
$\mathcal I_{\mathrm{neg},L}$.
Hence the quotient commutes with scalar extension. The Karoubi envelope is
included to make the quotient idempotent complete. We make no assertion here
that the standard braiding or ribbon twist is defined over $L$; in general,
additional roots of unity are required for those structures.

For $\lambda\in\Lambda_k$, the scalar extension of
\eqref{eq:cyclotomic-simple-module} to $\C$ is the simple module
$\Delta_\zeta(\lambda)\cong L_\zeta(\lambda)$. Thus
$\Delta_{\zeta,L}(\lambda)$ is an absolutely simple tilting module, and
\[
 \operatorname{End}_{U_{\zeta,L}(\mathfrak g)}
 \bigl(\Delta_{\zeta,L}(\lambda)\bigr)=L.
\]
After extending scalars, these objects give all the pairwise nonisomorphic
simples of $\mathcal C(\mathfrak g,k)$. By
\eqref{eq:negligible-ideal-base-change}, their endomorphism algebras in the
quotient, and hence in its Karoubi envelope, satisfy
\[
 \operatorname{End}_{\mathcal C_L(\mathfrak g,k)}
 \bigl(\Delta_{\zeta,L}(\lambda)\bigr)=L.
\]

We next justify semisimplicity over $L$. For an object
$X\in\mathcal C_L(\mathfrak g,k)$, put
$A_X=\operatorname{End}_{\mathcal C_L(\mathfrak g,k)}(X)$. Scalar extension
gives
\[
 A_X\otimes_L\C
 \cong
 \operatorname{End}_{\mathcal C(\mathfrak g,k)}(X_{\C}).
\]
The algebra on the right is semisimple. The Jacobson radical $J(A_X)$ is
nilpotent, so $J(A_X)\otimes_L\C$ is a nilpotent two-sided ideal of the
semisimple algebra $A_X\otimes_L\C$ and must vanish. Faithfulness of scalar
extension implies $J(A_X)=0$. More explicitly, applying this to $X\oplus Y$
shows that for every morphism $f:X\to Y$ there is a morphism $g:Y\to X$ such
that $fgf=f$. The idempotents $fg$ and $gf$ split in the Karoubi envelope, so
every morphism is a split map between direct summands. Hence
$\mathcal C_L(\mathfrak g,k)$ is semisimple.

Finally, let $X$ be a simple object of $\mathcal C_L(\mathfrak g,k)$. The
nonzero semisimple object $X_{\C}$ contains some alcove simple
$L_\zeta(\lambda)$. By scalar extension,
\[
 \operatorname{Hom}_{\mathcal C_L(\mathfrak g,k)}
 \bigl(\Delta_{\zeta,L}(\lambda),X\bigr)\otimes_L\C
 \cong
 \operatorname{Hom}_{\mathcal C(\mathfrak g,k)}
 \bigl(L_\zeta(\lambda),X_{\C}\bigr),
\]
and the right-hand side is nonzero. Hence there is a nonzero morphism
$\Delta_{\zeta,L}(\lambda)\to X$. Both objects are simple in the semisimple
category $\mathcal C_L(\mathfrak g,k)$, so this morphism is an isomorphism.
Thus the $\Delta_{\zeta,L}(\lambda)$ exhaust the simple objects and have
endomorphism algebra $L$. Therefore $\mathcal C_L(\mathfrak g,k)$ is split, and
its scalar extension to $\C$ is equivalent to $\mathcal C(\mathfrak g,k)$,
proving~\eqref{eq:Cgk-cyclotomic-complexification}.

The remaining assertions are the cyclotomic versions of
Lemmas~\ref{lem:bar-preserves-tilting-negligible} and
\ref{lem:canonical-simple-bars}. The involution $\gamma$ on $L$ is characterized
by $\gamma(\zeta)=\zeta^{-1}$. The bar preserves $U_{\mathcal A}$, every
homogeneous component of the quasi-$R$-matrix has coefficients in
$\mathcal A$, and the highest-weight bars preserve
$\Delta_{\mathcal A}(\lambda)$. Consequently, the tensor involution, tensorator,
coherence, and frames used above are all defined for the extension $L/K$.

Let $\tau_L$ be the resulting framed class. Its squared frame at
$\Delta_{\zeta,L}(\lambda)$ is
$f_\lambda\circ\Phi(f_\lambda)$, which equals
$\mu_{\Delta_{\zeta,L}(\lambda)}$ by the integral version of
\eqref{eq:positive-bar-sign}. Since $\mu:\Phi^2\Rightarrow\id$ is monoidal,
$\tau_L^2=1$. Hence $1\mapsto1$ and $\gamma\mapsto\tau_L$ define a section of
\[
 p_{\mathrm{fr}}:
 \Pi_{\mathrm{fr}}\bigl(\mathcal C_L(\mathfrak g,k)/K\bigr)
 \longrightarrow\Gal(L/K).
\]
\end{proof}

\begin{theorem}\label{thm:Cgk-split-cyclotomic}
\label{thm:Cgk-split-real}
For every complex simple Lie algebra $\mathfrak g$ and every integer $k\geq1$, the
underlying fusion category $\mathcal C(\mathfrak g,k)$ has a split form over the
maximal totally real subfield
\[
 K=\Q(\zeta+\zeta^{-1}),
 \qquad
 \zeta=\exp\!\left(
 \frac{\pi i}{r^\vee(k+h^\vee)}
 \right).
\]
More precisely, there is a split fusion category
$\mathcal C_K(\mathfrak g,k)$ over $K$ whose scalar extension to $\C$ is
$\mathcal C(\mathfrak g,k)$. In particular, it has a split real form.
\end{theorem}

\begin{proof}
Proposition~\ref{prop:bar-descent-Cgk} provides a section of
$p_{\mathrm{fr}}$ for the extension $L/K$. By
Proposition~\ref{prop:framed-split-descent}, there is a split fusion category
$\mathcal C_K(\mathfrak g,k)$ over $K$ such that
\[
 \mathcal C_K(\mathfrak g,k)\otimes_KL
 \simeq\mathcal C_L(\mathfrak g,k).
\]
Extending further to $\C$ and using
\eqref{eq:Cgk-cyclotomic-complexification} gives the asserted $K$-form of
$\mathcal C(\mathfrak g,k)$. The standard embedding $K\subset\R$ then gives a
split real form by extension of scalars.
\end{proof}

\begin{corollary}\label{cor:Cgk-galois-split}
Every cyclotomic Galois conjugate of the underlying fusion category
$\mathcal C(\mathfrak g,k)$ has a split form over the same totally real field
$K$ in~\eqref{eq:Cgk-cyclotomic-fields}.
\end{corollary}

\begin{proof}
Let $\sigma\in\Gal(L/\Q)$. Since the cyclotomic extension $L/\Q$ is abelian,
$\sigma$ commutes with $\gamma$ and therefore preserves its fixed field $K$.
Twist the split $K$-form of Theorem~\ref{thm:Cgk-split-cyclotomic} by
$\sigma|_K$. It remains split, and scalar extension gives
\[
 \bigl({}^{\sigma|_K}\mathcal C_K(\mathfrak g,k)\bigr)\otimes_KL
 \simeq{}^\sigma\mathcal C_L(\mathfrak g,k).
\]
Thus it is a split $K$-form of the $\sigma$-conjugate. Every absolute Galois
conjugate restricts to an element of $\Gal(L/\Q)$, which proves the assertion.
\end{proof}

\begin{remark}
Concretely, for $\lambda\in\Lambda_k$, the pairs
$(\Delta_{\zeta,L}(\lambda),f_\lambda)$ descend to simple objects with
endomorphism algebra $K$. After extension to $\R$, their endomorphism algebras
become $\R$. The bar construction splits $p_{\mathrm{fr}}$ rather than its
structured analogue, so it does not itself descend the braiding, ribbon
structure, or modular data. Theorem~\ref{thm:Cgk-split-braided-real} below shows
that this distinction is essential except in a small level-one family.
\end{remark}

\begin{corollary}\label{cor:Cgk-real-orthogonal-F}
For every complex simple Lie algebra $\mathfrak g$ and every integer $k\geq1$,
$\mathcal C(\mathfrak g,k)$ admits orthonormal fusion bases in which all
$F$-matrices are real and orthogonal.
\end{corollary}

\begin{proof}
The category $\mathcal C(\mathfrak g,k)$ is unitary by
\cite[Theorem~3.7]{WenzlCStar} or
\cite[Theorem~4.2]{RowellUnitaryQuantumGroups}. The divisibility hypothesis
$r^\vee\mid\ell$ in the second reference holds because
$\ell=r^\vee(k+h^\vee)$.
Combining Theorem~\ref{thm:Cgk-split-real} with
Subsection~\ref{subsec:unitary-consequences} gives the conclusion.
\end{proof}

For the self-dual Chern--Simons families, the Frobenius--Schur grading
criterion of~\cite[Section~6.7]{BHPReality} independently gives real
$F$-symbols and symmetric $R$-matrices. The corollary above applies to every
complex simple $\mathfrak g$ and every integer $k\geq1$ and arises from split
descent over an explicit totally real field.

\subsection{Split braided real forms}
\label{subsec:Cgk-split-braided-real}

\begin{theorem}\label{thm:Cgk-split-braided-real}
Equip $\mathcal C(\mathfrak g,k)$ with its standard unitary braiding and canonical
unitary ribbon structure. The following conditions are equivalent:
\begin{enumerate}[label=\textup{(\alph*)}]
\item the braided fusion category $\mathcal C(\mathfrak g,k)$ has a split real
form;
\item the ribbon fusion category $\mathcal C(\mathfrak g,k)$ has a split real
form;
\item either $(\mathfrak g,k)=(E_8,1)$ or
$(\mathfrak g,k)=(D_{4m},1)$ for some $m\geq1$.
\end{enumerate}
\end{theorem}

\begin{proof}
The implication \textup{(b)}$\Rightarrow$\textup{(a)} is immediate, while
Proposition~\ref{prop:unitary-braided-real-twist} gives
\textup{(a)}$\Rightarrow$\textup{(b)} and shows that either condition forces all
simple twists to belong to $\{\pm1\}$.

Let $(\ ,\ )_0=(r^\vee)^{-1}(\ ,\ )$, where $(\ ,\ )$ is the form fixed at the
beginning of Section~\ref{sec:Cgk}; thus the long roots have squared length $2$.
Put
\[
 x_\lambda(k)=
 \frac{(\lambda,\lambda+2\rho)_0}{k+h^\vee}.
\]
For $\lambda\in\Lambda_k$, the standard twist is
\begin{equation}\label{eq:Cgk-standard-twist}
 \theta_\lambda=\exp\!\bigl(\pi i x_\lambda(k)\bigr)
\end{equation}
by~\cite[Equation~(5.3)]{GalindoMoraRowellVerlinde}. Hence
$\theta_\lambda\in\{\pm1\}$ exactly when $x_\lambda(k)\in\mathbb Z$.

Suppose first that $k\geq2$. The highest long root $\vartheta$ labels the adjoint
simple object, since
$\langle\vartheta,\vartheta^\vee\rangle=2\leq k$. Moreover,
\[
 (\vartheta,\vartheta+2\rho)_0=2h^\vee,
 \qquad
 x_\vartheta(k)=\frac{2h^\vee}{k+h^\vee}.
\]
This number lies strictly between $0$ and $2$. If it is integral, it must equal
$1$, and therefore $k=h^\vee$. At this remaining level, the root data give the
following simple labels and twist exponents:
\[
\begin{array}{c|c|c@{\qquad}c|c|c}
\mathfrak g&\lambda&x_\lambda(h^\vee)
 &\mathfrak g&\lambda&x_\lambda(h^\vee)\\ \hline
A_r&\omega_1&\dfrac{r(r+2)}{2(r+1)^2}
 &E_6&\omega_1&\dfrac{13}{18}\\[5pt]
B_r&\omega_1&\dfrac{r}{2r-1}
 &E_7&\omega_7&\dfrac{19}{24}\\[5pt]
C_r&\omega_1&\dfrac{2r+1}{4(r+1)}
 &E_8&\omega_1&\dfrac85\\[5pt]
D_r&\omega_1&\dfrac{2r-1}{4(r-1)}
 &F_4&\omega_4&\dfrac23\\[5pt]
 &&&G_2&\omega_1&\dfrac12
\end{array}
\]
Here $r\geq1$ in type $A$, $r\geq2$ in types $B$ and $C$, and $r\geq4$ in
type $D$. These values follow directly from the fundamental weights and $2\rho$
listed in~\cite[Appendix]{GalindoMoraRowellVerlinde}; none is an integer. Thus
$k\geq2$ is impossible.

At level $1$, the same calculation gives
\[
\begin{array}{c|c|c@{\qquad}c|c|c}
\mathfrak g&\lambda&x_\lambda(1)
 &\mathfrak g&\lambda&x_\lambda(1)\\ \hline
A_r&\omega_1&\dfrac{r}{r+1}
 &E_6&\omega_1&\dfrac43\\[5pt]
B_r&\omega_r&\dfrac{2r+1}{8}
 &E_7&\omega_7&\dfrac32\\[5pt]
C_r&\omega_1&\dfrac{2r+1}{2(r+2)}
 &F_4&\omega_4&\dfrac65\\[5pt]
D_r&\omega_r&\dfrac r4
 &G_2&\omega_1&\dfrac45
\end{array}
\]
Every displayed exponent outside type $D$ is nonintegral. In type $D_r$, the
level-one simples are
$0,\omega_1,\omega_{r-1},\omega_r$; the vector object has exponent $1$, while
the two spinor objects have exponent $r/4$. Thus all twists are real exactly
when $4\mid r$. Finally, $\Lambda_1=\{0\}$ in type $E_8$. This proves that
\textup{(a)} or \textup{(b)} implies \textup{(c)}.

It remains to prove existence in the listed cases. The category
$\mathcal C(E_8,1)$ is $\Vect_{\C}$, so it has its evident split braided and
ribbon real form. Let now $\mathfrak g=D_{4m}$ and $k=1$. All four simples are
invertible, and their metric group is $A=\F_2^2$
\cite[Proposition~5.2 and Tables~2--3]{GalindoMoraRowellVerlinde}. Taking the two
spinor objects as a basis, the vector object is their sum, its twist is $-1$, and
the spinor twists are $(-1)^m$. Consider the split pointed category of
finite-dimensional $A$-graded real vector spaces with trivial associator and one
of the following braidings on homogeneous objects:
\[
 \begin{aligned}
 c_0(a,b)&=(-1)^{a_1b_2},\\
 c_1(a,b)&=(-1)^{a_1b_1+a_2b_2+a_1b_2}.
 \end{aligned}
\]
Their quadratic forms $q_j(a)=c_j(a,a)$ satisfy
\[
 \begin{array}{c|ccc}
 &q_j(1,0)&q_j(0,1)&q_j(1,1)\\ \hline
 j=0&1&1&-1\\
 j=1&-1&-1&-1.
 \end{array}
\]
Both $c_0$ and $c_1$ are bicharacters, so they satisfy the hexagon equations
with the trivial associator, and their symmetrizations are nondegenerate. With
the standard spherical structure, their ribbon twists are the displayed
quadratic forms $q_0$ and $q_1$.
Thus the complexification of the first category has the metric group of
$\mathcal C(D_{4m},1)$ when $m$ is even, and the second has its metric group when
$m$ is odd. The classification of pointed braided fusion categories by metric
groups~\cite[Section~8.4]{EGNO} identifies the corresponding complexifications
with $\mathcal C(D_{4m},1)$. Since all associativity, braiding, and twist scalars
in these models lie in $\R$, they give split ribbon real forms. Hence
\textup{(c)} implies \textup{(b)}.
\end{proof}

\begin{remark}
\label{rem:affine-KZ-perspective}
Pavel Etingof suggested considering the relation between
$\mathcal C(\mathfrak g,k)$ and positive-level affine Lie algebra representations
from the perspective of real forms. The construction
in~\cite{McRaeAffineKZ} transports the tensor structure to a category of
finite-dimensional $\mathfrak g$-modules and describes its tensor product and
associator using the KZ equation. After choosing a Chevalley real form, the
quotient formulas for the tensor products are defined over $\R$. The residues of
the KZ equation are also defined over $\R$, which suggests that the regularized
connection matrix, and hence the associator, preserves the corresponding real
spaces.

This suggests an affine Lie algebra approach to the split real form problem for
$\mathcal C(\mathfrak g,k)$. Making this approach rigorous would require
constructing and verifying the corresponding tensor structure over $\R$.
\end{remark}

\subsection{Uniqueness and soft tensor autoequivalences}
\label{subsec:split-real-uniqueness}

The preceding constructions raise two distinct uniqueness questions.
Monoidal natural automorphisms of the identity control the coherent structures
attached to a fixed semilinear equivalence, whereas soft tensor autoequivalences
control the possible simple-fixing semilinear equivalences themselves. We begin
with a general uniqueness statement for a fixed simple-fixing section.

\begin{proposition}\label{prop:fixed-section-split-real-uniqueness}
Let $\mathcal C$ be a fusion category over $\C$, let
$C_2=\{1,\gamma\}$, and let
\[
 \Phi:C_2\longrightarrow\Pi_{\mathrm{sf}}(\mathcal C/\R)
\]
be a section of $p_{\mathrm{sf}}$. If $\Phi$ admits a lift to a section of
$p_{\mathrm{fr}}$, then all such lifts are conjugate by $Q_\Phi$.
Consequently, a split real form inducing the fixed simple-fixing section
$\Phi$ is unique up to equivalence compatible with the chosen identification
of its complexification with $\mathcal C$.
\end{proposition}

\begin{proof}
Put
\[
 E_\Phi=\Aut(\id_{\mathcal C})_\Phi,
 \qquad
 A_\Phi=\Aut_\otimes(\id_{\mathcal C})_\Phi,
 \qquad
 Q_\Phi=E_\Phi/A_\Phi.
\]
Pulling back the framed extension along $\Phi$ gives
\[
 1\longrightarrow Q_\Phi
 \longrightarrow
 C_2\mathop{\times}_{\Pi_{\mathrm{sf}}(\mathcal C/\R)}
 \Pi_{\mathrm{fr}}(\mathcal C/\R)
 \longrightarrow C_2
 \longrightarrow1.
\]
Once one splitting has been chosen, any other splitting differs from it by a
$1$-cocycle with values in $Q_\Phi$, and conjugation by an element of $Q_\Phi$
changes this cocycle by a coboundary. Thus the conjugacy classes of splittings
are parametrized by $H^1(C_2,Q_\Phi)$. We show that this group is trivial.

Since $\Phi$ fixes every simple isomorphism class, evaluation on simple objects
gives $C_2$-module identifications
\[
 E_\Phi\cong
 \operatorname{Map}\bigl(\operatorname{Irr}(\mathcal C),\C^\times\bigr),
 \qquad
 A_\Phi\cong
 \operatorname{Hom}\bigl(U(\mathcal C),\C^\times\bigr),
\]
where $C_2$ acts componentwise by complex conjugation. The cohomology sequence
associated with
\[
 1\longrightarrow A_\Phi
 \longrightarrow E_\Phi
 \longrightarrow Q_\Phi
 \longrightarrow1
\]
contains
\[
 H^1(C_2,E_\Phi)
 \longrightarrow H^1(C_2,Q_\Phi)
 \longrightarrow H^2(C_2,A_\Phi)
 \longrightarrow H^2(C_2,E_\Phi).
\]
Componentwise Hilbert Theorem~90 gives
\[
 H^1(C_2,E_\Phi)=1.
\]
Moreover,
\[
 H^2(C_2,A_\Phi)
 \cong\operatorname{Hom}\bigl(U(\mathcal C),\{\pm1\}\bigr),
 \qquad
 H^2(C_2,E_\Phi)
 \cong\{\pm1\}^{\operatorname{Irr}(\mathcal C)}.
\]
Under these identifications, the last map in the displayed cohomology sequence
is evaluation:
\[
 \chi\longmapsto
 \bigl(\chi(|X|)\bigr)_{X\in\operatorname{Irr}(\mathcal C)}.
\]
The universal grading is faithful, so every element of $U(\mathcal C)$ occurs as
the degree of a simple object. The map is therefore injective, and exactness gives
\[
 H^1(C_2,Q_\Phi)=1,
\]
which proves the assertion.
\end{proof}

This proposition also explains the role of
$\Aut_\otimes(\id_{\mathcal C})$. Although a fixed semilinear equivalence may
admit several coherent $C_2$-structures, changing the coherent structure by a
nontrivial element of
\[
 H^2(C_2,A_\Phi)
 \cong\operatorname{Hom}\bigl(U(\mathcal C),\{\pm1\}\bigr)
\]
changes the relative Brauer sign of at least one simple object. Consequently,
at most one coherent structure on a fixed simple-fixing section can make every
simple object real rather than quaternionic.

Following Davydov~\cite[Section~2.1]{DavydovSoft}, a tensor autoequivalence is
\emph{soft} if its underlying linear functor is naturally isomorphic to the
identity. For a fusion category this is equivalent to acting trivially on the
Grothendieck ring. Thus define the soft tensor autoequivalence group
\[
 \operatorname{Aut}^1_\otimes(\mathcal C)
 =
 \ker\left(
 \operatorname{Eq}(\mathcal C)
 \longrightarrow
 \operatorname{Aut}_{\mathrm{bas}}\bigl(K_0(\mathcal C)\bigr)
 \right),
\]
where $\operatorname{Eq}(\mathcal C)$ is the group of tensor autoequivalence
classes and
$\operatorname{Aut}_{\mathrm{bas}}(K_0(\mathcal C))$ is the group of based-ring
automorphisms. Thus $\operatorname{Aut}^1_\otimes(\mathcal C)$ consists
precisely of the tensor autoequivalence classes that fix every simple
isomorphism class. This group is denoted $\operatorname{Gauge}(\mathcal C)$,
and its elements are called gauge autoequivalences, in
\cite{EdieMichellAutoequivalences,EdieMichellGraded}.

\begin{corollary}\label{cor:gauge-trivial-real-uniqueness}
If
\[
 \operatorname{Aut}^1_\otimes(\mathcal C)=1,
\]
then $\mathcal C$ admits at most one split real form, up to equivalence compatible
with chosen identifications of the complexifications with $\mathcal C$.
\end{corollary}

\begin{proof}
Let $\Phi$ and $\Psi$ be the simple-fixing semilinear sections associated with
two split real forms. Their quotient over the nontrivial element of $C_2$,
\[
 [\Psi_\gamma][\Phi_\gamma]^{-1},
\]
is a complex-linear tensor autoequivalence that fixes every simple isomorphism
class. It therefore belongs to $\operatorname{Aut}^1_\otimes(\mathcal C)$. By the
hypothesis, the two sections coincide in
$\Pi_{\mathrm{sf}}(\mathcal C/\R)$. Proposition
\ref{prop:fixed-section-split-real-uniqueness} then shows that their framed lifts,
and hence the corresponding split real forms with the chosen complexification
identifications, are equivalent.
\end{proof}

\begin{corollary}\label{cor:ABCG-split-real-uniqueness}
Let $\mathfrak g$ be of type
\[
 A_r,\qquad B_r,\qquad C_r,\qquad\text{or}\qquad G_2.
\]
For every positive integral level $k$, the category
$\mathcal C(\mathfrak g,k)$ has a unique split real form, up to equivalence
compatible with a chosen identification of its complexification with
$\mathcal C(\mathfrak g,k)$.
\end{corollary}

\begin{proof}
The soft tensor autoequivalence groups of the Verlinde categories of types $A$,
$B$, $C$, and $G_2$ are trivial by
\cite[Theorem~3.3, Corollary~4.2, Lemma~5.1, and
Corollary~6.2]{EdieMichellAutoequivalences}, where they are called gauge
autoequivalence groups. Theorem~\ref{thm:Cgk-split-real} gives existence, and
Corollary~\ref{cor:gauge-trivial-real-uniqueness} gives uniqueness.
\end{proof}

After extension of the cyclotomic descent datum along $K\subset\R$, let
$\Phi_{\mathrm{bar}}$ denote the conjugate-linear tensor equivalence underlying
the resulting real bar descent. The even $D$-series has an additional ambiguity
visible already from the
universal grading. We first record the general construction that produces it.

\begin{proposition}\label{prop:grading-cocycle-twist}
Let $\mathcal C$ be a fusion category faithfully graded by a finite group $U$,
and let
\[
 \bigl(\Phi,\mu,(f_X)_{X\in\operatorname{Irr}(\mathcal C)}\bigr)
\]
be a simple-fixing split real descent datum. Let
\[
 \alpha\in Z^2(U,\{\pm1\})
\]
be a normalized cocycle. On homogeneous objects define a tensor
autoequivalence $G_\alpha$ whose underlying functor is the identity and whose
tensorator is
\[
 (G_\alpha)_2(X,Y)
 =
 \alpha(|X|,|Y|)\id_{X\otimes Y}.
\]
Then
\[
 \Phi_\alpha:=G_\alpha\circ\Phi
\]
admits the same coherence $\mu$ and the same frames $(f_X)_X$, and therefore
defines another simple-fixing split real descent datum.
We call this operation a \emph{grading-cocycle twist} of the descent datum.

If $\mathcal C$ is braided, then $U$ is abelian. If the commutator bicharacter
\[
 b_\alpha(g,h)
 =
 \frac{\alpha(g,h)}{\alpha(h,g)}
\]
is nontrivial, then $G_\alpha$ is a nontrivial soft tensor autoequivalence.
Consequently, $\Phi_\alpha$ and $\Phi$ determine distinct sections of
$p_{\mathrm{sf}}$.
\end{proposition}

\begin{proof}
The cocycle identity for $\alpha$ is exactly the tensor-functor coherence for
$G_\alpha$. Since $\Phi$ fixes every simple isomorphism class, it acts trivially
on the universal grading group. The tensorator of $\Phi_\alpha$ differs from
that of $\Phi$ on objects of degrees $g,h$ by the scalar $\alpha(g,h)$. On
squaring, the two additional factors are $\alpha(g,h)$ and its complex
conjugate. They multiply to $1$ because $\alpha$ takes values in $\{\pm1\}$.
Thus the tensorator of $\Phi_\alpha^2$ agrees with that of $\Phi^2$, so the
natural isomorphism
$\mu:\Phi^2\Rightarrow\id_{\mathcal C}$ is also monoidal for
$\Phi_\alpha^2$. The underlying functor of $G_\alpha$ is the identity on
objects and morphisms; hence
\[
 f_X\circ\Phi_\alpha(f_X)
 =
 f_X\circ\Phi(f_X)
 =
 \mu_X.
\]
Thus the original frames remain compatible and the new datum is split.

Suppose now that $\mathcal C$ is braided and that there were a monoidal natural
isomorphism
\[
 \eta:G_\alpha\Longrightarrow\id_{\mathcal C}.
\]
For homogeneous simple objects $X,Y$, monoidality gives
\[
 \eta_{X\otimes Y}\,
 \alpha(|X|,|Y|)
 =
 \eta_X\otimes\eta_Y.
\]
Applying the same equation to $Y,X$ and using naturality of $\eta$ with respect
to the braiding $c_{X,Y}:X\otimes Y\to Y\otimes X$ gives
\[
 \alpha(|X|,|Y|)
 =
 \alpha(|Y|,|X|).
\]
The grading is faithful, so this holds for every pair of elements of $U$,
contradicting the nontriviality of $b_\alpha$. Therefore $G_\alpha$ is
nontrivial.
\end{proof}

\begin{corollary}\label{cor:even-D-two-real-sections}
Let
\[
 \mathcal C=\mathcal C(D_{2r},k),
 \qquad r\geq2.
\]
Then the bar split real structure admits a nontrivial grading-cocycle
twist that gives a second simple-fixing split real section.

More explicitly, identify
\[
 U(\mathcal C)\cong
 C_2\times C_2\cong\F_2^2
\]
and write $a=(a_1,a_2)$. The formula
\begin{equation}\label{eq:even-D-real-gauge-cocycle}
 \alpha(a,b)=(-1)^{a_2b_1}
\end{equation}
defines a normalized $2$-cocycle. Its commutator satisfies
\[
 b_\alpha\bigl((1,0),(0,1)\bigr)=-1.
\]
Hence $G_\alpha$ is nontrivial and
\[
 \Phi_{\mathrm{bar}}
 \qquad\text{and}\qquad
 G_\alpha\circ\Phi_{\mathrm{bar}}
\]
define distinct simple-fixing sections of $p_{\mathrm{sf}}$, both of which lift
to split framed sections.
\end{corollary}

\begin{proof}
The universal grading computation follows from
\cite[Theorem~5.3 and Table~7]{GalindoMoraRowellVerlinde}.
The cocycle in~\eqref{eq:even-D-real-gauge-cocycle} is bilinear and hence
satisfies the cocycle identity. Its commutator is nontrivial, so
Proposition~\ref{prop:grading-cocycle-twist} applies.
\end{proof}

This phenomenon is special to the even $D$-series among the Verlinde categories.
Their universal grading group is the Klein four group, and
\[
 H^2(C_2\times C_2,\C^\times)\cong C_2.
\]
For the other simple Lie types the universal grading group is cyclic or trivial,
and hence
\[
 H^2(U(\mathcal C),\C^\times)=0.
\]
Thus no nontrivial soft tensor autoequivalence can be obtained in those types
merely by twisting the tensorator through the universal grading.

\begin{remark}
\label{rem:type-D-gauge-ambiguity}
The paragraph following
\cite[Theorem~9.1]{EdieMichellGraded} states that
\[
 \operatorname{Gauge}(\mathcal C(D_{2r+1},k))=1,
 \qquad
 \operatorname{Gauge}(\mathcal C(D_{2r},k))\cong C_2,
\]
but explicitly leaves the proof to a future publication. We therefore do not
use these statements as established results.

The preceding construction proves independently that the even $D$-series
contains the nontrivial soft class $[G_\alpha]$. If the announced calculation of
the full soft tensor autoequivalence group is established, this class generates
the entire group in even type $D$. It remains open here whether the real
categories obtained from $\Phi_{\mathrm{bar}}$ and
$G_\alpha\Phi_{\mathrm{bar}}$ are equivalent after forgetting the chosen
identifications of their complexified simple objects.

For odd $D$-type and for the exceptional types, the grading-cocycle
ambiguity is absent, but a complete calculation of all soft tensor
autoequivalences, or a complete uniqueness statement for split real forms, is
not supplied here.
\end{remark}

\subsection{Real forms under zesting}\label{subsec:Cgk-zesting}

We now study how changing the associator by zesting affects real descent. We
recall only the part of the construction needed here; see
\cite[Sections~2--3]{BraidedZesting} for the general obstruction theory and
\cite{GalindoMoraRowellVerlinde} for the categories
$\mathcal C(\mathfrak g,k)$.

Let $A$ be a finite abelian group and let
$\mathcal C=\bigoplus_{a\in A}\mathcal C_a$ be a faithfully $A$-graded braided
fusion category. The part of an associative zesting datum relevant here consists
of a normalized function
\[
 \lambda:A\times A\longrightarrow\operatorname{Inv}(\mathcal C_0),
 \qquad \lambda(0,a)=\lambda(a,0)=\mathbf 1,
\]
together with coherent identifications between
$\lambda(a,b)\otimes\lambda(a+b,c)$ and
$\lambda(b,c)\otimes\lambda(a,b+c)$ satisfying the pentagon. For
homogeneous objects the zested tensor product is
\begin{equation}\label{eq:zested-tensor-product}
 X_a\otimes_\lambda Y_b
 =X_a\otimes Y_b\otimes\lambda(a,b).
\end{equation}
Its associator combines the original associator, the half-braiding used to move
$\lambda(a,b)$ past the third factor, and the coherence isomorphisms of the
zesting. A braided zesting is an associative zesting equipped with additional
data satisfying two hexagon identities. When the object-valued part is trivial,
that is, $\lambda(a,b)=\mathbf 1$ for all $a,b$, an ordinary normalized cocycle
$\omega\in Z^3(A,\C^\times)$ gives the cohomological zesting
$\mathcal C^\omega$, with unchanged tensor product and associator
\begin{equation}\label{eq:cohomological-zesting-associator}
 \alpha^\omega_{X_a,Y_b,Z_c}
 =\omega(a,b,c)\alpha_{X_a,Y_b,Z_c}.
\end{equation}

The following observation shows that arbitrary associative zesting need not
preserve real descent.

\begin{proposition}\label{prop:cohomological-zesting-real-obstruction}
Let $\mathcal C=\bigoplus_{a\in A}\mathcal C_a$ be a faithfully graded complex
fusion category with a split real form whose associated semilinear involution
fixes every simple isomorphism class and preserves the $A$-grading. Let
$R:{}^\gamma\mathcal C\to\mathcal C$ be the corresponding twisted tensor
equivalence. Suppose that a real form of $\mathcal C^\omega$, after using $R$ to
identify ${}^\gamma(\mathcal C^\omega)$ with
$\mathcal C^{\overline\omega}$, gives a tensor equivalence
\[
 F:\mathcal C^{\overline\omega}\longrightarrow\mathcal C^\omega
\]
obtained from a tensor autoequivalence $T$ of $\mathcal C$ by using the same
underlying functor and tensorator. If $T$ induces $u\in\Aut(A)$, then
\begin{equation}\label{eq:zesting-arbitrary-real-necessary}
 u^*[\omega]=-[\omega].
\end{equation}
In particular, if $u=\id_A$, then $2[\omega]=0$.
\end{proposition}

\begin{proof}
The equivalence $R$ identifies ${}^\gamma(\mathcal C^\omega)$ with
$\mathcal C^{\overline\omega}$, where
$[\overline\omega]=-[\omega]$. Compare the monoidal coherence equation for
$F:\mathcal C^{\overline\omega}\to\mathcal C^\omega$ with the one for
$T:\mathcal C\to\mathcal C$. On homogeneous objects of degrees $a,b,c$, the
contributions from the original associator and from the common tensorator cancel.
The remaining scalar equation identifies $\overline\omega(a,b,c)$ with
$\omega(u(a),u(b),u(c))$, up to a coboundary. Hence
$[\overline\omega]=u^*[\omega]$, which is
\eqref{eq:zesting-arbitrary-real-necessary}.
\end{proof}

\begin{remark}\label{rem:zesting-obstruction-hypothesis}
The hypothesis relating $F$ to a tensor autoequivalence of $\mathcal C$ is
essential. The $H^3(A,\C^\times)$-torsor describes associators only after the
graded components and their system of tensor products have been fixed. A general
equivalence of graded extensions may also change these data through an
automorphism of $A$ and an autoequivalence of the neutral component; see
\cite{EdieMichellGraded}. Thus the conclusion of
Proposition~\ref{prop:cohomological-zesting-real-obstruction} does not follow from
the induced automorphism of $A$ alone.
\end{remark}

We apply this to type $A$. The universal grading group of
$\mathcal C(\mathfrak{sl}_N,k)$ is
$\mathbb Z/N\mathbb Z$~\cite[Section~5.1]{GalindoMoraRowellVerlinde}, and
\[
 H^3(\mathbb Z/N\mathbb Z,\C^\times)\cong\mathbb Z/N\mathbb Z.
\]
For $t\in\mathbb Z/N\mathbb Z$, choose a cocycle $\omega_t$ representing $t$;
for instance, on representatives $0\leq a,b,c<N$, one may take
\begin{equation}\label{eq:cyclic-zesting-cocycle}
 \omega_t(a,b,c)
 =\exp\left(
   \frac{2\pi i\,ta}{N}
   \left\lfloor\frac{b+c}{N}\right\rfloor
 \right).
\end{equation}

\begin{lemma}\label{lem:type-A-fixed-label-cocycle-rigidity}
Let $N\geq3$ and $k\geq2$. If a tensor equivalence
\[
 F:\mathcal C(\mathfrak{sl}_N,k)^{\omega_s}
 \longrightarrow
 \mathcal C(\mathfrak{sl}_N,k)^{\omega_t}
\]
induces the identity on the based Grothendieck ring, then
$s=t$ in $H^3(\mathbb Z/N\mathbb Z,\C^\times)$.
\end{lemma}

\begin{proof}
Put $\ell=N+k$ and $q=\exp(\pi i/\ell)$. The type-$A$ reconstruction theorem
in~\cite[Theorem~25.4]{CiamproneGiannonePinzari} associates to a semisimple
rigid tensor category with the $\mathfrak{sl}_N$ fusion rules and a fixed
based-ring identification a pair of invariants
\[
 (q_{\mathcal C},\tau_{\mathcal C}),
\]
which determines the tensor-equivalence class of the pair, up to simultaneous
inversion. For $\mathcal C(\mathfrak{sl}_N,k)$ one may take
$q_{\mathcal C}=q^2$. The cocycle modification in
\cite[equation~(25.1)]{CiamproneGiannonePinzari}, as its parameter ranges over
the $N$th roots of unity, gives the $N$ cohomological twists of the standard
category. Choose an identification compatible with these twists,
\[
 H^3(\mathbb Z/N\mathbb Z,\C^\times)
 \xrightarrow{\sim}\mu_N,
 \qquad
 [\omega_r]\longmapsto w_r.
\]
Under this identification, the same theorem gives
\[
 \tau_{\mathcal C^{\omega_r}}=w_r^{-1}\tau_{\mathcal C}.
\]
Thus the assignment
\[
 r\longmapsto\tau_{\mathcal C^{\omega_r}}
\]
is injective.

The only ambiguity in the classification is the simultaneous replacement
\[
 (q_{\mathcal C},\tau_{\mathcal C})
 \longmapsto
 (q_{\mathcal C}^{-1},\tau_{\mathcal C}^{-1}).
\]
Here $q_{\mathcal C}=q^2$ is a primitive $\ell$th root of unity. Since
$\ell=N+k>N+1$, one has
$q_{\mathcal C}\ne q_{\mathcal C}^{-1}$, so the simultaneous-inversion
alternative cannot identify two categories having this same value of
$q_{\mathcal C}$. An equivalence inducing the identity on the based ring
therefore forces
\[
 \tau_{\mathcal C^{\omega_s}}
 =\tau_{\mathcal C^{\omega_t}},
\]
and injectivity gives $s=t$.
\end{proof}

\begin{corollary}\label{cor:slN-associative-zesting-no-real}
Let $N\geq3$ and $k\geq1$.
If $2t\ne0$ in $\mathbb Z/N\mathbb Z$, then
$\mathcal C(\mathfrak{sl}_N,k)^{\omega_t}$ has no split real form. If
\begin{equation}\label{eq:cyclic-zesting-no-real-condition}
 (1+u^2)t\ne0\qquad
 \text{for every }u\in(\mathbb Z/N\mathbb Z)^\times,
\end{equation}
then it has no real form at all. In particular, for every $k\geq1$ and every
nonzero $t\in\mathbb Z/3\mathbb Z$, the category
$\mathcal C(\mathfrak{sl}_3,k)^{\omega_t}$ has no real form.
\end{corollary}

\begin{proof}
Put $\mathcal C=\mathcal C(\mathfrak{sl}_N,k)$. The split real form of
Theorem~\ref{thm:Cgk-split-real} identifies the conjugate of
$\mathcal C^{\omega_t}$ with $\mathcal C^{\omega_{-t}}$. Suppose first that
$k\geq2$ and that
\[
 F:\mathcal C^{\omega_{-t}}\longrightarrow
 \mathcal C^{\omega_t}
\]
is a tensor equivalence. It induces an automorphism of the type-$A$ fusion ring.
The classification of these automorphisms and their monoidal lifts
\cite{GannonFusionAutomorphisms,EdieMichellAutoequivalences} provides a tensor
autoequivalence $T$ of $\mathcal C$ inducing the same permutation of simple
objects. Let $u\in(\mathbb Z/N\mathbb Z)^\times$ be its action on the universal
grading. Applying $T^{-1}$ to the target and then composing with $F$ gives an
equivalence
\[
 \mathcal C^{\omega_{-t}}\longrightarrow
 \mathcal C^{u^*\omega_t}
\]
which induces the identity on the based fusion ring. The type-$A$
fixed-label rigidity of
Lemma~\ref{lem:type-A-fixed-label-cocycle-rigidity} gives
\begin{equation}\label{eq:type-A-zesting-real-relation}
 -t=u^*t=u^2t\qquad\text{in }\mathbb Z/N\mathbb Z.
\end{equation}
Here the last equality follows by identifying
$H^3(\mathbb Z/N\mathbb Z,\C^\times)$ with
$H^4(\mathbb Z/N\mathbb Z,\mathbb Z)$. The latter is generated by the square of a
degree-two class, and $a\mapsto ua$ acts by $u^2$.

If the real form is split, then $F$ fixes every simple isomorphism class and
$u=1$. Equation~\eqref{eq:type-A-zesting-real-relation} gives $2t=0$, proving the
first assertion for $k\geq2$. For an arbitrary real form it gives
$(1+u^2)t=0$, proving the second assertion.

When $k=1$, the category $\mathcal C(\mathfrak{sl}_N,1)$ is pointed. Write
\[
 \beta\in H^3(\mathbb Z/N\mathbb Z,\C^\times)
\]
for its associator class. It has order at most two because the category has a
split real form. If a real equivalence for the twist induces $a\mapsto ua$, the
classification of pointed categories gives
\[
 u^2(\beta+t)=-(\beta+t).
\]
Every automorphism of the cyclic cohomology group fixes its subgroup of elements
of order at most two, and hence $u^2\beta=\beta=-\beta$. The displayed relation
therefore reduces to $u^2t=-t$, which is again
\eqref{eq:type-A-zesting-real-relation}. This proves both assertions at level
one.

Finally, for $N=3$ every unit has square one, and multiplication by two is
invertible on $\mathbb Z/3\mathbb Z$. Thus
condition~\eqref{eq:cyclic-zesting-no-real-condition} holds for both nonzero
classes.
\end{proof}

The situation changes when the zesting is required to be braided. The additional
hexagon equations force the potentially nonreal coherence scalar to cancel the
corresponding phase of the half-braiding.

\begin{lemma}\label{lem:bar-reverses-standard-braiding}
For finite-dimensional type~$1$ integrable
$U_\zeta(\mathfrak g)$-modules, the tensor-product bar reverses the standard
quantum-group braiding:
\begin{equation}\label{eq:bar-reverses-braiding}
 \Phi_2(N,M)^{-1}\circ\Phi(c_{M,N})\circ\Phi_2(M,N)
 =c_{\Phi(N),\Phi(M)}^{-1}.
\end{equation}
The same identity holds after restriction to tilting modules and passage to the
semisimplification.
\end{lemma}

\begin{proof}
Write $P_{M,N}:M\otimes N\to N\otimes M$ for the flip and
$\mathcal G_{M,N}$ for Lusztig's toral diagonal operator on $M\otimes N$.
With the coproduct convention~\eqref{eq:quantum-coproduct} and the
quasi-$R$-matrix convention of Lemma~\ref{lem:specialized-quasi-R}, the
standard commutativity isomorphism is
\begin{equation}\label{eq:standard-R-factorization}
 c_{M,N}
 =\Theta_{\zeta,N,M}\circ\mathcal G_{N,M}\circ P_{M,N};
\end{equation}
compare~\cite[Theorem~32.1.5]{LusztigQuantumGroups} and
\cite[Section~3.3]{BakalovKirillov}. The diagonal operators are compatible with
the flip, and complex conjugation inverts them:
\[
 \mathcal G_{N,M}^{-1}P_{M,N}
 =P_{M,N}\mathcal G_{M,N}^{-1},
 \qquad
 \overline{\mathcal G_{M,N}}=\mathcal G_{M,N}^{-1}.
\]
Choose weight bases of $M$ and $N$. Let $T_{M,N}$, $G_{M,N}$, and $C_{M,N}$
denote the matrices in these bases of $\Theta_{\zeta,M,N}$,
$\mathcal G_{M,N}$, and $c_{M,N}$, respectively. Thus
\[
 C_{M,N}=T_{N,M}G_{N,M}P_{M,N}.
\]
Relative to the conjugate bases, the matrix of the tensorator
\eqref{eq:bar-tensorator} is $\overline{T_{M,N}}$, while that of
$\Phi(c_{M,N})$ is $\overline{C_{M,N}}$. Hence the matrix of the left-hand side
of~\eqref{eq:bar-reverses-braiding} is
\begin{align*}
 \overline{T_{N,M}}^{-1}\,
 \overline{C_{M,N}}\,
 \overline{T_{M,N}}
 &=\overline{T_{N,M}}^{-1}\,
   \overline{T_{N,M}}G_{N,M}^{-1}
   P_{M,N}\overline{T_{M,N}}\\
 &=P_{M,N}G_{M,N}^{-1}\overline{T_{M,N}}.
\end{align*}
Let $T_{M,N}^{\Phi}$ be the matrix of $\Theta_\zeta$ on
$\Phi(M)\otimes\Phi(N)$. By the definition of the bar-twisted module action and
\eqref{eq:bar-quasi-R-inverse},
\[
 T_{M,N}^{\Phi}
 =\overline{T_{M,N}^{-1}}.
\]
The weights of the bar-twisted modules are unchanged, so the toral diagonal
operator on $\Phi(M)\otimes\Phi(N)$ has matrix $G_{M,N}$. Therefore the inverse
of the braiding with the objects interchanged has matrix
\[
 P_{M,N}G_{M,N}^{-1}(T_{M,N}^{\Phi})^{-1}
 =P_{M,N}G_{M,N}^{-1}\overline{T_{M,N}},
\]
which is exactly the matrix obtained above. This proves the identity on the
integrable module category. All the operators involved
preserve tilting modules, and the braiding and tensorator preserve the negligible
ideal, so the identity descends to the semisimplification.
\end{proof}

\begin{lemma}\label{lem:zested-associator-real-factorization}
Let $\mathcal C$ be a braided fusion category with a split real form, and choose
fusion bases obtained by scalar extension from that real form. Let
$(\lambda,\omega)$ be an associative zesting whose relative half-braiding is the
reverse half-braiding, as in
\cite[Definition~4.1]{BraidedZesting}. Choose nonzero real vectors in the
one-dimensional fusion spaces between tensor products of the invertible objects
appearing in $\lambda$.

For simple homogeneous objects $X_r,Y_s,Z_t$, transport the original fusion
bases to the zested fusion spaces through tensor products with
$\lambda(r,s)$. In these bases, the matrix of the zested associator is a real
matrix multiplied by
\begin{equation}\label{eq:zested-associator-scalar-factor}
 z(r,s,t)=\omega(r,s,t)\,
 b\bigl(Z_t,\lambda(r,s)\bigr)^{-1},
\end{equation}
where $b(Z_t,\lambda(r,s))$ is the coefficient of
$c_{Z_t,\lambda(r,s)}$ in the chosen one-dimensional fusion bases.
Consequently, the zested $F$-matrices are real whenever all the scalars
$z(r,s,t)$ are real.
\end{lemma}

\begin{proof}
The defining diagram for the zested associator
\cite[equation~(3.4) and Figure~3]{BraidedZesting} consists of the original
associators, the relative half-braiding that moves $\lambda(r,s)$ past $Z_t$,
and the coherence isomorphism
\[
 \lambda(r,s)\otimes\lambda(r+s,t)
 \longrightarrow
 \lambda(s,t)\otimes\lambda(r,s+t).
\]
The original associators have real matrices in the chosen fusion bases. Tensoring
with invertible objects, rebracketing, and the identifications between their
one-dimensional fusion spaces also have real matrices when the chosen real
vectors are used. Relative to a real generator of the last displayed Hom space,
the zesting coherence isomorphism has coefficient $\omega(r,s,t)$. Changing
these real generators only multiplies $\omega$ by a nonzero real coboundary and
does not affect either
$\omega/\overline\omega$ or the reality conclusion.
The correspondence between pre-metric-group zestings and categorical zestings is
noncanonical precisely because it involves these choices
\cite[Remark~3.4]{GalindoMoraRowellVerlinde}; we choose it using the real
generators above. Any other choice gives an equivalent zested category, and the
existence of a split real form is invariant under tensor equivalence.

By the convention for the relative half-braiding, its remaining coefficient is
that of
\[
 c'_{\lambda(r,s),Z_t}
 =c_{Z_t,\lambda(r,s)}^{-1},
\]
namely $b(Z_t,\lambda(r,s))^{-1}$. Thus the complete defining diagram, not only
a selected part of it, has the scalar factor
\eqref{eq:zested-associator-scalar-factor}; every other matrix in the composition
is real. This proves the assertion.
\end{proof}

\begin{proposition}\label{prop:braided-zested-Cgk-real}
For every complex simple Lie algebra $\mathfrak g$ and every $k\geq1$, each
braided zesting of $\mathcal C(\mathfrak g,k)$ classified in
\cite{GalindoMoraRowellVerlinde} has a split real form as an underlying fusion
category.
\end{proposition}

\begin{proof}
By Corollary~\ref{cor:Cgk-real-orthogonal-F}, choose fusion bases for
$\mathcal C=\mathcal C(\mathfrak g,k)$ in which the $F$-matrices are real. Let
$(\Phi,\Phi_2)$ be the tensor involution constructed in
Theorem~\ref{thm:Cgk-split-real}, transported through the compatible frames on
the simple objects. By
Lemma~\ref{lem:bar-reverses-standard-braiding}, $\Phi$ reverses the braiding.

We record the consequence needed below. Let $h$ be an invertible object in the
neutral component and let $X_t$ be a simple object of universal degree $t$.
Choose real unit vectors in the one-dimensional fusion spaces involving $h$ and
$X_t$, and denote by $b(h,X_t)$ and $b(X_t,h)$ the corresponding coefficients of
$c_{h,X_t}$ and $c_{X_t,h}$. Equation~\eqref{eq:bar-reverses-braiding} gives
\begin{equation}\label{eq:bar-half-braiding-coefficient}
 \overline{b(h,X_t)}=b(X_t,h)^{-1}.
\end{equation}
The monodromy is the evaluation of the universal degree, hence
\begin{equation}\label{eq:half-braiding-monodromy}
 b(X_t,h)b(h,X_t)=\chi_t(h),
\end{equation}
where $\chi_t$ is the character of the group of invertible objects corresponding
to $t$.

Lemma~\ref{lem:zested-associator-real-factorization} identifies the complete
nonreal scalar in the zested associator as
\begin{equation}\label{eq:zested-associator-potentially-nonreal}
 z(r,s,t)=\omega(r,s,t)\,
 b\bigl(X_t,\lambda(r,s)\bigr)^{-1}.
\end{equation}
Using~\eqref{eq:bar-half-braiding-coefficient} with the arguments interchanged,
together with~\eqref{eq:half-braiding-monodromy}, gives
\begin{equation}\label{eq:zested-associator-reality-ratio}
 \frac{z(r,s,t)}{\overline{z(r,s,t)}}
 =\frac{\omega(r,s,t)/\overline{\omega(r,s,t)}}
        {\chi_t(\lambda(r,s))}.
\end{equation}
We verify from the explicit data of
\cite[Sections~4--5]{GalindoMoraRowellVerlinde} that the right-hand side is one.

If the universal grading group is trivial, there is nothing to prove. We first
treat the factorized cases, in which the group of invertible objects in the
neutral component is trivial. In these cases the braided zestings are represented
by abelian three-cocycles
\cite[Theorem~5.4 and the discussion preceding
Section~5.5]{GalindoMoraRowellVerlinde}.

Suppose first that the universal grading group is
$A=\mathbb Z/N\mathbb Z$. For representatives $0\leq u,v,w<N$, the cyclic
abelian three-cocycles of
\cite[equation~(3.3)]{GalindoMoraRowellVerlinde} may be chosen with ordinary
associator component
\[
 \omega_\eta(u,v,w)
 =\eta^{\,uN\left\lfloor(v+w)/N\right\rfloor},
 \qquad
 \eta^{N^2}=\eta^{2N}=1.
\]
Since $\eta^N\in\{\pm1\}$, the cocycle $\omega_\eta$ takes values in
$\{\pm1\}$.

If $A=\mathbb Z/2\mathbb Z\times\mathbb Z/2\mathbb Z$, the representatives of
\cite[Proposition~3.1]{GalindoMoraRowellVerlinde} are indexed by
\[
 (a,b,c)\in
 \mathbb Z/4\mathbb Z\times\mathbb Z/4\mathbb Z
 \times\mathbb Z/2\mathbb Z.
\]
Their ordinary associator component is
\[
 \omega_{a,b,c}(x,y,z)
 =(-1)^{a x_1y_1z_1+b x_2y_2z_2}.
\]
Although this component depends only on the parities of $a$ and $b$, the full
abelian three-cocycle also records $c$ and the classes of $a,b$ modulo~$4$ in
its braiding component.
Thus the associator corrections in all factorized cases are real. It remains to
treat the nonfactorized cases, for which the neutral component contains
nontrivial invertible objects.

Suppose next that the universal grading group is cyclic of order $N$ and that
the pointed subcategory is degenerate. Let the group of invertible objects in the
neutral component be cyclic of order $m$, generated by $h$, and let $(a,b)$ be
the parameters of the cyclic associative zesting. Write
$\theta_{h^a}=(-1)^{\varepsilon_a}$, where
$\theta$ is the ribbon twist, $\varepsilon_a\in\{0,1\}$, and put
$\xi=e^{\pi i/N}$ and $q=\xi^2$. For representatives $0\leq r,s,t<N$, set
$\kappa(r,s)=\lfloor(r+s)/N\rfloor$. The cyclic formulas give
\[
 \lambda(r,s)=h^{a\kappa(r,s)},
 \qquad
 \frac{\omega(r,s,t)}{\overline{\omega(r,s,t)}}
 =q^{(\varepsilon_a+2b)\kappa(r,s)t}.
\]
Under the standard identification of the universal grading group with
$\mathbb Z/N\mathbb Z$, one has $\chi_t(h)=q^{Nt/m}$ and hence
\[
 \chi_t(\lambda(r,s))=q^{aN\kappa(r,s)t/m}.
\]
The criterion for this associative zesting to admit a braiding is
\begin{equation}\label{eq:cyclic-braided-zesting-condition}
 \frac{aN}{m}\equiv\varepsilon_a+2b\pmod N
\end{equation}
\cite[equations~(5.7)--(5.8), Theorem~5.5, and
Table~10]{GalindoMoraRowellVerlinde}.
Thus the numerator and denominator in
\eqref{eq:zested-associator-reality-ratio} agree, and $z(r,s,t)$ is real.

The only noncyclic universal grading that occurs in the classification is
$\mathbb Z/2\mathbb Z\times\mathbb Z/2\mathbb Z$ in type $D_{2n}$.
The factorized cases, namely the odd levels, were treated above. At even level,
Corollary~5.12 of~\cite{GalindoMoraRowellVerlinde} applies. If $n$ is even, the
pointed subcategory is Tannakian, and the associative parts of the braided
zestings are the abelian three-cocycles of Proposition~3.1, whose ordinary
components take values in $\{\pm1\}$. If $n$ is odd, the pointed quadratic form is
$q(x_1,x_2)=(-1)^{x_1^2+x_2^2}$. The data of
\cite[equations~(4.9)--(4.11) and Theorem~4.6]{GalindoMoraRowellVerlinde}
are parametrized by $r,s\in\mathbb Z/4\mathbb Z$. Write $\bar r,\bar s$ for
their images in $\mathbb Z/2\mathbb Z$. For
$x,y,z\in(\mathbb Z/2\mathbb Z)^2$, their object-valued and scalar parts are
\[
 \lambda(x,y)=e_1^{\bar r x_1y_1}e_2^{\bar s x_2y_2},
 \qquad
 \omega(x,y,z)=i^{r x_1y_1z_1+s x_2y_2z_2},
\]
up to real factors, where $e_1,e_2$ are the two fermionic generators. Therefore
\[
 \frac{\omega(x,y,z)}{\overline{\omega(x,y,z)}}
 =(-1)^{r x_1y_1z_1+s x_2y_2z_2}.
\]
The degree $z$ evaluates on the two generators by
$\chi_z(e_j)=(-1)^{z_j}$. Consequently,
\[
 \chi_z(\lambda(x,y))
 =(-1)^{\bar r x_1y_1z_1+\bar s x_2y_2z_2}.
\]
These expressions agree because $r\equiv\bar r$ and
$s\equiv\bar s\pmod2$. Equation~\eqref{eq:zested-associator-reality-ratio}
again shows that the total associator correction is real.

In every case, the scalar~\eqref{eq:zested-associator-potentially-nonreal} is
therefore real. Lemma~\ref{lem:zested-associator-real-factorization} shows that
all zested $F$-matrices are real in the transported fusion bases.
Proposition~\ref{prop:dictionary} gives a split real form.
\end{proof}

\begin{remark}\label{rem:zested-braiding-not-real}
Proposition~\ref{prop:braided-zested-Cgk-real} concerns only the underlying
fusion category and makes no assertion about descent of the zested braiding or
ribbon structure. For example, in the fermionic zesting of
$\mathcal C(\mathfrak{sl}_2,k)$ with
$k\equiv2\pmod4$, one has $\lambda(1,1)=f$, where $f$ is the nontrivial
invertible object, together with $\omega(1,1,1)=i$ and
$t(1,1)=e^{\pi i/4}$. The
half-braiding of $f$ supplies the complementary imaginary factor; hence the total
associator correction is real. On the other hand, let $X_1$ be the simple object
with label~$1$ in the odd component. The explicit zested modular data give
\cite[Example~5.7 and equation~(5.15)]{GalindoMoraRowellVerlinde}
\[
 \theta_{X_1}^{\lambda}
 =e^{-\pi i/4}\theta_{X_1}
 =\exp\left[
   \pi i\left(\frac{3}{2(k+2)}-\frac14\right)
 \right].
\]
Here we used the convention
$T_{X,X}=\theta_X^{-1}$. Equation~(5.15) multiplies the $T$-matrix entry by
$e^{\pi i/4}$ and therefore multiplies the ribbon twist by its inverse.
If $k=4m+2$, with $m\geq0$, the exponent is
\[
 \frac{1-2m}{8(m+1)},
\]
which is not an integer because its numerator is odd and has absolute value
strictly smaller than its denominator. Thus
$\theta_{X_1}^{\lambda}\notin\{\pm1\}$. The zested braided category is unitary
after choosing the one-dimensional coherence maps involving $f$ to be unitary.
The original associator and braiding are unitary, and the scalar zesting data have
absolute value one, so the inherited dagger makes the zested tensor constraints
and braiding unitary. By
Proposition~\ref{prop:unitary-braided-real-twist}, the zested braided structure,
and therefore its ribbon structure, has no split real form.
\end{remark}

\section{Real descent for representation categories of finite groups}
\label{sec:groups}

Let $G$ be a finite group. We make the criterion of
Section~\ref{sec:descent} explicit for $\Rep_{\C}(G)$ when the semilinear functor
is induced by an automorphism of $G$.
Proposition~\ref{prop:automorphism-induced-real-descent} identifies the resulting
class $\Omega_{s_\alpha}$ with a vector of generalized twisted
Frobenius--Schur indicators.
In the involutive case with trivial coherence,
\cite[Appendix~A]{BHPReality} independently identifies the
Kawanaka--Matsuyama indicators with the sign data governing flat charge
conjugation. Proposition~\ref{prop:automorphism-induced-real-descent} interprets
these signs as relative Brauer classes and also allows coherent implementers for
which $\alpha^2$ is inner, including the possible central corrections.

A \emph{split symmetric $K$-form} below means a split $K$-form whose
scalar-extension equivalence identifies its symmetry with the canonical symmetry on
$\Rep_{\C}(G)$, as in Subsection~\ref{subsec:structured-forms}.
Write $C_2=\Gal(\C/\R)=\{1,\gamma\}$.

Let $\alpha\in\Aut(G)$.  On
$\Rep_{\C}(G)$ consider the conjugate-linear strict tensor functor
\begin{equation}\label{eq:alpha-real-functor}
 \Phi_\alpha(X,\rho)
 =\bigl(\overline X,\overline\rho\circ\alpha^{-1}\bigr),
 \qquad \Phi_\alpha(f)=\overline f.
\end{equation}
Complex conjugation and precomposition with $\alpha^{-1}$ leave the symmetry maps
unchanged. Hence
\[
 \Phi_\alpha(c_{X,Y})
 =c_{\Phi_\alpha(X),\Phi_\alpha(Y)}.
\]
It fixes every simple isomorphism class precisely when $\alpha$ is
\emph{class-inverting}, meaning that $\alpha(g)$ is conjugate to $g^{-1}$ for every
$g\in G$.  Indeed, the character of~\eqref{eq:alpha-real-functor} is
$g\mapsto\overline{\chi(\alpha^{-1}(g))}$, and irreducible characters separate
conjugacy classes.

Assume henceforth that $\alpha$ is class-inverting. Suppose that
\begin{equation}\label{eq:alpha-square-inner}
 \alpha^2=\operatorname{Ad}_a
 \qquad(a\in G).
\end{equation}
Here $\operatorname{Ad}_a(g)=aga^{-1}$.
Then
\begin{equation}\label{eq:alpha-coherence-map}
 \mu_X=\rho(a):\Phi_\alpha^2(X)\longrightarrow X
\end{equation}
is a monoidal natural isomorphism. Consequently,
\[
 [(\Phi_\alpha,\gamma)]
 \in\Pi_{\mathrm{sf}}(\Rep_{\C}(G)/\R)
\]
is an involution and determines the section
\[
 s_\alpha:C_2\longrightarrow
 \Pi_{\mathrm{sf}}(\Rep_{\C}(G)/\R),
 \qquad
 1\longmapsto1,
 \quad
 \gamma\longmapsto[(\Phi_\alpha,\gamma)].
\]
The $C_2$-coherence identity for $\mu$ is
equivalent to
\begin{equation}\label{eq:alpha-fixes-a}
 \alpha(a)=a.
\end{equation}
These conditions have a useful intrinsic form.  If~\eqref{eq:alpha-square-inner}
holds, then $\alpha(a)a^{-1}\in Z(G)$.
Since a class-inverting automorphism acts by inversion on $Z(G)$, replacing $a$ by
$az$, with $z\in Z(G)$, replaces this element by
$\alpha(a)a^{-1}z^{-2}$. By Tannakian
reconstruction, every monoidal natural isomorphism
$\Phi_\alpha^2\Rightarrow\id$ has component $\mu_X=\rho(b)$ at $(X,\rho)$, where
$\alpha^2=\operatorname{Ad}_b$. Consequently,
$\Phi_\alpha$ admits a coherent $C_2$-structure if and only if
$\alpha(a)a^{-1}=z^2$ for some $z\in Z(G)$.
Once one coherent implementer $a$ has been chosen, all the others are $az$ with
$z\in Z(G)[2]$, where $Z(G)[2]=\{z\in Z(G)\mid z^2=1\}$.

A coherent pair $(\alpha,a)$ determines the group
\begin{equation}\label{eq:index-two-extension}
 \widetilde G_{\alpha,a}
 =\left\langle G,t\ \middle|\
 tgt^{-1}=\alpha(g),\ t^2=a\right\rangle.
\end{equation}
The relations define a group on the set $G\sqcup Gt$ exactly because
$\alpha^2=\operatorname{Ad}_a$ and $\alpha(a)=a$. For
$\chi\in\operatorname{Irr}(G)$ put
\begin{equation}\label{eq:generalized-twisted-indicator}
 \nu_{\alpha,a}(\chi)
 =\frac{1}{|G|}\sum_{g\in G}\chi\bigl(g\alpha(g)a\bigr)
 =\frac{1}{|G|}\sum_{x\in\widetilde G_{\alpha,a}\setminus G}\chi(x^2).
\end{equation}
This is the index-two-extension form of the twisted Frobenius--Schur indicator of
\cite[Section~2, especially Theorem~2.5]{KawanakaMatsuyama}. The next proposition
identifies it with the relative Brauer signs and hence with the obstruction of
Section~\ref{sec:descent}.

For $z\in Z(G)[2]$ and $\chi\in\operatorname{Irr}(G)$, the scalar by which $z$
acts on the irreducible representation with character $\chi$ is
$\chi(z)/\chi(1)$. Define
\[
 \operatorname{ev}:Z(G)[2]
 \longrightarrow\{\pm1\}^{\operatorname{Irr}(G)},
 \qquad z\longmapsto\left(\frac{\chi(z)}{\chi(1)}\right)_\chi.
\]
Tannakian reconstruction identifies
$\Aut_\otimes(\id_{\Rep_{\C}(G)})$ with $Z(G)$; hence
$\operatorname{ev}(Z(G)[2])$ is precisely the subgroup of monoidal sign vectors.
Put
\[
 \boldsymbol\nu_{\alpha,a}
 =\bigl(\nu_{\alpha,a}(\chi)\bigr)_{\chi\in\operatorname{Irr}(G)}.
\]

\begin{proposition}\label{prop:automorphism-induced-real-descent}
Let $\alpha$ be class-inverting and let $(\alpha,a)$ satisfy
\eqref{eq:alpha-square-inner} and~\eqref{eq:alpha-fixes-a}. Then
for every $\chi\in\operatorname{Irr}(G)$, the number
$\nu_{\alpha,a}(\chi)\in\{1,-1\}$ is the relative Brauer sign of the corresponding
simple representation, and
\[
 [\boldsymbol\nu_{\alpha,a}]
 \in
 \frac{\{\pm1\}^{\operatorname{Irr}(G)}}
 {\operatorname{ev}(Z(G)[2])}
\]
maps to $\Omega_{s_\alpha}(\Rep_{\C}(G))$ under the injection in
Corollary~\ref{cor:real-signs}. The fixed coherent datum is split precisely when
$\boldsymbol\nu_{\alpha,a}=1$. Some coherent implementer gives a split symmetric
real form precisely when there is $z\in Z(G)[2]$ such that
\begin{equation}\label{eq:central-indicator-correction}
 \frac{\chi(z)}{\chi(1)}\nu_{\alpha,a}(\chi)=1
 \qquad\bigl(\chi\in\operatorname{Irr}(G)\bigr).
\end{equation}
\end{proposition}

\begin{proof}
Let $(X,\rho)$ be irreducible.  Since $\Phi_\alpha(X)\cong X$, there is an
antilinear map $T_X:X\to X$ satisfying
\begin{equation}\label{eq:twisted-real-operator}
 T_X\rho(g)T_X^{-1}=\rho(\alpha(g)).
\end{equation}
The operator $\rho(a)^{-1}T_X^2$ commutes with $\rho(G)$ and is therefore scalar.
Moreover, $T_X$ commutes with $\rho(a)$ because $\alpha(a)=a$; conjugating that scalar
by $T_X$ shows that it is real. Choose a $G$-invariant Hermitian form and rescale
$T_X$ to be antiunitary. Antiunitarity forces the scalar to have absolute value one.
Hence there is a uniquely determined sign
$\varepsilon_X\in\{1,-1\}$ such that
\begin{equation}\label{eq:indicator-sign-equation}
 T_X^2=\varepsilon_X\rho(a).
\end{equation}
In the stable-object construction,~\eqref{eq:indicator-sign-equation} is precisely
the relative Brauer sign.

It remains to verify that~\eqref{eq:generalized-twisted-indicator} computes this
sign. In an orthonormal basis write $T_X=C\circ\mathrm{conj}$. Then $C$ is unitary,
\[
 \rho(\alpha(g))=C\overline{\rho(g)}C^{-1},
 \qquad C\overline C=\varepsilon_X\rho(a).
\]
Schur orthogonality, with $d=\dim X$, gives
\[
 \sum_{g\in G}\rho(g)C\overline{\rho(g)}
 =\frac{|G|}{d}C^{\mathsf T}.
\]
Since $C$ is unitary, the second displayed identity is equivalent to
\[
 C^{\mathsf T}C^{-1}=\varepsilon_X\rho(a)^{-1}.
\]
Taking traces now gives
\[
 \frac1{|G|}\sum_{g\in G}
 \chi\bigl(g\alpha(g)a\bigr)
 =\frac1d\operatorname{Tr}
   \bigl(C^{\mathsf T}C^{-1}\rho(a)\bigr)
 =\varepsilon_X.
\]
This proves the first assertions and the criterion for the fixed coherent datum.

The discussion preceding the proposition proves the assertions concerning the
existence and variation of coherent implementers.  Finally, if $z\in Z(G)[2]$, then
\[
 \nu_{\alpha,az}(\chi)=\frac{\chi(z)}{\chi(1)}\nu_{\alpha,a}(\chi).
\]
Corollary~\ref{cor:real-signs},
together with the structured criterion of
Subsection~\ref{subsec:structured-forms}, therefore identifies the class of
$\boldsymbol\nu_{\alpha,a}$ modulo $\operatorname{ev}(Z(G)[2])$ with
$\Omega_{s_\alpha}(\Rep_{\C}(G))$ and gives
\eqref{eq:central-indicator-correction}.
\end{proof}

Every braided tensor autoequivalence of $\Rep_{\C}(G)$, equivalently every
symmetric tensor autoequivalence,
is, up to monoidal isomorphism, induced by an automorphism of $G$
\cite[Section~2.2, (3)]{DavydovDualizing}. Comparing a semilinear symmetric
equivalence with coefficientwise conjugation reduces it to such an
autoequivalence. Thus no additional cases occur in the symmetric problem.
For arbitrary tensor forms one may also have an \emph{invariant twist}, namely a
Drinfeld twist $J\in\C[G]^{\otimes2}$ commuting with every $\Delta(g)$ and changing
only the tensor structure of the identity functor
\cite[Introduction and Section~4]{DavydovTwistedGroup}. These twists are controlled
for irreducible Frobenius groups in Section~\ref{sec:frobenius-groups}.

If $G$ has odd order, then by
\cite[Proposition~3.1]{DavydovTwistedGroup}, a twisted automorphism decomposes
into an automorphism of $G$ and an invariant twist; the latter does not permute
the simple objects. Hence a split real form forces a class-inverting automorphism.
Such a group is abelian by
\cite[Theorem~2.14]{DavydovDualizing}. Conversely, for abelian $G$ the pointed
category $\Vect_{\operatorname{Hom}(G,\C^\times),\Q}$, with trivial associator and
its canonical symmetry, is a split symmetric
$\Q$-form of $\Rep_{\C}(G)$. Thus, for groups of odd order, split real forms exist
precisely in the abelian case and may then be chosen symmetric and rational.

\section{Orbit-rigid irreducible Frobenius groups}\label{sec:frobenius-groups}

We apply the preceding criterion to irreducible Frobenius groups. Under orbit
rigidity and in odd characteristic, Theorem~\ref{thm:orbit-rigid-frobenius-real}
reduces arbitrary split real descent to the symmetric case and detects it through
indicators of the complement. Cyclic complements and their minimal split fields are
treated separately in Section~\ref{sec:cyclic-frobenius}.

\subsection{Frobenius modules and tensor autoequivalences}

Let
\[
 G=V\rtimes H,
\]
where $V$ is a nonzero finite-dimensional $\F_p$-vector space and
$1\ne H\leq\operatorname{GL}_{\F_p}(V)$ acts irreducibly, with no nonidentity
element fixing a nonzero vector.
We call $G$ an \emph{irreducible Frobenius group}.  Necessarily $p\nmid|H|$;
otherwise an element of order $p$ in $H$ would act unipotently and would have a
nonzero fixed vector.

Put $\widehat V=\operatorname{Hom}(V,\C^\times)$, with
$(h\lambda)(v)=\lambda(h^{-1}v)$.  The action of $H$ on
$\widehat V\setminus\{1\}$ is free, and Clifford correspondence
\cite[Theorem~6.11]{Isaacs} gives
\begin{equation}\label{eq:general-frobenius-irreducibles}
 \operatorname{Irr}(G)
 =\operatorname{Irr}(H)\sqcup
 \left\{
  W_\lambda=\operatorname{Ind}_V^G(\lambda)
  \ \middle|\
  H\lambda\in H\backslash(\widehat V\setminus\{1\})
 \right\}.
\end{equation}
The first family is inflated from $H$, every $W_\lambda$ has degree $|H|$, and
\begin{equation}\label{eq:general-frobenius-rank}
 \operatorname{rank}\bigl(\Rep_{\C}(G)\bigr)
 =|\operatorname{Irr}(H)|+\frac{|V|-1}{|H|}.
\end{equation}

Every abelian normal subgroup of $G$ is contained in $V$.  Indeed, if an abelian
normal subgroup contains $(v,h)$ with $h\ne1$, then its commutators with $V$ fill
$V$, because $h-1$ is invertible; this contradicts commutativity.  Irreducibility
then shows that the only possibilities are $0$ and $V$. The center is also trivial:
\begin{equation}\label{eq:frobenius-center-trivial}
 Z(G)=1.
\end{equation}
Indeed, a central element has trivial component in $H$ by faithfulness and its
component in $V$ belongs to $V^H=0$. Hence
$\Aut_\otimes(\id_{\Rep_{\C}(G)})\cong Z(G)$ is trivial. In automorphism-induced
descent, the coherent structure therefore has no central modification, and its sign
vector admits no central correction.

The reduction theorem for twisted automorphisms of group algebras
\cite[Theorem~2.4]{DavydovTwistedGroup}, together with coprime cohomology, therefore
has the following consequence.

\begin{proposition}\label{prop:frobenius-autoequivalence-permutations}
Every tensor autoequivalence of $\Rep_{\C}(G)$ induces on $\operatorname{Irr}(G)$
the same permutation as pullback along an automorphism of $G$.  Consequently, if
$\Rep_{\C}(G)$ has a split real form, then $G$ has a class-inverting automorphism.
\end{proposition}

\begin{proof}
The correspondence with bi-Galois algebras
in~\cite[Corollary~6.2]{DavydovGalois}, followed by the reduction theorem
\cite[Theorem~2.4]{DavydovTwistedGroup}, shows that after a gauge transformation
the twist in a tensor autoequivalence is supported on an abelian normal subgroup. By the
preceding observation, the only nontrivial support is $V$. The alternation of the
corresponding cocycle is $H$-invariant by
\cite[Theorem~2.4]{DavydovTwistedGroup}. Since $\C^\times$ is
divisible, alternation gives the $H$-equivariant identification
\begin{equation}\label{eq:frobenius-alternation-cohomology}
 H^2(\widehat V,\C^\times)
 \cong\operatorname{Hom}(\textstyle\bigwedge^2\widehat V,\C^\times)
\end{equation}
\cite[Section~2]{DavydovTwistedGroup}. Thus the twist determines a class
\[
 [c]\in H^2(\widehat V,\C^\times)^H.
\]
The transgression construction in
\cite[equations~(1)--(2) and Lemma~3.7]{EGIso} places the
transgression of $[c]$ in
$H^2(H,\widehat{\widehat V})\cong H^2(H,V)$. This group vanishes because
$p\nmid|H|$~\cite[Corollary~III.10.2]{Brown}. We claim that $[c]$ has an
$H$-invariant cocycle representative.

Choose a normalized cocycle $c$ and normalized cochains $z_h$ such that
${}^h c/c=dz_h$. The cocycle defining the transgression is represented by
$z_{hk}/(z_h\,{}^h z_k)$. Since the transgression
vanishes, the $z_h$ may be multiplied by characters to ensure that
$z_{hk}=z_h\,{}^h z_k$.  Thus $z$ is a $1$-cocycle with values in
\[
 C^1_0(\widehat V,\C^\times)
 =\{u:\widehat V\to\C^\times\mid u(1)=1\}.
\]
The action of $H$ on $\widehat V\setminus\{1\}$ is free.  Consequently, this
module is a product of coinduced modules.  Shapiro's lemma
\cite[Chapter~III, Section~6]{Brown} gives
\[
 H^1(H,C^1_0(\widehat V,\C^\times))=0.
\]
Hence
$z_h={}^h u/u$ for a normalized cochain $u$, and replacing $c$ by $c/du$ makes it
$H$-invariant.

The corresponding twist $J\in\C[V]\otimes\C[V]$ commutes with $\Delta(g)$ for
every $g\in G$. Thus twisting leaves the coproduct unchanged, and the algebra
automorphism in the twisted automorphism is a Hopf automorphism, hence comes from
an automorphism of $G$. The remaining invariant twist fixes every object. This
proves the assertion about permutations.

If a split real form exists, its conjugate-linear descent functor fixes the simple
isomorphism classes. Composing it with coefficientwise conjugation gives a
$\C$-linear tensor autoequivalence inducing the duality permutation. The first
assertion supplies $\alpha\in\Aut(G)$ such that
$\chi(\alpha(g))=\chi(g^{-1})$ for every irreducible character.  Therefore
$\alpha$ is class-inverting.
\end{proof}

To express class inversion in terms of the Frobenius module, put
\[
 \mathcal N=N_{\operatorname{GL}_{\F_p}(V)}(H),
 \qquad \alpha_U^H=\operatorname{Ad}_U|_H
 \quad(U\in\mathcal N).
\]

\begin{proposition}\label{prop:frobenius-class-inverting-criterion}
The group $G=V\rtimes H$ has a class-inverting automorphism if and only if there is
$U\in\mathcal N$ such that
\begin{equation}\label{eq:frobenius-module-orbit-criterion}
 \alpha_U^H(h)\text{ is conjugate in }H\text{ to }h^{-1}\quad(h\in H),
 \qquad U(v)\in H\cdot(-v)\quad(v\in V).
\end{equation}
Such a map defines the class-inverting automorphism
\begin{equation}\label{eq:alpha-from-U}
 \alpha_U(vh)=U(v)\alpha_U^H(h).
\end{equation}
Consequently, the existence of a split real form implies the existence of such a
$U$.
\end{proposition}

\begin{proof}
The subgroup $V$ is characteristic. Every complement to $V$ is conjugate to $H$
because $H^1(H,V)=0$~\cite[Corollary~III.10.2]{Brown}. Thus, after composing an
automorphism with an inner
automorphism induced by $V$, we may assume that it preserves $H$.  Its restriction
to $V$ is then some $U\in\mathcal N$, and its restriction to $H$ is
$\alpha_U^H$.

The conjugacy class of $v\in V$ is $H\cdot v$.  If $h\ne1$, then $1-h$ is
invertible on $V$.  Hence every element of $Vh$ is conjugate by an element of $V$ to
$h$.  Moreover, two elements of $H$ are conjugate in $G$ exactly when they are
conjugate in $H$.  These observations identify class inversion on $V$ and on
$G\setminus V$ with the two conditions in
\eqref{eq:frobenius-module-orbit-criterion}, respectively.  They also prove the
converse and~\eqref{eq:alpha-from-U}.  The final assertion follows from
Proposition~\ref{prop:frobenius-autoequivalence-permutations}.
\end{proof}

\subsection{Orbit rigidity and indicators}

Coherence imposes an additional condition. The square of the class-inverting
automorphism must be inner, and the following group measures this defect.

Define the \emph{orbit-preserving normalizer} of the Frobenius complement by
\begin{equation}\label{eq:orbit-preserving-normalizer}
 N_{\mathrm{orb}}(H,V)
 =\{T\in\mathcal N\mid T(v)\in H\cdot v\text{ for every }v\in V\}.
\end{equation}
This is a subgroup of $\mathcal N$ containing $H$ as a normal subgroup. If
$S,T\in N_{\mathrm{orb}}(H,V)$,
normalization of $H$ shows that $ST$ again preserves every $H$-orbit, and the same
argument applied to $T^{-1}$ proves closure under inverses.

If $U$ satisfies~\eqref{eq:frobenius-module-orbit-criterion}, then
\begin{equation}\label{eq:U-square-in-orbit-preserving-normalizer}
 U^2\in N_{\mathrm{orb}}(H,V).
\end{equation}
Indeed, writing $U(v)=-h_v\cdot v$ with $h_v\in H$ gives
$U^2(v)=(\alpha_U^H(h_v)h_v)\cdot v$. Thus the coset $U^2H$ in
$N_{\mathrm{orb}}(H,V)/H$ measures whether the class-inverting automorphism
\eqref{eq:alpha-from-U} has inner square.  In particular, if
\begin{equation}\label{eq:orbit-rigidity}
 N_{\mathrm{orb}}(H,V)=H,
\end{equation}
then $a=U^2\in H$ and
\begin{equation}\label{eq:alpha-U-coherent-pair}
 \alpha_U^2=\operatorname{Ad}_a,
 \qquad \alpha_U(a)=a.
\end{equation}

There is a convenient sufficient condition for~\eqref{eq:orbit-rigidity}.  If
$T\in N_{\mathrm{orb}}(H,V)$, then
\[
 V=\bigcup_{h\in H}\ker(T-h).
\]
If $T\notin H$, all these subspaces are proper.  Consequently,
\begin{equation}\label{eq:large-characteristic-orbit-rigidity}
 p>|H|\quad\Longrightarrow\quad N_{\mathrm{orb}}(H,V)=H,
\end{equation}
because otherwise
$|V|\leq |H|p^{\dim V-1}<p^{\dim V}$.

For $U$ satisfying~\eqref{eq:frobenius-module-orbit-criterion} under
\eqref{eq:orbit-rigidity} and $\tau\in\operatorname{Irr}(H)$, let $\chi_\tau$ be
the character of $\tau$ and define
\begin{equation}\label{eq:complement-twisted-indicator}
 \nu_{\alpha_U^H,U^2}(\tau)
 =\frac1{|H|}\sum_{h\in H}\chi_\tau\bigl(h\alpha_U^H(h)U^2\bigr).
\end{equation}
These are the indicators of the simple representations of $G$ inflated from $H$.

The induced simple objects contribute no further signs.

\begin{proposition}\label{prop:frobenius-induced-positive}
Let $U$ satisfy~\eqref{eq:frobenius-module-orbit-criterion} and suppose
$a=U^2\in H$.  For the coherent pair $(\alpha_U,a)$, every induced simple has
positive indicator:
\[
 \nu_{\alpha_U,a}(W_\lambda)=1
 \qquad\bigl(\lambda\in\widehat V\setminus\{1\}\bigr).
\]
\end{proposition}

\begin{proof}
Realize $W_\lambda$ on a basis $\{e_h\mid h\in H\}$ by
\begin{equation}\label{eq:induced-frobenius-model}
 \rho_\lambda(v)e_h=\lambda(h^{-1}v)e_h,
 \qquad \rho_\lambda(k)e_h=e_{kh}.
\end{equation}
Since $\alpha_U$ is class-inverting, the isomorphism class of $W_\lambda$ is fixed
by the functor~\eqref{eq:alpha-real-functor}. Comparing the one-dimensional
$V$-weight spaces in the model~\eqref{eq:induced-frobenius-model} gives
$c_\lambda\in H$ such that
\begin{equation}\label{eq:c-lambda-condition}
 \lambda(c_\lambda^{-1}Uv)=\lambda(-v)
 \qquad(v\in V).
\end{equation}
It is unique because $H$ acts freely on
$\widehat V\setminus\{1\}$.  Define an antilinear map by
\[
 T_\lambda e_h=e_{\alpha_U^H(h)c_\lambda}.
\]
Equations~\eqref{eq:induced-frobenius-model} and
\eqref{eq:c-lambda-condition} give
\[
 T_\lambda\rho_\lambda(g)T_\lambda^{-1}
 =\rho_\lambda(\alpha_U(g))
 \qquad(g\in G).
\]

Apply~\eqref{eq:c-lambda-condition} twice.  One obtains
\[
 \lambda\bigl((\alpha_U^H(c_\lambda)c_\lambda)^{-1}av\bigr)=\lambda(v)
 \qquad(v\in V).
\]
Freeness of the $H$-action on the nontrivial characters yields
$\alpha_U^H(c_\lambda)c_\lambda=a$. Therefore
\[
 T_\lambda^2e_h
 =e_{(\alpha_U^H)^2(h)\alpha_U^H(c_\lambda)c_\lambda}
 =e_{aha^{-1}a}
 =\rho_\lambda(a)e_h.
\]
Proposition~\ref{prop:automorphism-induced-real-descent} now gives the result.
\end{proof}

For simples inflated from $H$, summation over $V$ reduces
\eqref{eq:generalized-twisted-indicator} to
\eqref{eq:complement-twisted-indicator}. Hence the complement indicators determine
$\Omega_{s_{\alpha_U}}(\Rep_{\C}(G))$.

\emph{Residual twists.}
Orbit rigidity also controls tensor autoequivalences that fix every simple
isomorphism class. If $N_{\mathrm{orb}}(H,V)=H$, every class-preserving automorphism of
$G$ is inner. Indeed, after conjugating by an element of $V$, it preserves $H$; its
restriction to $V$ belongs to $N_{\mathrm{orb}}(H,V)=H$, and the automorphism is
conjugation by that element of $H$. Let $\operatorname{Tw}(V,H)$ denote the kernel of the
action of tensor autoequivalence classes on simple isomorphism classes. After the
class-preserving automorphism has been removed, the proof of
Proposition~\ref{prop:frobenius-autoequivalence-permutations} represents every such
class by a $G$-invariant twist supported on $V$. The classes of these twists form a
subgroup of $H^2(\widehat V,\C^\times)^H$ and map onto $\operatorname{Tw}(V,H)$; the
quotient accounts for gauge identifications and class-preserving automorphisms
\cite[Section~2 and Propositions~4.1--4.2]{DavydovTwistedGroup}. The product is induced by
multiplying twists in the commutative algebra $\C[V]^{\otimes2}$. Since $\widehat V$
is an elementary abelian $p$-group,~\eqref{eq:frobenius-alternation-cohomology}
shows that $\operatorname{Tw}(V,H)$ is a finite abelian $p$-group. If $p$ is odd,
coprime cohomology gives, for every $C_2$-action on this group,
\begin{equation}\label{eq:residual-twist-H1}
 H^1(C_2,\operatorname{Tw}(V,H))=1.
\end{equation}
Indeed, multiplication by $2$ is an automorphism of this finite $p$-group; see
\cite[Corollary~III.10.2]{Brown}.

\begin{theorem}\label{thm:orbit-rigid-frobenius-real}
Let $G=V\rtimes H$ be an irreducible Frobenius group over $\F_p$, where $p$ is odd,
and suppose that $N_{\mathrm{orb}}(H,V)=H$. Then $\Rep_{\C}(G)$ has a split real form if
and only if it has a split symmetric real form. These conditions hold if and only
if there is $U\in\mathcal N$ satisfying
\eqref{eq:frobenius-module-orbit-criterion} such that
\begin{equation}\label{eq:orbit-rigid-indicator-criterion}
 \nu_{\alpha_U^H,U^2}(\tau)=1
 \qquad\bigl(\tau\in\operatorname{Irr}(H)\bigr).
\end{equation}
\end{theorem}

\begin{proof}
Suppose first that~\eqref{eq:frobenius-module-orbit-criterion} and
\eqref{eq:orbit-rigid-indicator-criterion} hold.  Orbit rigidity gives the coherent
pair~\eqref{eq:alpha-U-coherent-pair}.  The inflated simples have positive sign by
\eqref{eq:orbit-rigid-indicator-criterion}, and the induced simples have positive
sign by Proposition~\ref{prop:frobenius-induced-positive}.  Hence
Proposition~\ref{prop:automorphism-induced-real-descent} supplies a split symmetric
real form.

Conversely, let a split real form be given. We first show that its descent element
can be conjugated to a purely automorphism-induced one. Let $\gamma$ denote complex
conjugation. Proposition~\ref{prop:framed-split-descent} supplies an element
\[
 t\in\Pi_{\mathrm{fr}}(\Rep_{\C}(G)/\R),
 \qquad p_{\mathrm{fr}}(t)=\gamma,
 \qquad t^2=1.
\]
Let $t_{\mathrm{sf}}$ be its image after forgetting the frames. By
Propositions~\ref{prop:frobenius-autoequivalence-permutations}
and~\ref{prop:frobenius-class-inverting-criterion}, there is
$U\in\mathcal N$ satisfying~\eqref{eq:frobenius-module-orbit-criterion}.
Orbit rigidity gives
$a=U^2\in H$ and~\eqref{eq:alpha-U-coherent-pair}.

Both $t_{\mathrm{sf}}$ and $[\Phi_{\alpha_U}]$ lie over $\gamma$ and fix every
simple isomorphism class. Hence
\[
 s:=t_{\mathrm{sf}}[\Phi_{\alpha_U}]^{-1}
 \in\operatorname{Tw}(V,H),
 \qquad
 t_{\mathrm{sf}}=s[\Phi_{\alpha_U}].
\]
Since the coherent pair
makes $[\Phi_{\alpha_U}]^2=1$, conjugation by $[\Phi_{\alpha_U}]$ defines an
involution $\theta$ of $\operatorname{Tw}(V,H)$. Since $t^2=1$, its unframed image has
order two, and therefore
\[
 s\theta(s)=1.
\]
By~\eqref{eq:residual-twist-H1}, $s=r^{-1}\theta(r)$ for some
$r\in\operatorname{Tw}(V,H)$. Choose a representative of $r$ together with frames on the
simple objects, thereby obtaining
$\widetilde r\in\ker(p_{\mathrm{fr}})$, and put
\[
 t'=\widetilde r\,t\,\widetilde r^{-1}.
\]
Then $(t')^2=1$, while its underlying unframed class is
\[
 r\bigl(s[\Phi_{\alpha_U}]\bigr)r^{-1}
 =rs\theta(r^{-1})[\Phi_{\alpha_U}]
 =[\Phi_{\alpha_U}].
\]
Transport the frames and the square-one structure of $t'$ along a monoidal
isomorphism from its underlying functor to $\Phi_{\alpha_U}$. Since
$\Phi_{\alpha_U}$ preserves the canonical symmetry, the resulting element belongs
to $\Pi_{\mathrm{fr}}^{\mathrm{sym}}(\Rep_{\C}(G)/\R)$, still lies over
$\gamma$, and still has square one. Together with the identity it defines a
section of $p_{\mathrm{fr}}^{\mathrm{sym}}$, and hence a split symmetric real form
by Subsection~\ref{subsec:structured-forms}.

The transported frames make every simple object stable, and hence all their Brauer
signs are positive. Since $Z(G)=1$ by~\eqref{eq:frobenius-center-trivial}, the
coherence of the pure automorphism-induced datum is unique. On the simples inflated
from $H$, Proposition~\ref{prop:automorphism-induced-real-descent} identifies these
signs with the numbers~\eqref{eq:complement-twisted-indicator}. This proves
\eqref{eq:orbit-rigid-indicator-criterion}.
\end{proof}

By~\eqref{eq:large-characteristic-orbit-rigidity}, the orbit-rigidity hypothesis is
automatic when $p>|H|$.

\subsection{A quaternionic counterexample}

The complement indicators cannot be omitted. Recall that a finite group is
\emph{ambivalent} if every element is conjugate to its inverse.

\begin{proposition}\label{prop:Q8-frobenius-family}
For every prime $p>8$, the quaternion group $Q_8$ has a faithful irreducible
fixed-point-free
action on $\F_p^2$ such that
\[
 G_p=(\F_p^2)^+\rtimes Q_8
\]
is ambivalent, but $\Rep_{\C}(G_p)$ has no split real form.
\end{proposition}

\begin{proof}
The sets of squares in $\F_p$ and of elements of the form $-1-y^2$ both have cardinality
$(p+1)/2$; hence there are $s_0,t_0\in\F_p$ with
$s_0^2+t_0^2=-1$. Put
\begin{equation}\label{eq:Q8-frobenius-matrices}
 A=\begin{pmatrix}0&1\\-1&0\end{pmatrix},
 \qquad
 B=\begin{pmatrix}s_0&t_0\\t_0&-s_0\end{pmatrix}.
\end{equation}
Then $A^2=B^2=-I$ and $AB=-BA$.  Hence
$H=\langle A,B\rangle\cong Q_8$.  Its action on $V=\F_p^2$ is irreducible because
a common invariant line would contradict $AB=-BA$ in odd characteristic.  Every
nonidentity element acts without nonzero fixed vectors, since every element of
order four has square $-I$. Thus $G_p$ is an irreducible Frobenius group.

The group $Q_8$ is ambivalent and contains $-I$.  Taking $U=I$ in
\eqref{eq:frobenius-module-orbit-criterion} therefore shows that the identity
automorphism of $G_p$ is class-inverting; equivalently, every complex irreducible
representation of $G_p$ is self-dual.  Since $p>|Q_8|$, however,
$N_{\mathrm{orb}}(Q_8,V)=Q_8$.  Any $U$ satisfying the module-orbit condition belongs to
$Q_8$, and its associated coherent pair is
$(\operatorname{Ad}_U,U^2)$.  Replacing the generator $t$ in
\eqref{eq:index-two-extension} by $U^{-1}t$ identifies this pair with
$(\id,1)$.  Its indicators are therefore the ordinary Frobenius--Schur indicators
of $Q_8$.  The two-dimensional irreducible representation of $Q_8$ has indicator
$-1$~\cite[Proposition~4.3]{MasonNgFS}, and
Theorem~\ref{thm:orbit-rigid-frobenius-real} rules out a split real form. Since
$Z(G_p)=1$, the negative quaternionic sign cannot be corrected; equivalently,
$\Omega_{s_{\id}}(\Rep_{\C}(G_p))\ne0$.
\end{proof}

The representation category of $\F_3^2\rtimes Q_8$ is studied independently
in~\cite[Section~7]{BHPReality}, where explicit $F$-symbol calculations show
that it has inherently complex $F$-symbols. The proposition above gives a
complementary infinite family in characteristics $p>8$ by descent and indicator
methods.

\section{Cyclic Frobenius groups and minimal split fields}
\label{sec:cyclic-frobenius}

We now turn to cyclic complements. The field problem was raised in
\cite[Section~3.8]{EGDescent}; Theorem~\ref{thm:frobenius} determines the smallest
split field for this family and shows that the form may always be chosen symmetric.

Let $q=p^n$, let $m>1$ divide $q-1$, and assume
\begin{equation}\label{eq:frobenius-hyp}
 \operatorname{ord}_m(p)=n.
\end{equation}
Choose $r\in\F_q^\times$ of order $m$ and put
\[
 V=\F_q^+,
 \qquad H=\langle r\rangle\cong C_m,
 \qquad
 G_{q,m}=V\rtimes H,
 \qquad d=\frac{q-1}{m}.
\]
The complement acts on $V$ by multiplication. Write $\langle p\rangle$ for the
subgroup generated by $p$ in $(\mathbb Z/m\mathbb Z)^\times$.
Every irreducible Frobenius group with elementary abelian kernel and cyclic
complement has this form, since irreducibility gives
$\F_p[r]\cong\F_{p^n}$ and
$[\F_p(r):\F_p]=\operatorname{ord}_m(p)$; the converse follows as well.

The normalizer of $H$ in the linear group is
\begin{equation}\label{eq:cyclic-normalizer}
 N_{\operatorname{GL}_{\F_p}(\F_q)}(H)
 =\{x\mapsto\lambda x^{p^j}\mid
 \lambda\in\F_q^\times,\ 0\leq j<n\}.
\end{equation}
A normalizing map sends multiplication by $r$ to multiplication by $r^h$
for some $h$. The two operators are similar only when $h=p^j$ modulo $m$; after
composing with the inverse Frobenius, the map centralizes multiplication by $r$ and
is therefore multiplication by an element of $\F_q^\times$.

The subgroup controlling the nonlinear labels is the kernel of the action of $\langle p\rangle$
on $\F_q^\times/H$, namely
\begin{equation}\label{eq:Sqm}
 S_{q,m}
 =\{p^j\bmod m\mid 0\leq j<n,\ p^j\equiv1\pmod d\}.
\end{equation}
The group $\F_q^\times/H$ is cyclic of order $d$. Put
\[
 L=\Q(\zeta_p,\zeta_m),
 \qquad K_{q,m}=\Q(\zeta_m)^{S_{q,m}}.
\]
The superscript denotes the fixed field for the cyclotomic action of
$(\mathbb Z/m\mathbb Z)^\times$.

\begin{theorem}\label{thm:frobenius}
For every subfield $K\subseteq\C$, the category $\Rep_{\C}(G_{q,m})$ has a split
$K$-form if and only if it has a split symmetric $K$-form. These conditions hold
if and only if
\[
 K_{q,m}\subseteq K.
\]
\end{theorem}

\begin{proof}
\emph{Necessity.}
Write $\operatorname{Tr}=\operatorname{Tr}_{\F_q/\F_p}$.
Specializing~\eqref{eq:general-frobenius-irreducibles}, besides the $m$ linear
characters inflated from $H$, the remaining irreducible representations are
\begin{equation}\label{eq:nonlinear-irreps}
 W_{tH}=\operatorname{Ind}_V^{G_{q,m}}(\psi_t),
 \qquad
 \psi_t(x)=\zeta_p^{\operatorname{Tr}(tx)},
 \qquad tH\in\F_q^\times/H.
\end{equation}
The stabilizer of every nontrivial $\psi_t$ in $H$ is trivial; hence these
representations are irreducible of degree $m$. The displayed models show that $L$
is a splitting field. Since $p\nmid m$,
\[
 \Gal(L/\Q)
 \cong\F_p^\times\times(\mathbb Z/m\mathbb Z)^\times,
\]
where $(u,s)$ sends $\zeta_p$ to $\zeta_p^u$ and $\zeta_m$ to $\zeta_m^s$.

Let $K\subseteq\C$ be a field over which a split form exists, and put $F=K\cap L$.
The extension $L/F$ is Galois because $L/\Q$ is abelian, and $L$ and $K$ are
linearly disjoint over $F$. Thus every $\sigma\in\Gal(L/F)$ extends to $LK$ by the
identity on $K$. Extend $\sigma$ by the identity on a transcendence basis of
$\C/LK$, and then extend it to the algebraic closure. This gives
$\widetilde\sigma\in\Aut(\C/K)$. Since the $K$-form is split, its simple objects
remain simple after scalar extension and exhaust the complex simples. The canonical
semilinear action obtained by scalar extension therefore fixes every simple
isomorphism class. After transport to $\Rep_{\C}(G_{q,m})$, comparison with
coefficientwise $\widetilde\sigma$-twisting gives a $\C$-linear tensor
autoequivalence which
compensates the Galois permutation of the simple objects. By
Proposition~\ref{prop:frobenius-autoequivalence-permutations}, the same permutation
is induced by an automorphism $\alpha$ of $G_{q,m}$.

The subgroup $V$ is characteristic. After composing $\alpha$ with an inner
automorphism induced by $V$, we may assume that it preserves $H$. If
$\sigma=(u,s)$, its action on the linear characters forces
$\alpha(r)=r^{s^{-1}}$. The normalizer description~\eqref{eq:cyclic-normalizer}
therefore gives $s=p^j$ for some $0\leq j<n$ and
\[
 U(x)=\lambda x^{p^{-j}}
 \qquad(\lambda\in\F_q^\times).
\]
Here $U=\alpha|_V$ and $x^{p^{-j}}=x^{p^{n-j}}$ denotes the inverse $j$-fold
Frobenius on $\F_q$.
The compensated action on the nonlinear labels is
\[
 tH\longmapsto u\lambda^{p^j}t^{p^j}H.
\]
Indeed,
\[
 (({}^\sigma\!\psi_t)\circ U)(x)
 =\zeta_p^{\operatorname{Tr}((ut\lambda)^{p^j}x)}
\]
because
$\operatorname{Tr}(z x^{p^{-j}})=\operatorname{Tr}(z^{p^j}x)$.
Fixing the label $H$ forces $u\lambda^{p^j}\in H$. The displayed action then
fixes every label exactly when the
$p^j$-power map is trivial on $\F_q^\times/H$, or equivalently when
$p^j\equiv1\pmod d$. Thus
\[
 \Gal(L/F)\subseteq\F_p^\times\times S_{q,m}.
\]
The fixed field of the group on the right is $K_{q,m}$, and Galois correspondence
gives $K_{q,m}\subseteq F\subseteq K$.

\emph{Existence.}
Under the preceding identification, every
$\sigma\in\Gal(L/K_{q,m})=\F_p^\times\times S_{q,m}$ has the form
$\sigma=(u,p^j)$. Define
\[
 a_{u,j}(x)=u^{-1}x^{p^{-j}},
 \qquad a_{u,j}(r)=r^{p^{-j}}.
\]
These automorphisms satisfy
\[
 a_{u,j}\circ a_{v,k}=a_{uv,j+k},
\]
with indices modulo $n$. On $\Rep_L(G_{q,m})$, applying $\sigma$ to representation
matrices and then precomposing with $a_{u,j}$ defines twisted tensor functors
\[
 \Phi_\sigma(\rho)={}^{\sigma}\!\rho\circ a_{u,j}.
\]
They preserve the canonical symmetry and satisfy the descent identities exactly.
They also fix every linear character; on these simples take the identity frames.
Realize $W_{tH}$ on the basis $\{e_z\mid z\in tH\}$ by
\[
 x e_z=\zeta_p^{\operatorname{Tr}(zx)}e_z,
 \qquad r e_z=e_{r^{-1}z}.
\]
Since $p^j\equiv1\pmod d$, the set $tH$ is invariant under
$z\mapsto z^{p^j}$, and
\[
 P_j(e_z)=e_{z^{p^j}}
\]
intertwines $\Phi_\sigma(W_{tH})$ with $W_{tH}$. If
$\tau=(v,p^k)$, these permutation matrices have rational entries, and their frame
compatibility condition is
\[
 P_j\circ\Phi_\sigma\bigl({}^{\sigma}P_k\bigr)
 =P_jP_k=P_{j+k}.
\]
Together with the identity frames on the linear simples, the maps $P_j$ frame the
twisted equivalences $(\Phi_\sigma,\sigma)$. The strict descent identities and the
last display show that their framed classes define a section of
$p_{\mathrm{fr}}^{\mathrm{sym}}$. The structured framed criterion
of Subsection~\ref{subsec:structured-forms} therefore gives a split symmetric
$K_{q,m}$-form of $\Rep_L(G_{q,m})$, and hence, after extending $L$ to $\C$, of
$\Rep_{\C}(G_{q,m})$. Extension of scalars from $K_{q,m}$ gives a split symmetric
form over every $K\supseteq K_{q,m}$.
\end{proof}

Over $\R$, the theorem has the following equivalent formulations.

\begin{corollary}\label{cor:frobenius-real-equivalences}
The following conditions are equivalent.
\begin{enumerate}[label=\textup{(\alph*)}]
\item $\Rep_{\C}(G_{q,m})$ has a split real form.
\item There is an integer $j$, with $0\leq j<n$, such that
\begin{equation}\label{eq:two-congruence-real-criterion}
 p^j\equiv-1\pmod m,
 \qquad p^j\equiv1\pmod d.
\end{equation}
\item The group $G_{q,m}$ has a class-inverting automorphism.
\end{enumerate}
When these conditions hold, the form may be chosen symmetric. If they fail, no
real subfield of $\C$ admits a split form.
\end{corollary}

\begin{proof}
By Theorem~\ref{thm:frobenius}, the first condition holds exactly when
$K_{q,m}\subseteq\R$, and the resulting form may be chosen symmetric. Complex
conjugation acts on $\Q(\zeta_m)$ by $-1$; hence this is equivalent to
$-1\in S_{q,m}$. The definition of $S_{q,m}$ gives the two congruences in
\eqref{eq:two-congruence-real-criterion}.

It remains to compare these conditions with class inversion. By
\eqref{eq:cyclic-normalizer} and
Proposition~\ref{prop:frobenius-class-inverting-criterion}, a class-inverting
automorphism must be represented on $V$ by
$U(x)=\lambda x^{p^j}$. Its action on $H$ is inversion precisely when
$p^j\equiv-1\pmod m$. The orbit condition is
\[
 \lambda x^{p^j}\in H\cdot(-x)\qquad(x\ne0).
\]
Taking $x=1$ gives $\lambda\in-H$, after which the condition becomes
$x^{p^j-1}\in H$ for every $x\ne0$. Since $\F_q^\times/H$ has order $d$, this is
equivalent to $p^j\equiv1\pmod d$. Conversely, the two congruences make
$U(x)=-x^{p^j}$ satisfy both conditions of
Proposition~\ref{prop:frobenius-class-inverting-criterion}. This proves the
equivalences, and the last assertion follows again from
Theorem~\ref{thm:frobenius}.
\end{proof}

Both the degree of the minimal field and whether it is totally real can be read
directly from~\eqref{eq:Sqm}.
With the convention $\operatorname{ord}_1(p)=1$, one has
\[
 S_{q,m}=\langle p^{\operatorname{ord}_d(p)}\rangle,
 \qquad
 [K_{q,m}:\Q]=\frac{\operatorname{ord}_d(p)\varphi(m)}{n}.
\]
Here $\varphi$ is Euler's totient function.
The field $K_{q,m}$ is totally real precisely when $-1\in S_{q,m}$; otherwise it
is a CM field, that is, a totally imaginary quadratic extension of a totally real
field, with maximal real subfield
\[
 \Q(\zeta_m)^{\langle S_{q,m},-1\rangle}.
\]
These assertions follow from cyclotomic Galois theory
\cite[Chapter~2]{Washington}.

The congruence modulo $d$ is essential because it measures the action on the
nonlinear simple objects. For $(q,m)=(25,8)$ one has $S_{25,8}=\{1\}$ and
$K_{25,8}=\Q(\zeta_8)$; omitting the congruence modulo $d$ would incorrectly give
the smaller field $\Q(i)$.

We record two specializations. First,
specializing~\eqref{eq:general-frobenius-rank} gives
\[
 \operatorname{rank}\bigl(\Rep_{\C}(G_{q,m})\bigr)=m+d,
\]
and the rank-five members of the irreducible cyclic family are
\[
 G_{7,2}\cong D_{14}\ \text{(the dihedral group of order $14$)},
 \qquad G_{7,3},
 \qquad G_{5,4}.
\]
The smallest fields over which these categories have split forms, equivalently
split symmetric forms, are respectively
$\Q$, $\Q(\sqrt{-3})$, and $\Q(i)$. The last two are precisely the examples of
\cite{BHP}; the gauge-invariant contractions computed in
\cite[Section~4]{BHP} independently detect the absence of a split real form.

Second, the case $m=q-1$ gives the one-dimensional affine general linear group
$G_{q,q-1}=\operatorname{AGL}_1(\F_q)$, for which the fixed field takes a
particularly simple form. For every prime power $q=p^n\geq3$, the smallest field
over which
\[
 \Rep_{\C}\bigl(\operatorname{AGL}_1(\F_q)\bigr)
\]
has a split form, equivalently a split symmetric form, is
\[
 \Q(\zeta_{q-1})^{\langle p\rangle},
 \qquad
 \left[\Q(\zeta_{q-1})^{\langle p\rangle}:\Q\right]
 =\frac{\varphi(q-1)}{n}.
\]
This is the decomposition field of $p$ in $\Q(\zeta_{q-1})$
\cite[Chapter~2]{Washington}. For $q=3,4$ this field is $\Q$, while for
$q\geq5$ there is no split real form.
Indeed, here $d=1$ and $S_{q,q-1}=\langle p\rangle$, while
$\operatorname{ord}_{q-1}(p)=n$ because for $0<j<n$ one has
$0<p^j-1<q-1$. If $q\geq5$, then $1\not\equiv-1\pmod{q-1}$, and for
$0<j<n$ one also has $0<p^j+1<q-1$. Hence
$-1\notin\langle p\rangle$.

\bibliographystyle{amsplain}
\bibliography{split_forms_references}
\enlargethispage{2\baselineskip}

\end{document}